\documentclass[11pt,pdftex]{article}

\usepackage[truedimen,margin=25truemm, bottom=20truemm]{geometry}

\usepackage{latexsym,amssymb,color}
\usepackage{amsthm}
\usepackage{amsmath}
\usepackage{color}
\usepackage{enumerate}
\usepackage{mathrsfs}
\usepackage{graphicx}
\usepackage{float}
\usepackage{wrapfig}
\usepackage{layout}
\usepackage{cite}
\usepackage{braket}
\usepackage{bm}

\usepackage
[draft=false,setpagesize=false,pdfstartview=FitH,
colorlinks=true,linkcolor=blue,citecolor=magenta,pagebackref=true,bookmarks=false]
{hyperref}

\theoremstyle{definition}
\newtheorem{Definition}{Definition}[section]
\theoremstyle{plain}
\newtheorem{Theorem}{Theorem}[section]
\newtheorem{Corollary}{Corollary}[section]
\newtheorem{Lemma}{Lemma}[section]
\newtheorem{Proposition}{Proposition}[section]
\theoremstyle{remark}
\newtheorem{Remark}{\bf {\rm{{\bf Remark}}}} [section]

\newenvironment{theorem}{\begin{Theorem}$\!\!\!$}{\end{Theorem}}
\newenvironment{lemma}{\begin{Lemma}$\!\!\!$}{\end{Lemma}}
\newenvironment{proposition}{\begin{Proposition}$\!\!\!$}{\end{Proposition}}
\newenvironment{corollary}{\begin{Corollary}$\!\!\!$}{\end{Corollary}}
\newenvironment{remark}{\begin{Remark}$\!\!\!$}{\end{Remark}}

\newenvironment{definition}{\begin{Definition}$\!\!\!$}{\end{Definition}}

\numberwithin{equation}{section}

\def\wt#1{\widetilde{#1}}
\definecolor{orange}{cmyk}{0,0.8,1,0.3}

\newcommand{\cH}{\mathcal{H}}
\newcommand{\cO}{\mathcal{O}}
\newcommand{\cS}{\mathcal{S}}
\newcommand{\ov}[1]{\overline{#1}}
\newcommand{\scrF}{\mathscr{F}}
\newcommand{\scrS}{\mathscr{S}}

\DeclareMathOperator{\dist}{dist}
\DeclareMathOperator{\diver}{div}
\DeclareMathOperator{\supp}{supp}

\newcommand{\R}{\mathbb{R}}

\newcommand{\RN}{\mathbb{R}^N}
\newcommand{\HS}{\mathbb{R}^{N+1}_+}
\newcommand{\N}{\mathbb{N}}
\newcommand{\e}{\varepsilon}
\newcommand{\la}{\left\langle}
\newcommand{\ra}{\right\rangle}

\newcommand{\Hsloc}{H^s_{\mathrm{loc}}}
\newcommand{\Hloc}{H^1_{\mathrm{loc}}}
\newcommand{\dHs}{\dot{H}^s}
\newcommand{\rd}{\mathrm{d}}

\begin{document}
\title{On weak solutions to a fractional Hardy-H\'enon equation:\\ Part II: Existence}
\author{
	Shoichi Hasegawa\\
	Department of Mathematics,\\ 
	School of Fundamental Science and Engineering,\\
	Waseda University,\\
	3-4-1 Okubo, Shinjuku-ku, Tokyo 169-8555, Japan
	\\
	\\
	Norihisa Ikoma\\
	Department of Mathematics,
	Faculty of Science and Technology,\\
	Keio University,\\
	3-14-1 Hiyoshi, Kohoku-ku, Yokohama, 223-8522, Japan
	\\
	\\
	Tatsuki Kawakami\\
	Applied Mathematics and Informatics Course,\\
	Faculty of Advanced Science and Technology,\\ 
	Ryukoku University,\\
	1-5 Yokotani, Seta Oe-cho, Otsu, Shiga 520-2194, Japan
}
\date{}
\maketitle

\begin{abstract}
This paper and \cite{HIK-20} treat the existence and nonexistence of stable
weak solutions 
to a fractional Hardy--H\'enon equation 
$(-\Delta)^s u = |x|^\ell |u|^{p-1} u$ in $\mathbb{R}^N$, 
where $0 < s < 1$, $\ell > -2s$, $p>1$, $N \geq 1$ 
and $N > 2s$. 
In this paper, when $p$ is critical or supercritical in the sense of the Joseph--Lundgren, 
we prove the existence of a family of positive radial stable solutions, 
which satisfies the separation property.
We also show the multiple existence of the Joseph--Lundgren critical exponent 
for some $\ell \in (0,\infty)$ and $s \in (0,1)$, 
and this property does not hold in the case $s=1$.

\end{abstract}


\noindent \textbf{AMS Mathematics Subject Classification System 2020}: 
35R11, 35B09, 35B33, 35B35, 35D30.

\section{Introduction}
\label{section:1}

	This paper is a continuation of \cite{HIK-20}, and 
we consider the existence of stable solutions for a fractional Hardy-H\'enon equation
\begin{equation}
	\label{eq:1.1}
	(-\Delta)^s u=|x|^\ell |u|^{p-1}u\qquad\mbox{in}\quad{\mathbb R}^N.
\end{equation}
Throughout this paper, 
we always assume the following condition on $s,\ell,p,N$:
\begin{equation}\label{eq:1.2}
	0 < s < 1, \quad  \ell>-2s, \quad p>1, \quad  N\ge 1, \quad N>2s. 
\end{equation}
Here $(-\Delta)^s$ is the fractional Laplacian, which is defined for any $\varphi\in C^\infty_c(\mathbb R^N)$ by
\[
	(-\Delta)^s\varphi(x)
	:= C_{N,s}\mbox{P.V.}\int_{\mathbb R^N}\frac{\varphi(x)-\varphi(y)}{|x-y|^{N+2s}}\,dy
	=C_{N,s}\lim_{\varepsilon\to0}\int_{|x-y|>\varepsilon}\frac{\varphi(x)-\varphi(y)}{|x-y|^{N+2s}}\,dy
\]
for $x\in \mathbb R^N$,
where P.V. stands for the Cauchy principal value integral and
\[
	C_{N,s}:=2^{2s}s(1-s)\pi^{-\frac{N}{2}}\frac{\Gamma(\frac{N}{2}+s)}{\Gamma(2-s)}
\]
with the gamma function $\Gamma$. For the notion of solutions and stability, 
see Definitions \ref{Definition:1.1} and \ref{Definition:1.3} below.

\subsection{The case $s=1$}
\label{section:1.1}

	When $s=1$, we regard $(-\Delta)^s$ in \eqref{eq:1.1} as the usual Laplacian $-\Delta$ and \eqref{eq:1.1} becomes 
	\begin{equation}\label{eq:1.3}
		-\Delta u = |x|^\ell |u|^{p-1} u \quad \text{in} \ \RN.
	\end{equation}
Equation \eqref{eq:1.3} is called the Lane--Emden equation ($\ell = 0$), the H\'enon equation ($\ell \geq 0$) or 
the Hardy--H\'enon equation ($\ell > -2$) and there are a lot of works for \eqref{eq:1.3} in which 
the existence/nonexistence of solutions and their properties were studied. 
For instance, we refer to \cite{Farina,QS19} ($\ell = 0$), \cite{DDG,LLD00,Wa93,WY-12} ($\ell >-2$) 
and references therein. 
There are also works for \eqref{eq:1.3} on manifolds \cite{BS12,BFG14,BGGV13,Ha15,Ha17,MS08}
and for \eqref{eq:1.3} with the higher order operators \cite{DDWW14,GG06,GW10,HHY14,Ka09,WY13}.

	For \eqref{eq:1.3} with $\ell = 0$, among other things, in the seminal work \cite{Farina}, Farina 
proved the nonexistence of nontrivial stable solutions for $ p \in (1, p_c(N) )$ and the existence of positive radial stable solutions 
for $p \in [ p_c(N) , \infty )$, where $p_c(N)$ is the Joseph--Lundgren exponent and defined by 
	\[
		p_c(N) := \infty \quad \text{if} \ N \leq 10, \quad 
		p_c(N) := \frac{ (N-2)^2 -4 N + 8 \sqrt{N-1}  }{(N-2) (N-10)} \quad \text{if} \ N \geq 11.
	\] 
This result is extended to the case $\ell > -2$ by \cite{DDG,WY-12} and in this case, 
the threshold $p_c(N)$ is replaced by 
	\begin{equation}\label{eq:1.4}
		p_+(N,\ell) := \left\{\begin{aligned}
			& \frac{ (N-2)^2 - 2 (\ell + 2) (\ell + N) + 2\sqrt{ (\ell +2)^3 (\ell + 2 N - 2) } }
			{(N-2) (N-4 \ell - 10) } 
			& &
			\text{if} \ N > 10 + 4 \ell ,
			\\
			& \infty 
			& &
			\text{if} \ 2 \leq N \leq 10 + 4 \ell.
		\end{aligned}\right.
	\end{equation}

\subsection{The case $0<s<1$}
\label{section:1.2}

	Next, we turn to \eqref{eq:1.1}, which is a fractional counterpart of \eqref{eq:1.3}. 
For the nonexistence of stable solutions to \eqref{eq:1.1}, 
in \cite{HIK-20}, the authors of this paper addressed this issue 
under the condition that $p$ is subcritical in the sense of Joseph--Lundgren (see Definition~\ref{Definition:1.2} below). 
For references on the nonexistence of stable solutions, we refer to \cite{HIK-20} and references therein. 
In this paper, we focus on the existence of stable solutions.

	To state known results of \eqref{eq:1.1}, we first remark that due to \cite{CS-07}, 
the fractional Laplacian $(-\Delta)^s u$ can be expressed as the limit 
$- \lim_{t \searrow 0} t^{1-2s}  \partial_t U (x,t)$, 
where $U(x,t)$ is a solution of some elliptic equation on $\HS$ with 
$U(x,0) = u(x)$. 
Exploiting this idea, we may rewrite \eqref{eq:1.1} as the equation in $\HS$ with 
the nonlinear boundary condition:
	\begin{equation}\label{eq:1.5}
		\left\{\begin{aligned}
		- \diver \left( t^{1-2s} \nabla U \right) &= 0 & & \text{in} \ \HS, 
		\\
		- \lim_{t \searrow 0} t^{1-2s} \partial_t U(x,t) &= \kappa_s |x|^\ell 
		\left| U(x,0) \right|^{p-1} U(x,0) & &\text{for $x \in \RN$},
		\end{aligned}\right.
	\end{equation}
where $\kappa_s := 2^{1-2s} \Gamma (1-s) / \Gamma (s)$. 
In particular, when $s=1/2$, \eqref{eq:1.5} becomes 
	\begin{equation}\label{eq:1.6}
		-\Delta U = 0 \quad \text{in} \ \HS, \quad - \lim_{t \searrow 0} \partial_t U(x,t) 
		= \kappa_s |x|^\ell \left| U(x,0) \right|^{p-1} U(x,0) \quad \text{for $x \in \RN$}.
	\end{equation}

	We first state results for $s = 1/2$ and $\ell = 0$, namely, \eqref{eq:1.6}. 
In \cite{CCFS}, Chipot, Chleb\'{\i}k, Fila and Shafrir proved
the existence of positive solutions for $p \geq (N+1)/(N-1)=p_S(N,0)$, 
where $p_S(N,\ell)$ is defined in \eqref{eq:1.7}. 
On the other hand, 
Quittner and Reichel \cite{QR} showed
the existence of the singular solution. 
Moreover, in \cite{H}, under the assumption that $p$ is either critical or 
supercritical in the sense of the Joseph--Lundgren, 
Harada showed the existence of a family $(u_\alpha)_{\alpha>0}$ of solutions of \eqref{eq:1.1} 
with the separation property $u_{\alpha_1} < u_{\alpha_2}$ for $\alpha_1 < \alpha_2$ and $u_\alpha \to u_\infty$ as $\alpha \to \infty$, 
where $u_\infty$ is a singular solution of \eqref{eq:1.1}. 
In addition, though it is not explicitly stated in \cite{H}, the family $(u_\alpha)_{\alpha>0}$ becomes stable. 
We also refer to \cite{DDM} for the notion of the Joseph--Lundgren exponent when $s=1/2$ and $\ell = 0$.

	In the case $s \in (0,1)$ with $\ell = 0$, 
the existence of a family $(u_\alpha)_{\alpha > 0}$ of positive radial stable solutions of \eqref{eq:1.1} was proved in \cite{DDW}. 
Moreover, we remark that the existence of singular solution of \eqref{eq:1.1} was pointed out in \cite{FW-16} 
(see also \cite{Du-16,Fall}) for the case $s \in (0,1)$ and $\ell > 0$. 
Here we note that 
in \cite{DDW}, the separation property $u_{\alpha_1} < u_{\alpha_2}$ for $\alpha_1 < \alpha_2$ was not proved. 
Furthermore, in \cite{FW-16}, the nonexistence of stable solutions was shown, however, 
the existence of a family of positive radial stable solutions was not dealt.

	From the above results, the following are not known:
	\begin{itemize}
	\item 
	the existence of a family $(u_\alpha)_{\alpha > 0}$ 
	of stable solutions to \eqref{eq:1.1} when $\ell \neq 0$;
	\item 
	the separation property $u_{\alpha_1} < u_{\alpha_2}$ for $\alpha_1 < \alpha_2$ 
	for the case $0<s<1$ and $\ell > -2s$.
	\end{itemize}
Our aim in this paper is to address these questions. 
In addition, we also investigate the range of $p$ 
where $p$ is subcritical, critical or supercritical in the sense of Joseph--Lundgren.

\subsection{Results}
\label{section:1.3}

	To state our results, we first give the definition of solutions of \eqref{eq:1.1}. 
To this end, we set
	\[
		\la u , v \ra_{\dHs(\RN)} := \frac{C_{N,s}}{2} \int_{ \RN \times \RN } 
		\frac{ \left( u(x) - u(y) \right) \left( v (x) - v (y) \right) }
		{|x-y|^{N+2s}} \, dx dy.
	\]
It is known (see \cite{DPV12}) that for each $u,\varphi \in C^\infty_c(\RN)$, 
	\[
		\int_{  \RN} \left( \left( -\Delta \right)^s u \right) \varphi \, dx 
		= \la u , \varphi \ra_{\dHs (\RN)}.
	\]
In \cite{HIK-20}, we treat solutions in 
$H^s_{\rm loc} (\RN) \cap L^\infty_{\rm loc} (\RN) \cap L^1( \RN , (1+|x|)^{-N-2s} dx )$, where 
	\[
		L^1 ( \RN , (1+|x|)^{-N-2s} dx ) := 
		\Set{ u : \RN \to \R | \int_{  \RN} \left( 1 + |x| \right)^{-N-2s} \left| u(x) \right| dx < \infty  }.
	\]
However, since we deal with singular solutions of \eqref{eq:1.1} in this paper, 
we slightly change the notion of solutions as follows:
	\begin{definition}\label{Definition:1.1}
		We say that $u$ is a \emph{solution of \eqref{eq:1.1}} 
		if $u \in H^s_{\rm loc} (\RN) \cap L^1( \RN , (1+|x|)^{-N-2s} dx )$ and $u$ satisfies 
		\[
			|x|^\ell |u|^{p-1} u \in L^{ \frac{2N}{N+2s} }_{\rm loc} (\RN), \quad 
			\la u , \varphi \ra_{\dHs (\RN)} = \int_{  \RN} |x|^\ell |u|^{p-1} u \varphi \, dx \quad 
			\text{for each $\varphi \in C^\infty_c(\RN)$}.
		\]
	\end{definition}
	
	\begin{remark}\label{Remark:1.1}
		\begin{enumerate}
		\item 
		For later use, we require the stronger condition $|x|^\ell |u|^{p-1} u \in L^{ \frac{2N}{N+2s} }_{\rm loc} (\RN)$ 
		than $|x|^\ell |u|^{p-1} u \in L^1_{\rm loc} (\RN)$ in Definition \ref{Definition:1.1}. 
		\item 
		Since $\ell>-2s$ and $N>2s$ hold by \eqref{eq:1.2}, we see that 
		if $u \in L^\infty_{\rm loc} (\RN)$, then $|x|^\ell |u|^{p-1} u \in L^{ \frac{2N}{N+2s} }_{\rm loc} (\RN)$. 
		We remark that in Theorem \ref{Theorem:1.1}, we find positive solutions of \eqref{eq:1.1} belonging to 
		$H^s_{\rm loc} (\RN) \cap L^\infty (\RN) \cap C(\RN)$.
		\end{enumerate}
	\end{remark}

	Next, following \cite{DDM,DDW,H}, we define \emph{subcritical, critical and supercritical in the sense of Joseph--Lundgren}. 
For this purpose, we set 
	\begin{equation}\label{eq:1.7}
		\begin{aligned}
		p_S(N,\ell) := \frac{N+2s+2\ell}{N-2s}, 
		\quad 
		\lambda(\alpha)
		:=2^{2s}\frac{\Gamma(\frac{N+2s+2\alpha}{4})\,\Gamma(\frac{N+2s-2\alpha}{4})}
		{\Gamma(\frac{N-2s-2\alpha}{4})\,\Gamma(\frac{N-2s+2\alpha}{4})}.
		\end{aligned}
	\end{equation}

\begin{definition}\label{Definition:1.2}
	We say that $p$ is 
	\emph{subcritical in the sense of the Joseph--Lundgren} (for short, we write \emph{JL-subcritical}) if 
	\[
		\text{either} \quad 1 < p \leq p_S(N,\ell) \quad \text{or} \quad 
		p>p_S(N,\ell) , \quad p \lambda \left( \frac{N-2s}{2} - \frac{2s+\ell}{p-1} \right) > \lambda (0).
	\]
	On the other hand, $p$ is said to be \emph{supercritical in the sense of the Joseph--Lundgren} 
	(for short, we write \emph{JL-supercritical}) provided 
	\[
		p>p_S(N,\ell),  \quad p \lambda \left( \frac{N-2s}{2} - \frac{2s+\ell}{p-1} \right) < \lambda (0).
	\]
	We also call $p$ \emph{critical in the sense of the Joseph--Lundgren} 
	(for short, \emph{JL-critical}) when  
	\[
		p>p_S(N,\ell) , \quad p \lambda \left( \frac{N-2s}{2} - \frac{2s+\ell}{p-1} \right) = \lambda (0).
	\]
\end{definition}
We remark that Definition \ref{Definition:1.2} coincides with those in \cite{DDM,DDW,H} when $\ell = 0$.

	Finally, we introduce the notation of stability of solutions of \eqref{eq:1.1}:

\begin{definition}\label{Definition:1.3}
	Let $u$ be a solution of \eqref{eq:1.1}. 
	We say that $u$ is \emph{stable} if $u$ satisfies 
	\[
		p \int_{\RN} |x|^\ell |u|^{p-1} \varphi^2 \, dx 
		\leq \| \varphi \|_{\dHs(\RN)}^2 \quad \text{for every $\varphi \in C^\infty_c(\RN )$},
	\]
	where 
		\[
			\begin{aligned}
				&\| \varphi \|_{\dHs (\RN)}^2 := \int_{  \RN} \left( 2\pi |\xi|^2 \right)^s 
				\left| \scrF \varphi (\xi) \right|^2 d \xi 
				= \frac{C_{N,s}}{2} \left[ \varphi \right]_{\dHs(\RN)}^2, 
				\\
				&
				\scrF \varphi (\xi) := \int_{  \RN} e^{-2\pi i x \cdot \xi} \varphi(x) \, dx, 
				\quad 
				\left[ \varphi \right]_{\dHs(\RN)}^2 
				:= \int_{  \RN \times \RN} \frac{  \left( \varphi(x) - \varphi(y) \right)^2  }{|x-y|^{N+2s}} \, dx dy.
			\end{aligned}
		\]
\end{definition}

	Now we are in position to state our main results in this paper:

\begin{theorem}\label{Theorem:1.1}
	Suppose \eqref{eq:1.2} and put
	\begin{equation}\label{eq:1.8}
	\theta_0:=\frac{2s+\ell}{p-1}.
	\end{equation}
	Assume $p$ is JL-critical or JL-supercritical, namely
	\begin{equation}\label{eq:1.9}
		p>p_S(N,\ell), \quad p \lambda \left( \frac{N-2s}{2} - \theta_0 \right)
		\leq \lambda (0).
	\end{equation}
	Then the following hold: 
	\begin{enumerate}
		\item[{\rm (i)}]
		      There exists a singular stable solution $u_S(x) := A_0 |x|^{ - \theta_0} \in H^s_{\rm loc} (\RN) $ of \eqref{eq:1.1} 
		      with $u_S \not \in \dHs (\RN)$, where
		      \begin{equation}\label{eq:1.10}
			      A_0 := \lambda \left( \frac{N-2s}{2} - \theta_0 \right)^{\frac{1}{p-1}}.
		      \end{equation}
		\item[{\rm (ii)}]
		      There exists a family $( u_\alpha )_{\alpha>0} \subset \Hsloc(\RN) \cap L^\infty(\RN) \cap C(\RN) $ of 
		      solutions of \eqref{eq:1.1} such that 
		      \begin{enumerate}
			      \item[\rm(a)]
			            $u_\alpha$ is positive, radial, stable and 
			            $u_\alpha (x) < u_S(x)$ for each $x \in \RN \setminus \set{0}$,
			      \item[\rm(b)]
			            for each $\alpha > 0$ and $x \in \RN$, 
			            $ u_\alpha (0) = \alpha$ and 
			            $ u_\alpha (x) = \alpha u_1 ( \alpha^{1/\theta_0} x )$,
			      \item[\rm(c)]
			            for every $\alpha < \beta$ and $x \in \RN$, 
			            $u_\alpha(x) < u_\beta (x)$,
			      \item[\rm(d)]
			            if $u \in \Hsloc (\RN) \cap L^\infty (\RN) $
			            is a positive radial stable solution of \eqref{eq:1.1} satisfying $u(x) \to 0$ as $|x| \to \infty$, then 
			            $u = u_\alpha$ with $u(0) = \alpha$,
			      \item[\rm(e)]
			      		As $\alpha \to \infty$, 
			      		$u_\alpha (x) \to u_S(x)$ for each $x \in \RN \setminus \set{0}$ and 
			            $u_\alpha (x) \to u_S(x)$ in $\Hsloc(\RN)$. 
		      \end{enumerate}
	\end{enumerate}
\end{theorem}

	Next we look at the range of $p$ in which $p$ is JL-subcritical, JL-critical and JL-supercritical 
for the case $ \ell \leq 0$:

\begin{theorem}\label{Theorem:1.2}
	Suppose \eqref{eq:1.2}. Then the following hold:
	\begin{enumerate}
		\item[\emph{(i)}]
		      Assume $-2s < \ell \leq 0$. Then there exists a unique $p_{JL} \in (p_S(N,\ell) , \infty ]$ such that 
		      each $p \in (1, p_{JL} )$ is JL-subcritical. In addition, if $p_{JL} < \infty$, 
		      then $p=p_{JL}$ (resp. $p>p_{JL}$) is JL-critical (resp. JL-supercritical);
		\item[\emph{(ii)}]
		      Assume $\ell = 0$. 
		      \begin{itemize}
			      \item[\rm(a)]
			            When $N=1$ with $0<s<1/2$ and $2 \leq N \leq 7$ with $0<s<1$, $p_{JL} = \infty$ holds. 
			      \item[\rm(b)] 
			            When $N=8,9$, there exists a unique $s_N \in (0,1)$ such that 
			            $p_{JL} < \infty$ if $0 < s < s_N$ and $p_{JL} = \infty$ if $s_N \leq s < 1$. 
			      \item[\rm(c)] 
			            When $N \geq 10$ with $0<s<1$, $p_{JL} < \infty$ always holds;
		      \end{itemize}
		\item[\emph{(iii)}]
		      Assume $-2 < \ell < 0$. 
		      \begin{itemize}
			      \item[\rm(a)]
			            When $N=1$, there exists an $s_{1,\ell} \in (-\ell/2, 1/2)$ such that 
			            $p_{JL} < \infty$ if $ - \ell /2 < s < s_{1,\ell}$ and 
			            $p_{JL} = \infty$ if $ s_{1,\ell} \leq s < 1/2$. 
			      \item[\rm(b)] 
			            When $2 \leq N < 10 + 4\ell $, there exist $s_{N,\ell}$, $\tilde{s}_{N,\ell}\in ( - \ell/2 , 1 )$ 
			            with $s_{N,\ell}\le \tilde{s}_{N,\ell}$ such that $p_{JL} < \infty$ if $ - \ell / 2 < s < s_{N,\ell}$ and 
			            $p_{JL} = \infty$ if $ \tilde{s}_{N,\ell} \leq s < 1$. 
			            Moreover, $s_{N,\ell}= \tilde{s}_{N,\ell}$ holds when $ 2 \leq N \leq 6$. 
			      \item[\rm(c)] 
			            When $N \geq 10+ 4\ell$, 
			            $p_{JL} < \infty$ holds for $-\ell/2 < s < 1$. 
		      \end{itemize}
	\end{enumerate}
\end{theorem}

	Theorem \ref{Theorem:1.2} is also useful for $\ell > 0$. 
To see this, we introduce the following condition on $N$ and $s$:
	\begin{enumerate}
	\item[(A)] $N=1$ with $0<s<1/2$, or $2 \leq N \leq 7$ with $0<s<1$, or $N=8,9$ with $s_N \leq s < 1$.
	\end{enumerate}
By Theorem \ref{Theorem:1.2} (ii), $p_{JL} = \infty$ and each $p \in (1,\infty)$ is JL-subcritical under (A) and $\ell = 0$. 
As a corollary of Theorem \ref{Theorem:1.2}, we obtain the following result in which 
we also compare $s_N$ and $s_{N,\ell}$ for the case $\ell < 0$ and $N=8,9$:

	\begin{Corollary}\label{Corollary:1.1}
		Suppose \eqref{eq:1.2}, and let $s_N$ and $s_{N,\ell}$ be the numbers in Theorem {\rm\ref{Theorem:1.2}}. 
		\begin{enumerate}
		\item[{\rm (i)}] 
			Assume {\rm (A)} and $\ell > 0$. Then every $p \in (1,\infty)$ is JL-subcritical. 
		\item[{\rm (ii)}] 
			Assume $-2s<\ell < 0$ and $N=8,9$ with $0<s<s_N$ and $N < 10+4\ell$.
			Then $p_{JL} < \infty$. Therefore, $s_{N} \leq s_{N,\ell}$ holds.
		\end{enumerate}
	\end{Corollary}

	Finally, we focus on the case $0<s \ll 1$ and $\ell > 0$. 
	Under a certain condition, we show that 
	there are at least \emph{two JL-critical exponents} 
	(Theorem \ref{Theorem:1.3} (ii)), which never occurs in the case $s=1$.

\begin{theorem}\label{Theorem:1.3}
Let $N \geq 8$. 
Then there exist $\wt{s}_N \in (0,1)$, $(\ell_{1,s})_{0<s<\wt{s}_{N}}$ 
and  $(\ell_{2,s})_{0<s<\wt{s}_N}$ with $\ell_{i,s} = \ell_i(N,s)$ and 
$0<\ell_{1,s} < \ell_{2,s}$ for each $s \in (0,\wt{s}_N)$ such that 
\begin{enumerate}
\item[{\rm (i)}] 
if $ s \in (0, \wt{s}_N ) $ and $\ell \in (0,\ell_{1,s})$,
then there exist $p_i = p_i(N,s,\ell)$ $(i=1,2)$ such that 
	\begin{enumerate}
	\item[\rm(a)] $p_S (N,\ell) < p_{1} \leq p_{2} < \infty$,
	\item[\rm(b)] every $p \in (1, p_{1}) $ is JL-subcritical,
	\item[\rm(c)] each $p \in ( p_{2} , \infty )$ is JL-supercritical,
	\item[\rm(d)] $p_1$ and $p_2$ are JL-critical,
	\end{enumerate}
\end{enumerate}
and
\begin{enumerate}
\item[{\rm (ii)}] 
	if $ s \in (0,\wt{s}_N) $ and $\ell \in (\ell_{1,s},\ell_{2,s})$, then 
	there exist $p_i = p_i(N,s,\ell)$ $(i=1,2,3,4)$ such that 
		\begin{enumerate}
		\item[\rm(a)] $p_S(N,\ell) < p_1 \leq p_2 < p_3 \leq p_4 < \infty$,
		\item[\rm(b)] every $p \in (1,p_1) \cup (p_4,\infty)$ is JL-subcritical,
		\item[\rm(c)] each $p \in (p_2,p_3)$ is JL-supercritical,
		\item[\rm(d)] $p_1,p_2,p_3$ and $p_4$ are JL-critical. 
		\end{enumerate}
\end{enumerate}
In addition, there exist $(\ell_{3,s})_{0<s<1}$ when $N \geq 10$ and 
$(\ell_{3,s})_{0<s<s_N}$ when $N=8,9$ with $\ell_{3,s} = \ell_3(N,s)$ such that 
if $\ell \geq \ell_{3,s}$, then each $p>1$ is JL-subcritical. 
\end{theorem}

	\begin{remark}\label{Remark:1.2}
		We can also show that for any given $(\ell,s) \in (-2,\infty) \times (0,1)$ with $-2s < \ell$, 
		there exist $N_\ast (\ell,s) \in \N $ and $p_\ast (\ell,s)>1$ such that 
		if $N \geq N_\ast (\ell,s)$, then every $p \in (p_\ast (\ell,s) , \infty)$ is JL-supercritical.  
		See Lemma \ref{Lemma:5.5} and Remark \ref{Remark:5.1}. 
	\end{remark}

\subsection{Comments}
\label{section:1.4}

	We first compare our results with the literature and give some remarks on the results. 
About the existence of the family of positive radial stable solutions, 
Theorem \ref{Theorem:1.1} is an extension of \cite{H,DDW}. 
In \cite{H}, the case $s= 1/2$ and $\ell = 0$ was considered and a similar result to Theorem \ref{Theorem:1.1} was shown. 
On the other hand, in \cite{DDW}, the case $ 0 < s < 1$ and $\ell = 0$ was considered, however, 
the properties (c), (d) and (e) in Theorem \ref{Theorem:1.1} were not proved (see \cite[section 7]{DDW}). 
Therefore, Theorem \ref{Theorem:1.1} (c)--(e) are new even in the case $s \neq 1/2$ and $\ell = 0$, 
and we obtain a counterpart of \cite{H} for the case $0<s<1$ and $-2s < \ell$. 
Here we point out that for the case $s=1$, 
the separation property of positive radial stable solutions was shown in \cite{Wa93,LLD00}. 
In addition, when $\ell = 0$, the separation property is useful in the study of the corresponding parabolic problem. 
For instance, see \cite{GNW-92,GNW-01,PY03,PY05}. 
We hope that our results in this paper might be useful in the study of parabolic problem 
corresponding to \eqref{eq:1.1}.

	For Theorem \ref{Theorem:1.2},  
in \cite[section 4]{H}, when $2 \leq N \leq 5$, $\ell = 0$ and $s = 1/2$, he proved $p_{JL} = \infty$. 
Therefore, Theorem \ref{Theorem:1.2} (i) and (ii) generalize this result even in the case $s = 1/2$ and $\ell = 0$. 
Moreover, we may also treat the case $s \in (0,1)$ and $\ell \in (-2s,0)$ in Theorem \ref{Theorem:1.2}. 
From Theorem \ref{Theorem:1.2}, we see that the case $s \in (0,1)$ and $-2s < \ell \leq 0$ 
is similar to the case $s=1$ in the sense that 
the range of $p$ for being JL-subcritical is bounded by the range of $p$ for being JL-supercritical 
and there is at most one JL-critical exponent.

	On the number $10+4\ell$ in Theorem \ref{Theorem:1.2}, 
we recall that this is the threshold of $N$ for $p_+(N,\ell) = \infty$ when $s=1$ 
(see \cite{DDG} and \eqref{eq:1.4}). 
In Section \ref{section:5} (Lemma \ref{Lemma:5.4} and Remark \ref{Remark:5.1}), 
we observe that as $s \nearrow 1$ and $p \to \infty$, 
Definition \ref{Definition:1.2} coincides with $N< 10+4\ell$, $N=10+4\ell$ and $N>10+4\ell$, 
respectively. Hence, we recover the case $s=1$ from Definition \ref{Definition:1.2}. 
This seems not observed in the literature.

	About Theorem \ref{Theorem:1.3}, we find the situation which is completely different from the case $s=1$. 
In Theorem \ref{Theorem:1.3} (ii), we show that there are at least two JL-critical exponents. 
To the best of the authors' knowledge, this is the first result for the multiple existence of JL-critical exponents. 
We also remark that the range of $p$ for being JL-supercritical is bounded from above and below 
by the range of $p$ for being JL-subcritical. 
This structure is different from the case $s=1$, too. 
We emphasize that this structure and the multiple existence of JL-critical exponents 
are subtle since these phenomena disappear when $\ell > 0$ is sufficiently large 
due to the last statement of Theorem \ref{Theorem:1.3}.

	Next, we comment on the proofs of Theorems \ref{Theorem:1.1}--\ref{Theorem:1.3}. 
For Theorems \ref{Theorem:1.2} and \ref{Theorem:1.3}, 
we first introduce the change of variables and 
show the concavity of some function related to the inequalities in Definition \ref{Definition:1.2} 
for the case $\ell \leq 0$. 
We remark that this change of variables also enables us to treat the case $p = \infty$ and 
to find a function whose positivity (resp. negativity) is equivalent to being JL-subcritical (reps. JL-supercritical) 
at $p=\infty$. Through the analysis at $p=\infty$, the critical number $10+4\ell$ appears as $s \nearrow 1$ 
(see Lemma \ref{Lemma:5.4}).

	For Theorem \ref{Theorem:1.1}, we first show the existence of the family of positive radial stable solutions satisfying 
Theorem \ref{Theorem:1.1} (a) and (b). To this end, 
as in \cite{DDW}, we use an idea in \cite{BCMR-96} 
(see also \cite{D}). 
More precisely, using the singular solution $u_S$ and an auxiliary function in Proposition \ref{Proposition:3.3}, 
we construct a bounded supersolution to the equation on $B_1$ (see \eqref{eq:3.1} and \eqref{eq:3.3}). 
Then we find a family of minimal radial solutions of \eqref{eq:3.3} and 
show the existence of bounded positive stable solutions of \eqref{eq:1.1} via the family of minimal solutions. 
We remark that in \cite{DDW,H}, they considered the equation on $\HS$, namely, \eqref{eq:1.5}. 
On the other hand, we mainly study the equation on $B_1 \subset \RN$. 
Since we study the case $\ell < 0$, weak solutions and unbounded supersolutions on $B_1$, 
\eqref{eq:3.1} or \eqref{eq:3.3} seems simple to treat.  
Therefore, the details are different from \cite{DDW,H}.

	After the existence of the family of solutions, 
we show the separation property (Proposition \ref{Proposition:4.1}) 
for positive radial stable solutions of \eqref{eq:1.1}, which is a key to prove Theorem \ref{Theorem:1.1} (c)--(e). 
Here we argue in the spirit of \cite[section 7]{H}. 
As mentioned in the above, \cite{H} treated the case $s=1/2$ with $\ell = 0$ 
and the extension problem \eqref{eq:1.6} was considered. 
Since $U$ in \eqref{eq:1.6} is harmonic in $\HS$, 
he could use the analyticity of solutions to \eqref{eq:1.6} on $\ov{\HS}$ 
as well as the properties of the spherical harmonic functions for proving the separation property. 
On the other hand, we consider \eqref{eq:1.5} and clearly, the analyticity of solutions to \eqref{eq:1.5} on $\ov{\HS}$ 
cannot be used. Thus, we need to modify the argument in \cite[section 7]{H}. 
Though our proof is lengthy compared to \cite[section 7]{H}, 
we stress that our arguments are mostly elementary.

	To end this introduction, we list some open problems related to our results:
	\begin{enumerate}
	\item[(i)]
	Let $u$ be a radial stable solution of \eqref{eq:1.1}. 
	Then does $u \geq 0$ (resp. $u \leq 0$) in $\RN$ hold?
	\item[(ii)] 
	 Let $u$ be a positive radial solution with $\lim_{r\to 0}u(r)=\infty$. Then does $u=u_S$ hold?
	\end{enumerate}

	When $s=1$ and $\ell = 0$, (i) was proved in \cite[Theorem 5]{Farina}. 
On the other hand, when $s=1$, (ii) was addressed in \cite[Theorem 2.4]{DDG}. 
In both proofs, the ODE techniques were used. 
Thus, we cannot apply the arguments to the case $s \in (0,1)$ directly and 
need another idea. 
For (ii), we remark that 
we show $u = u_S$ under additional assumptions in Corollary \ref{Corollary:4.2}. 
It is not clear that we may remove these additional assumptions. 
Finally, as mentioned above, when $s=1$, the assertions in Theorem \ref{Theorem:1.1} are useful 
in the study of the corresponding parabolic problem. 
Therefore, it is interesting to study the parabolic problem corresponding to \eqref{eq:1.1}.

This paper is organized as follows. In Section \ref{section:2}, we study the existence of the singular solution and its property, 
and prove Theorem \ref{Theorem:1.1} (i). 
In Section \ref{section:3}, we construct the family of positive radial stable solutions. 
Section \ref{section:4} is devoted to the proof of the separation property (Proposition \ref{Proposition:4.1}) and 
its consequences. After that, we complete the proof of Theorem \ref{Theorem:1.1}. 
In Section \ref{section:5}, we prove Theorem \ref{Theorem:1.2}, Corollary \ref{Corollary:1.1} and Theorem \ref{Theorem:1.3}. 
Appendices contain proofs of some technical lemmata.



\section{Proof of Theorem \ref{Theorem:1.1} (i)}\label{section:2}


In this section, we prove Theorem \ref{Theorem:1.1} (i).

\begin{proposition}\label{Proposition:2.1}
Suppose \eqref{eq:1.2}.
Let $u_S(x) := A_0 |x|^{ - \theta_0}$, where $\theta_0$ and $A_0$ are as in \eqref{eq:1.8} and \eqref{eq:1.10}, respectively.
	\begin{enumerate}
	\item[{\rm (i)}] 
	 Let $p>(N+\ell)/(N-2s)$.
		Then $u_S$ is a singular solution of \eqref{eq:1.1} in the sense of $\scrS^\ast$, 
		where $\scrS$ denotes the Schwartz space. 
	\item[{\rm (ii)}] 
		Let $p>p_S(N,\ell) $. Then $u_{S}\in \Hsloc(\RN)$, $u_{S}\not\in \dHs(\RN)$ and 
		$u_S$ is a solution of \eqref{eq:1.1} in the sense of Definition {\rm\ref{Definition:1.1}}.
	\end{enumerate}
\end{proposition}

\begin{proof}
(i) For $0< \alpha < N$, we set 
\[
	\gamma(\alpha) := \frac{\pi^{\frac{N}{2}} 2^\alpha \Gamma \left( \frac{\alpha}{2} \right) }
	{\Gamma \left( \frac{N-\alpha}{2} \right)}.
\]
It is known that (see, e.g., \cite[Chapter V]{Ste-70})
	\begin{equation}\label{eq:2.1}
			 \scrF \left[ |x|^{-\alpha} \right] (\xi) 
			 = \gamma  (N-\alpha) | (2\pi |\xi|)^{-(N-\alpha)} \quad 
			\text{in} \ (\scrS)^\ast.
	\end{equation}
Notice that for $\varphi \in \scrS$,  $(-\Delta)^s \varphi = \ov{(-\Delta)^s \ov{\varphi}}$ and 
	\[
		\left( - \Delta \right)^s \varphi = \scrF^{-1} \left( \left( 2\pi |\xi| \right)^{2s} \scrF \varphi \right) 
		= \scrF \left( \left( 2\pi |\xi| \right)^{2s} \scrF^{-1} \varphi \right). 
	\]
Thus, it suffices to show 
	\begin{equation}\label{eq:2.2}
		\left( 2\pi |\xi| \right)^{2s} \scrF u_S (\xi) = \scrF \left( |x|^\ell u_S^p(x) \right) (\xi) \quad \text{in} \ \left( \scrS \right)^\ast.
	\end{equation}
Since it follows from \eqref{eq:1.2} and $p>(N+\ell)/(N-2s)$ that $0< \theta_0<N$, 
by \eqref{eq:2.1} with $\alpha=\theta_0$ 
we observe that
	\begin{equation}\label{eq:2.3}
		\begin{aligned}
			 (2\pi |\xi|)^{2s} \scrF u_S (\xi) 
			 & = A_0 \gamma \left( N - \theta_0 \right)
			(2\pi |\xi|)^{ - \left( N - \theta_0 \right) + 2s }  
			= A_0 \gamma \left( N - \theta_0 \right)
			(2\pi |\xi|)^{ - \left( N - \frac{2sp + \ell}{p-1} \right)  } .
		\end{aligned}
	\end{equation}
	On the other hand, since
	\[
		|x|^\ell u_S^p(x) 
		= A_0^p |x|^{ \ell - p\theta_0 } 
		= A_0^p |x|^{ - \frac{\ell + 2sp }{p-1} },
	\]
	and it follows from $p>(N+\ell)/(N-2s)$ that $(2sp+\ell)/(p-1)<N$,
	by \eqref{eq:2.1} with $\alpha=(2sp+\ell)/(p-1)$ we have 
	\begin{equation}\label{eq:2.4}
		\scrF \left( |x|^\ell u_S^p \right) (\xi)
		= A_0^p \gamma \left( N - \frac{\ell + 2sp}{p-1} \right)
		\left( 2\pi |\xi| \right)^{ - \left( N - \frac{\ell + 2sp}{p-1} \right) }.
	\end{equation}
	Thus, by \eqref{eq:1.1}, \eqref{eq:2.3} and \eqref{eq:2.4}, \eqref{eq:2.2} holds provided 
	\begin{equation}
		A_0^{p-1} 
		= \frac{\gamma \left( N - \theta_0 \right)}
		{\gamma \left( N - \frac{\ell + 2sp}{p-1} \right)}.\notag
	\end{equation}
	Recalling the definitions of $\gamma$ and $\lambda$, we have
	\begin{equation*}
		\frac{\gamma \left( N - \theta_0 \right)}
		{\gamma \left( N - \frac{\ell + 2sp}{p-1} \right)}
		= 2^{2s} \frac{\Gamma \left( \frac{1}{2} \left( N - \theta_0 \right) \right)
			\Gamma \left( \frac{2sp + \ell}{2(p-1)} \right) }
		{ \Gamma \left( \frac{1}{2} \left( N - \frac{2sp+\ell}{p-1} \right) \right)
			\Gamma \left(  \frac{\theta_0}{2} \right)
		}
		= \lambda \left( \frac{N-2s}{2} - \theta_0 \right)
		= A_0^{p-1},
	\end{equation*}
	hence, \eqref{eq:2.2} and (i) hold.

	(ii) 
	We first remark that $p > p_S(N,\ell)$ yields 
		\begin{equation}\label{eq:2.5}
			\frac{\ell + 2sp}{p-1} < N, \quad 
			|x|^\ell u_S^p = A_0^p |x|^{ - \frac{\ell + 2sp}{p-1} } \in L^{ \frac{2N}{N+2s} }_{\rm loc} (\RN), \quad 
			u_S \in L^1(\RN, (1+|x|)^{-N-2s}).
		\end{equation}
	Next, we prove $u_{S}\in \Hsloc(\RN)$. 
	Let $\varphi \in \scrS$.
	We shall show $ (1+|\xi|^2)^{s} \left| \scrF[\varphi u_S] (\xi) \right|^2 \in L^1(\RN)$. 
	Since $\scrF u_S (\xi) = C |\xi|^{- N + \theta_0  }$ for some $C > 0$ 
	and $\theta_0<N$, 
	we have $\scrF[\varphi u_S]= \left( \scrF\varphi \ast \scrF u_S \right)\in L^\infty (\RN) $. 
	Moreover,
	it follows from \eqref{eq:1.2} and $p>p_S(N,\ell)$ that $\theta_0<(N-2s)/2$ and  
	\begin{equation}\label{eq:2.6}
		-N < - N + \theta_0 < - \frac{N}{2} -s.
	\end{equation}	
	Then it is easily seen that there exists a $C_0 > 0$ such that 
	\[
		\left| \left( \scrF\varphi \ast \scrF u_S \right) (\xi) \right| 
		\leq C_0 \left( 1 + |\xi| \right)^{-N + \theta_0 } \quad 
		\text{for all } \xi \in \RN.
	\]
	By \eqref{eq:2.6} one sees that 
	\[
		2s + 2 \left( - N + \theta_0 \right)  < - N, \quad 
		\left( 1 + |\xi|^2 \right)^s \left| \scrF[\varphi u_S] (\xi) \right|^2 \in L^1(\RN).
	\]
	This implies $\varphi u_S \in H^s(\RN)$, and we obtain $u_{S}\in \Hsloc(\RN)$.
	
	Next, we prove $u_{S}\not\in \dHs(\RN)$. 
	Since $\scrF u_S (\xi) = C |\xi|^{- N+\theta_0}$ for some $C>0$, 
	by considering $|\xi|^{2s} |\scrF u_S (\xi)|^2$ and looking at the index, we get 
	\[
		2 s - 2 N + 2\theta_0
		< 2s - 2N + (N-2s) = - N, \quad |\xi|^{2s} \left|\scrF u_S (\xi) \right|^2 \not \in L^1(B_1). 
	\]
	Thus, $u_S \not\in \dot{H}^s(\RN)$.

	Finally, by \cite[Lemma 2.1]{HIK-20}, $u_S \in H^s_{\rm loc} (\RN)$ and \eqref{eq:2.5}, 
	we observe that for any $\varphi \in C^\infty_c (\RN)$,
	\[
		\int_{  \RN} |x|^\ell u_S^p \varphi \, dx 
		= 
		\la \left( -\Delta \right)^s u_S, \varphi \ra_{\scrS^\ast, \scrS} 
		= \int_{  \RN} u_S \left( -\Delta \right)^s \varphi \, dx = \la u_S , \varphi \ra_{\dHs (\RN)}.
	\]
	Thus, $u_S$ is a solution of \eqref{eq:1.1} in the sense of Definition \ref{Definition:1.1} and we complete the proof. 
\end{proof}

\begin{proof}[Proof of Theorem {\rm\ref{Theorem:1.1}} {\rm(i)}]
	By Proposition \ref{Proposition:2.1} it is enough to show that $u_S$ is stable, that is, 
	\begin{equation}\label{eq:2.7}
		\| \varphi \|_{\dot{H}^s(\RN)}^2 - p \int_{\RN} |x|^\ell u_S^{p-1} \varphi^2 \,d x 
		\geq 0 
		\quad \text{for every $\varphi \in C^\infty_0(\RN)$}.
	\end{equation}
	To this end, we first recall the Hardy type inequality (see, e.g., \cite{FLS-08, H-77,Y-99})
	\begin{equation}\label{eq:2.8}
		\lambda (0) \int_{\RN} |x|^{-2s} \varphi^2 \,dx 
		\leq \int_{\RN} |\xi|^{2s} \left| F \left[\varphi\right](\xi) \right|^2 \, d \xi,
	\end{equation}
	where
	\[
	F \left[ \varphi \right] (\xi)
		:= (2\pi)^{-\frac{N}{2}} \int_{\RN} e^{- ix \cdot \xi} \varphi(x) \,d x.
	\]
	By $F[\varphi] (\xi) = (2\pi)^{-N/2} \scrF \varphi \left( \xi / (2\pi) \right)$ and 
	\begin{equation}\label{eq:2.9}
		\int_{\RN} |\xi|^{2s} \left| F \left[\varphi\right](\xi) \right|^2 \, d \xi 
		= \int_{\RN} \left( 2\pi |\xi| \right)^{2s} \left| \scrF{\varphi} (\xi)  \right|^2 \, d \xi 
		= \| \varphi \|_{\dot{H}^s(\RN)}^2, 
	\end{equation}
	it follows from \eqref{eq:2.8} and \eqref{eq:2.9} that
	\begin{equation}\label{eq:2.10}
		\lambda(0) \int_{\RN} |x|^{-2s} \varphi^2 \, dx
		\leq \| \varphi \|_{\dot{H}^s(\RN)}^2 
		\quad \text{for each $\varphi \in C^\infty_0(\RN)$}.
	\end{equation}
	On the other hand, since 
	\[
		|x|^\ell u_S^{p-1} (x) 
		=A_0^{p-1} |x|^{-2s} 
		= \lambda \left( \frac{N-2s}{2} - \theta_0 \right) |x|^{-2s},
	\]
	we observe from \eqref{eq:1.9} and \eqref{eq:2.10} that 
	\[
		\begin{aligned}
			\| \varphi \|_{\dot{H}^s(\RN)}^2 - p \int_{\RN} |x|^\ell u_S^{p-1} \varphi^2 \, dx 
			 & = \| \varphi \|_{\dot{H}^s(\RN)}^2 
			- p  \lambda \left( \frac{N-2s}{2} - \theta_0 \right) \int_{\RN} |x|^{-2s} \varphi^2 \, dx 
			\\
			 & \geq \| \varphi \|_{\dot{H}^s(\RN)}^2 
			- \lambda (0) \int_{\RN} |x|^{-2s} \varphi^2 \, dx \geq 0.
		\end{aligned}
	\]
	Hence, \eqref{eq:2.7} holds and we complete the proof. 
\end{proof}


\section{Construction of a family to stable radial solutions}\label{section:3}


In this section, we construct stable solutions of \eqref{eq:1.1}.
Throughout this section, we assume \eqref{eq:1.9}.
Let $\mu\in(0,1)$ and $u_S$ be a singular solution of \eqref{eq:1.1} in Theorem~\ref{Theorem:1.1} (i).
We first consider the following equation: 
\begin{equation}\label{eq:3.1}
	(-\Delta )^s u = |x|^\ell |u|^{p-1} u \quad \text{in} \ B_1, \quad 
	u = \mu \eta_\mu u_S \quad \text{in} \ \RN \setminus B_1,
\end{equation}
where $\eta_\mu \in C^\infty_c(\RN)$ is a cut-off function satisfying 
\begin{equation}\label{eq:3.2}
	\begin{aligned}
		 & \eta_\mu (x) := \eta_0 \left( (1-\mu) x \right), \quad \eta_0(x) = \eta_0(|x|), \quad
		0 \leq \eta_0(r) \leq 1\ \ \text{in} \ \ [0,\infty),                                               \\
		 & \eta_0 (r) \equiv 1 \ \ 
		\text{in} \ \ [0,1], \quad \eta_0(r) \equiv 0 \ \ \text{in} \ \ [2,\infty),\quad
		\eta_0'(r) \leq 0 \ \ \text{in} \ \ [0,\infty).
	\end{aligned}
\end{equation}
Remark that $\eta_{\mu} u_S \in H^s(\RN)$ and $\eta_{\mu} (x) \to 1$ as $\mu \nearrow 1$.

	In what follows, we only treat the existence of positive solutions. 
Therefore, we replace $ |u|^{p-1} u $ in \eqref{eq:3.1} by $u_+^p$ and consider 
	\begin{equation}\label{eq:3.3}
		\left( - \Delta \right)^s u = |x|^\ell u_+^p \quad \text{in} \ B_1, \quad u = \mu \eta_{\mu} u_S \quad 
		\text{in} \ \RN \setminus B_1.
	\end{equation}

First we give the definition of solutions, supersolutions and subsolutions of \eqref{eq:3.3}. 
Remark that $L^\infty(\RN) \subset L^1(\RN, (1+|x|)^{-N-2s}  d x)$.

\begin{definition}\label{Definition:3.1}
	We say that $u$ is a \emph{solution of \eqref{eq:3.3}} if $u$ satisfies 
	\[
		\begin{aligned}
			 & u \in \Hsloc (\RN), \quad  u= \mu \eta_\mu u_S  \quad 
			\text{in} \ \RN \setminus B_1, \quad 
			|x|^\ell u_+^p \in L^{ \frac{2N}{N+2s} } (B_1),
			\\
			 & \la u , \varphi \ra_{\dot{H}^s(\RN)}
			= \int_{\RN} |x|^\ell u^p_+ \varphi \, d x 
			\quad \text{for every $\varphi \in C^\infty_c(B_1)$}.
		\end{aligned}
	\]
	Similarly, $u$ is said to be a \emph{supersolution of \eqref{eq:3.3}}
	(resp. \emph{subsolution}) if $u$ satisfies 
	\[
		\begin{aligned}
			 & u \in \Hsloc(\RN) \cap L^1 ( \RN, (1+|x|)^{-N-2s} d x ), \quad 
			|x|^\ell u_+^p \in L^{ \frac{2N}{N+2s} } (B_1) ,
			\\
			 & u \geq \mu \eta_\mu u_S \quad \text{in} \ \RN \setminus B_1
			\qquad \left( \text{resp.} \ u \leq \mu \eta_\mu u_S \quad \text{in} \ \RN \setminus B_1  \right),
			\\
			 & \la u , \varphi \ra_{\dot{H}^s(\RN)} \geq \int_{\RN} |x|^\ell u_+^p \varphi  \, d x 
			\quad \text{for all $\varphi \in C^\infty_c(B_1)$ with $\varphi \geq0$}
			\\
			 & \left( \text{resp.} \ \
			\la u , \varphi \ra_{\dot{H}^s(\RN)} \leq \int_{\RN} |x|^\ell u_+^p \varphi \, d x 
			\quad \text{for all $\varphi \in C^\infty_c(B_1)$ with $\varphi \geq0$}
			\right).
		\end{aligned}
	\]
\end{definition}

\begin{remark}\label{Remark:3.1}
	\begin{enumerate}
		\item[(i)] 
			  Replacing $u_+^p$ by $|u|^{p-1}u$ in Definition \ref{Definition:3.1}, 
			  we may define the notation of solutions, supersolutions and subsolutions of \eqref{eq:3.1}. 
		\item[(ii)] 
			Since $p>p_S(N,\ell)$ implies $|x|^\ell u_S^p \in L^{ \frac{2N}{N+2s} } (B_1)$, 
			it is easily seen that $u_S$ is a supersolution of \eqref{eq:3.3} for all $ \mu \in (0,1]$. 
		\item[(iii)]
		      For solutions $u$ of \eqref{eq:3.1}, we require $u$ to satisfy $u = \mu \eta_\mu u_S$ in 
		      $\RN \setminus B_1$. Hence, $ u \in H^s(\RN) \cap L^1(\RN, (1+|x|)^{-N-2s} d x)$.
		\item[(iv)]
		      Since $H^s(\RN) \subset L^{2^\ast_s} (\RN)$ with $2^\ast_s:=(2N)/(N-2s)$, 
		      $L^{2^\ast_s} (B_1) = ( L^{ \frac{2N}{N+2s} } (B_1) )^\ast $ and 
		      we require the integrability condition $|x|^\ell u_+^p \in L^{ \frac{2N}{N+2s} } (B_1)$, 
		      in Definition \ref{Definition:3.1}, we may replace $C^\infty_c(B_1)$ by $\cH^s_0(B_1)$, where 
		      \[
			      \cH^s_0(B_1) := \Set{ u \in H^s(\RN) | u \equiv 0 \quad \text{in} \ \RN \setminus B_1 }.
		      \]
		      We remark that the following may be checked:  
		      \[
			      \cH^s_0(B_1) = \ov{C^\infty_c(B_1)}^{ \| \cdot \|_{H^s(\RN)} }. 
		      \]
	\end{enumerate}
\end{remark}

In order to prove the existence of solutions of \eqref{eq:3.1}, 
we shall find a bounded supersolution of \eqref{eq:3.3} for $0 < \mu < 1$. 
To this end, we need some preparations. 
\begin{proposition}\label{Proposition:3.1}
	Suppose that $u,v \in H^s_{\rm loc} (\RN) \cap L^1(\RN,  (1+|x|)^{-N-2s} d x ) $ satisfy 
	\begin{equation}\label{eq:3.4}
		\begin{aligned}
			 & u \leq v \quad \text{in} \ \RN \setminus B_1, \quad 
			 \la u, \varphi \ra_{\dot{H}^s(\RN)} \leq \la v, \varphi \ra_{\dot{H}^s(\RN)}
			\quad \text{for all $\varphi \in \cH^s_0(B_1)$ with $\varphi \geq 0$ in $B_1$}.
		\end{aligned}
	\end{equation}
	Then $u \leq v $ in $\RN$. In particular, if $v$ is any supersolution of \eqref{eq:3.3}, then $v \geq 0$ in $\RN$. 
\end{proposition}

\begin{remark}\label{Remark:3.2}
	By \cite[Lemma 2.1]{HIK-20} and the density argument, 
	$\la u , \varphi \ra_{\dot{H}^{s} (\RN)} \in \R$ for all $u \in H^s_{\rm loc} (\RN) \cap L^1(\RN, (1+|x|)^{-N-2s} d x )$
	and $\varphi \in \cH^s_0(B_1)$.
\end{remark}

\begin{proof}[Proof of Proposition {\rm\ref{Proposition:3.1}}]
	Set $w:=u-v$.
	Remark that $(a_+-b_+)^2 \leq (a_+-b_+) ( a-b ) \leq (a-b)^2 $ for all $a,b \in \R$. 
	Thus, $w_+ \in \cH^s_0(B_1)$ holds by the assumption and 
	\begin{equation}\label{eq:3.5}
		0 \geq \la w, w_+ \ra_{\dot{H}^s(\RN)} \geq \| w_+ \|_{\dot{H}^s(\RN)}^2.
	\end{equation}
	Hence, $w_+ \equiv 0$ in $\RN$ and $u \leq v$ in $\RN$ holds.

	Next, let $u \equiv 0$. If $v$ is any supersolution of \eqref{eq:3.3}, then 
for every $\varphi \in \cH^s_0(B_1)$ with $\varphi \geq 0$ in $B_1$, 
	\[
		\begin{aligned}
			 & \la v, \varphi \ra_{\dot{H}^s (\RN)} \geq \int_{ \RN } |x|^\ell v_+^p \varphi \, \rd x 
			\geq 0 = \la u, \varphi \ra_{\dot{H}^s (\RN)}, \quad 
			 v \geq \mu \eta_\mu u_S \geq 0 = u \quad \text{in} \ \RN \setminus B_1.
		\end{aligned}
	\]
	Since $u$ and $v$ satisfy \eqref{eq:3.4}, we have $ 0 = u(x) \leq v(x)$ in $\RN$. 
\end{proof}
	
	Next we recall the strong maximum principle for the fractional Laplacian.
	
	\begin{proposition}{\rm(\cite[Theorem~1.1]{JW-19} and \cite[Theorem 4.1]{MN19})}
\label{Proposition:3.2}
	Let $\Omega\subset \RN$ be a domain, $c \in L^\infty_{\rm loc} (\Omega)$ and 
	$u \in \Hsloc( \RN ) \cap L^2 ( \RN, (1+|x|)^{-N-2s} d x )$ satisfy 
	\[
		u \geq 0 \quad \text{in} \ \RN, \quad \la u , \varphi \ra_{\dHs (\RN)} \geq \int_{\Omega} c(x) u \varphi \, dx 
		\quad \text{for any $\varphi \in C^\infty_c(\Omega)$ with $\varphi \geq 0$}.
	\]
	Then either $u \equiv 0$ in $\Omega$ or else $\inf_{K} u > 0$ for each compact $K \subset \Omega$. 
	In particular, if $u \not \equiv 0$, $u \geq 0$ in $\RN$ and $\la u, \varphi \ra_{\dHs(\RN)} \geq 0$ 
	for every $\varphi \in C^\infty_c(\Omega)$ with $\varphi \geq 0$, 
	then $u > 0$ a.e. in $\Omega$. 
\end{proposition}

	Next, to construct a bounded supersolution of \eqref{eq:3.3}, 
we follow the scheme in \cite{BCMR-96} and show

\begin{proposition}\label{Proposition:3.3}
	For each $\mu \in (0,1)$, there exists $\Phi_\mu: [0,\infty) \to [0,\infty)$ satisfying the following{\rm :}
	\begin{enumerate}
		\item[{\rm (i)}] $\Phi_\mu \in W^{2,\infty}( (0,\infty) ) $;
		\item[{\rm (ii)}] $\Phi_\mu'(u)$ is nonincreasing on $[0,\infty)$ and $\Phi_\mu'(u) > 0$ in $[0,\infty)$;
		\item[{\rm (iii)}] $\Phi_\mu (u) =  u $ if $u \in [0,M_0]$, where $M_0 := u_S(1)$;
		\item[{\rm (iv)}] For each $u \in [0,\infty)$, $\Phi_\mu'(u) u^p \geq \mu^{p-1} \Phi_\mu^p (u) $.
	\end{enumerate}
\end{proposition}

\begin{proof}
	Fix $\varphi_0 \in C^1( [0, M_0+1] )$ so that 
	$\varphi_0 (u) = 1$ on $[0,M_0]$, $\varphi_0$ is nonincreasing in $[M_0,M_0+1]$ and 
	$\varphi_0 (M_0+1) = 0$. Remark that due to $\mu \in (0,1)$ and $p > 1$, 
	\begin{equation*}
		\varphi_0 (u) u^p > 
		\mu^{p-1} u^p =   \mu^{p-1} \left( \int_{0}^{u} \varphi_0 (t) \, d t \right)^p 
		\quad \mathrm{for\ all}\ u \in (0,M_0].
	\end{equation*}
	We set 
	\[
		\zeta_\mu := \sup \Set{ u \in [M_0,M_0+1] |
		\varphi_0 (t) t^p > \mu^{p-1} \left( \int_{0}^{t} \varphi_0 (\tau ) \, d \tau \right)^p
		\quad \mathrm{for\ each} \ t \in [M_0,u] }.
	\]
	By $\varphi_0 (M_0+1) = 0$ and the monotonicity of $\varphi_0$, 
	we see that 
	\begin{equation}\label{eq:3.6}
		\begin{aligned}
			 & M_0 < \zeta_\mu < M_0 + 1, \quad 
			\varphi_0 (\zeta_\mu) \zeta_\mu^p 
			= \mu^{p-1} \left( \int_{0}^{\zeta_\mu} \varphi_0 (\tau) \, d \tau \right)^p, 
			\\
			 & \varphi_0 (u) u^p \geq \mu^{p-1} \left( \int_{0}^{u} \varphi_0 (t) \, d t \right)^p 
			\quad \mathrm{for \ any \ } u \in [0,\zeta_\mu], \quad 
			\varphi_0 (\zeta_\mu) < 1.
		\end{aligned}
	\end{equation}

	Now we define $\Phi_\mu(u)$ by 
	\begin{equation}\label{eq:3.7}
		\Phi_\mu (u) := \left\{\begin{aligned}
			 & \int_{0}^{u} \varphi_0 (t)\, d t &  & 
			\mathrm{if} \ 0 \leq u \leq \zeta_\mu,                                 \\
			 & \Psi_\mu (u)                       &  & \mathrm{if} \ \zeta_\mu < u,
		\end{aligned}\right.
	\end{equation}
	where $\Psi_\mu(u)$ is a unique solution of 
	\begin{equation}\label{eq:3.8}
		\Psi_\mu '(u) = \mu^{p-1} \Psi_\mu^p (u) u^{-p} \quad 
		\mathrm{in} \ (\zeta_\mu , \infty), \quad 
		\Psi_\mu (\zeta_\mu) = \int_{0}^{\zeta_\mu} \varphi_0 (\tau) \, d \tau.
	\end{equation}
	Rewriting \eqref{eq:3.8} as 
	\[
		\frac{d}{d u} \Psi_\mu^{1-p} (u) = \mu^{p-1} \frac{d}{d u} u^{1-p} \quad 
		\mathrm{in} \ (\zeta_\mu,\infty),
	\]
	we see that 
	\begin{equation}\label{eq:3.9}
		\Psi_\mu (u) = \left\{ \Psi_\mu^{1-p}(\zeta_\mu)
		+ \mu^{p-1} \left( u^{1-p} - \zeta_\mu^{1-p} \right) \right\}^{- \frac{1}{p-1}}.
	\end{equation}
	Moreover, it follows from \eqref{eq:3.6} and \eqref{eq:3.8} that 
	\begin{equation}\label{eq:3.10}
		\Psi_\mu'( \zeta_\mu ) = \mu^{p-1} \Psi_\mu^p (\zeta_\mu ) \zeta_\mu^{-p}
		= \varphi_0 (\zeta_\mu)<1.
	\end{equation}

	We shall prove $\Psi_\mu \in W^{2,\infty} ( (\zeta_\mu,\infty) )$. 
	To this end, thanks to \eqref{eq:3.9}, it suffices to show 
	\begin{equation}\label{eq:3.11}
		\zeta_\mu^{p-1} > \mu^{p-1} \Psi_\mu^{p-1} (\zeta_\mu) .
	\end{equation}
	Noting the monotonicity of $\varphi_0$ and $ \varphi_0 (t) \leq 1$ in $[0,M_0+1]$ with $\varphi_0 (\zeta_\mu) < 1$, 
	we get 
	\begin{equation}\label{eq:3.12}
		\Psi_\mu (\zeta_\mu ) = \int_{0}^{\zeta_\mu} \varphi_0 (t) \, d t < \zeta_\mu.
	\end{equation}
	Since $\mu \in (0,1)$, \eqref{eq:3.11} is verified. 
	Thus, $\Phi_\mu \in W^{2,\infty} ((0,\infty) )$, and (i) holds.

	For (ii), by the choice of $\varphi_0$ and \eqref{eq:3.7}, 
	it is enough to show $\Psi_\mu''(u) \leq 0$ on $(\zeta_\mu,\infty)$. 
	Noting \eqref{eq:3.10} and \eqref{eq:3.12}, 
	we have $\Psi_\mu (u) \leq u$ in $[\zeta_\mu,\zeta_\mu + L_0]$ for some $L_0>0$. 
	By \eqref{eq:3.8}, if the inequality $\Psi_\mu (u) \leq u$ in $[\zeta_\mu, \zeta_\mu + L]$ holds for some $L>0$, then 
	$\Psi_\mu'(u) \leq \mu^{p-1} < 1$ in $[\zeta_\mu, \zeta_\mu + L]$
	and $\Psi_\mu(u) \leq u$ in $[\zeta_\mu, \zeta_\mu+L_1]$ for some $L_1 > L$. Hence, repeating this argument and 
	noting $\Psi_\mu (u) \leq u$ in $[\zeta_\mu,\zeta_\mu+L_0]$, 
	we observe $\Psi_\mu (u) \leq u$ for all $u \in [\zeta_\mu,\infty)$. 
	This together with \eqref{eq:3.8} implies that for $\zeta_\mu < u$,
	\[
		\begin{aligned}
			\Psi_\mu''(u) 
			 & = p \mu^{p-1} \Psi_\mu^{p-1} (u) \Psi_\mu'(u) u^{-p}
			- p \mu^{p-1} \Psi_\mu^p (u) u^{-p-1}
			\\
			 & = p \mu^{2(p-1)} u^{-2p} \Psi_\mu^{2p-1}  (u) - p \mu^{p-1} \Psi_\mu^p (u) u^{-p-1}
			\\
			 & = p \mu^{p-1} u^{-2p} \Psi_\mu^p (u) 
			\left\{ \mu^{p-1} \Psi_\mu^{p-1}  (u) - u^{p-1} \right\} 
			\leq 0.
		\end{aligned}
	\]
	Hence, $\Phi_\mu''(u) \leq 0$ in $(\zeta_\mu,\infty)$.

	(iii) and (iv) follow from \eqref{eq:3.6}, \eqref{eq:3.8} and the definition of $\Phi_\mu$ and $\varphi_0$, 
	and we complete the proof. 
\end{proof}

\begin{remark}\label{Remark:3.3}
	From the proof of Proposition \ref{Proposition:3.3}, we observe that 
	$\Phi_\mu (u) \leq u$ and $0< \Phi_\mu'(u) \leq 1$ for any $u \in [0,\infty)$. 
\end{remark}

\begin{proposition}\label{Proposition:3.4}
	For $ 0< \mu \leq 1$, set 
	\[
		\ov{\cS}_\mu := \Set{ u  | \text{$u$ is a supersolution of \eqref{eq:3.3}} }.
	\]
	If $u \in \ov{\cS}_1$, then for any $\mu \in (0,1)$, 
	$ \Phi_\mu \left( \mu u \right) \in \ov{\cS}_\mu \cap L^\infty(\RN) $. 
\end{proposition}

\begin{proof}
	Let $u \in \ov{\cS}_1$. 
	Then, since it follows from Proposition \ref{Proposition:3.1} that $u \geq 0$ in $\RN$,
	we see that $v(x) := \Phi_\mu ( \mu u(x) )$ is well-defined. 
	By $u \in H^s_{\rm loc} (\RN)$ and $\Phi_\mu \in W^{2,\infty} ((0,\infty))$ due to Proposition \ref{Proposition:3.3}, 
	we find $v \in H^s_{\rm loc} (\RN) \cap L^\infty(\RN)$ and $|x|^\ell v^p \in L^{\frac{2N}{N+2s}} (B_1)$.

	In what follows, we shall show $v \in \ov{\cS}_\mu$. 
	Set 
	\[
		\varphi_k (x) := \varphi_1 (k^{-1} x), \quad 
		\varphi_1 \in C^\infty_c(\RN), \quad 
		\varphi_1 \equiv 1 \quad \text{in} \ B_1, \quad 
		\varphi_1 \equiv 0 \quad \text{in} \ \RN \setminus B_2.
	\]
	For each $\psi \in C^\infty_c(B_1)$, thanks to $u \in  \Hsloc (\RN) \cap L^1(\RN, (1+|x|)^{-N-2s} d x) $, 
	by \cite[Lemma 2.1]{HIK-20} we have 
	\begin{equation}\label{eq:3.13}
		\begin{aligned}
			        & \left| \la u , \psi \ra_{\dot{H}^s (\RN)}
			- \la \varphi_{k} u , \psi \ra_{\dot{H}^s (\RN)}  \right| 
			\\
			\leq \  & 
			C
			\left( \|  u - \varphi_{k} u \|_{\dHs (B_{2})} \| \psi \|_{\dHs(B_2)} + 
			\|  u - \varphi_{k} u \|_{L^2(B_{2})} \| \psi \|_{L^2(B_2)}
			\right.
			\\
			 &\hspace{4cm} 
			 \left. 
			+ \|  u - \varphi_{k} u \|_{L^1(\RN , (1+|x|)^{-N-2s} d x )} \| \psi \|_{L^1(B_2)}   \right),
		\end{aligned}
	\end{equation}
	where $C>0$ depends only on $N$ and $s$. 
	Therefore, it is easily seen that 
	\begin{equation*}
		\la \varphi_k u , \psi \ra_{\dot{H}^s (\RN)} \to \la u , \psi \ra_{\dot{H}^s (\RN)}
		\quad \text{as} \ k \to \infty.
	\end{equation*}

	Next, let $(\rho_\e)_\e$ be a mollifier and put $u_{\e,k} := \rho_\e \ast ( \varphi_k u) $. 
	Then $\| u_{\e,k} - \varphi_k u \|_{H^s(\RN)} \to 0$ as $\e \to 0$. 
	Now we define $v_{\e,k}$ and $v_k$ by
	\[
		\begin{aligned}
			v_{\e,k} (x) := \Phi_\mu \left( \mu u_{\e,k} (x) \right), \quad 
			v_k (x) := \Phi_\mu \left( \mu \varphi_k(x) u(x) \right). \quad 
		\end{aligned}
	\]
	Since $\Phi_\mu \in W^{2,\infty} ((0,\infty))$ and $\varphi_k u \in H^s(\RN)$, we see that 
	\begin{equation}\label{eq:3.14}
		\begin{aligned}
			 & v_{\e,k} \in W^{2,\infty}(\RN) \cap L^\infty(\RN), \quad 
			\text{$\supp v_{\e,k}$ is compact}, \quad 
			v_k \in H^s(\RN) \cap L^\infty(\RN),
			\\
			&
			v_{\e,k} \rightharpoonup v_k \quad \text{weakly in } H^s(\RN) \quad \text{as $\e \to 0$.}
		\end{aligned}
	\end{equation}
	Moreover, by Proposition \ref{Proposition:3.3} (ii) and Lemma \ref{Lemma:B.1}, it may be verified that
	\[ 
		\left( - \Delta \right)^s v_{\e,k} (x)
		= \left( - \Delta \right)^s \Phi_\mu \left( \mu u_{\e,k} \right) (x)
		\geq
		\Phi_\mu' \left( \mu  u_{\e,k}(x) \right) \left(-\Delta\right)^s \left( \mu u_{\e,k} \right) (x)
		\quad \text{for each $x \in \RN$}.
	\]
Thus, for $\psi \in C^\infty_c(B_1)$ with $\psi \geq 0$, we obtain 
	\begin{equation}\label{eq:3.15}
		\la v_{\e,k} , \psi \ra_{\dot{H}^s (\RN)}
		\geq 
		\la \mu u_{\e,k} , \Phi_\mu'( \mu u_{\e,k} ) \psi \ra_{\dot{H}^s (\RN)} .
	\end{equation}
	Due to $\Phi_\mu'( t ) \in W^{1,\infty} ((0,\infty) )$, we observe that 
	$( \Phi_\mu'(\mu u_{\e,k}) \psi )_{\e} $ is bounded in $H^s(\RN)$ and 
	it is easily seen that 
	\[
		\Phi_\mu'( \mu u_{\e,k} ) \psi 
		\rightharpoonup \Phi_\mu'( \mu \varphi_k u ) \psi
		\quad \text{weakly in $H^s(\RN)$ as $\e \to 0$.}
	\]
	By letting $\e \to 0$ in \eqref{eq:3.15}, it follows from \eqref{eq:3.14} and $\| u_{\e,k} - \varphi_k u \|_{H^s(\RN)} \to 0$ that 
	\begin{equation}\label{eq:3.16}
		\la v_k , \psi \ra_{\dot{H}^s (\RN)} \geq 
		\la \mu \varphi_k u , \Phi_\mu'( \mu \varphi_k u ) \psi \ra_{\dot{H}^s (\RN)} .
	\end{equation}

	Next, noting for $k \geq 2$, $\varphi_k u \equiv u$ on $B_2$ and 
	$\Phi_\mu'( \mu \varphi_k u ) \psi \equiv \Phi_\mu'( \mu u ) \psi$ on $\RN$, 
	$|\varphi_k u| \leq |u|$ and 
	$u \in L^1( \RN, (1+|x|)^{-N-2s} d x)$, 
	as in \eqref{eq:3.13}, we may verify that, for any $\psi \in C^\infty_c(B_1)$ with $\psi \geq 0$, 
	\[
		\la v_k , \psi \ra_{\dot{H}^s (\RN)} \to \la v , \psi \ra_{\dot{H}^s (\RN)}, 
		\quad 
		\la \mu \varphi_k u , \Phi_\mu'(\mu \varphi_k u) \psi \ra_{\dot{H}^s (\RN)}
		\to \la \mu u , \Phi_\mu'( \mu u ) \psi \ra_{\dot{H}^s (\RN)}. 
	\]
	Sending $k \to \infty$ in \eqref{eq:3.16} and noting $u \in \ov{\cS}_1$ and 
	$\Phi_\mu'(\mu u) \psi \in \cH^s_0(B_1)$, we see from Proposition~\ref{Proposition:3.3}~(iv) that 
	\[
		\begin{aligned}
			\la v, \psi \ra_{\dot{H}^s (\RN)}
			\geq \la \mu u , \Phi_\mu'(\mu u) \psi \ra_{\dot{H}^s (\RN)}
			\geq \mu \int_{B_1} |x|^\ell u^p \Phi_\mu'(\mu u ) \psi \, d x 
			 & \geq \int_{B_1} |x|^\ell \Phi_\mu^p (\mu u) \psi \, d x
			\\
			 & = \int_{B_1} |x|^\ell v^p \psi \, d x.
		\end{aligned}
	\]

	Finally, since $u \in \ov{\cS}_1$, we have $u_S(x) \leq u(x)$ for all $x \in \RN \setminus B_1(0)$. 
	If $\mu u(x) \geq M_0 $ holds at $x \in \RN \setminus B_1$, 
	then $ \mu u_S(x) \leq M_0 \leq \mu u(x)$, 
	and  Proposition \ref{Proposition:3.3}~(ii) and (iii) yield 
	$v(x) = \Phi_\mu (\mu u(x) ) \geq \Phi_\mu (\mu u_S(x)) = \mu u_S(x)$. 
	On the other hand, if $\mu u(x) \leq M_0$ holds at $x \in \RN \setminus B_1$, then 
	$v(x) = \Phi_\mu (\mu u(x)) = \mu u(x) \geq \mu u_S(x)$. 
	Thus, $v \in \ov{S}_\mu$. 
\end{proof}

\begin{Remark}\label{Remark:3.4}
Since $u_S \in \ov{\cS}_1$, it follows that $\Phi_\mu ( \mu u_S ) \in \ov{\cS}_\mu \cap L^\infty(\RN)$
	for every $\mu \in (0,1)$. 
\end{Remark}

Now, we shall construct a positive solution of \eqref{eq:3.3} (and \eqref{eq:3.1}).

\begin{proposition}\label{Proposition:3.5}
	For each $\mu \in (0,1)$, there exists a positive minimal solution $u_\mu \in H^s(\RN) \cap L^\infty(\RN) $
	of \eqref{eq:3.3}. Furthermore, $u_\mu$ is radial and $u_\mu \leq u$ in $\RN$ holds for 
	any $u \in \ov{\cS}_\mu$. 
\end{proposition}

\begin{proof}
	Let $w_1 \in H^s(\RN)$ be a solution of 
	\begin{equation}\label{eq:3.17}
		(-\Delta)^s w_1 = 0 \quad \text{in} \ B_1 , \quad 
		w_1 = \mu \eta_\mu u_S \quad \text{in} \ \RN \setminus B_1
	\end{equation}
	in the sense that 
	\[
		\la w_1 , \varphi \ra_{\dot{H}^{s} (\RN)} = 0 \quad \text{for all $\varphi \in C^\infty_c(B_1)$}, \quad 
		w_1 = \mu \eta_\mu u_S \quad \text{in} \ \RN \setminus B_1.
	\]
	To see that \eqref{eq:3.17} has a solution, choose $\wt{u}_S \in C^\infty(\RN)$ satisfying 
	\[
		\wt{u}_S  = u_S  \quad \text{in} \ \RN \setminus B_1, \quad 
		\wt{u}_S \geq 0 \quad \text{in} \ B_1. 
	\]
	Define $\wt{u}_{S,\mu}$ by 
	\begin{equation}\label{eq:3.18}
		\wt{u}_{S,\mu} := \mu \eta_\mu \wt{u}_S \in C^\infty_c(\RN). 
	\end{equation}
	Then the equation 
	\[
		\left( -\Delta \right)^s \wt{w}_1 =  \left( -\Delta \right)^s ( \wt{u}_{S,\mu} ) \quad 
		\text{in} \ B_1, \quad \wt{w}_1 \in \cH^s_0(B_1)
	\]
	has a unique solution $\wt{w}_1$,
	where we use an equivalent norm $\| \cdot \|_{\dHs(\RN)}$ on $\cH^s_0(B_1)$ and $\wt{w}_1$
	satisfies 
	\[
		\la \wt{w}_1 , \varphi \ra_{\dot{H}^{s} (\RN)} = \la \wt{u}_{S,\mu} , \varphi \ra_{\dot{H}^s (\RN)} \quad 
		\text{for all $\varphi \in \cH^s_0(B_1)$}. 
	\]
	Hence, $w_1 := \wt{u}_{S,\mu} - \wt{w}_1 \in H^s(\RN)$ is a solution of \eqref{eq:3.17}:
	\[
		\la w_1 , \varphi \ra_{\dot{H}^{s} (\RN)} = 0 \quad \text{for all $\varphi \in \cH^s_0(B_1)$}, \quad 
		w_1 = \wt{u}_{S,\mu} = \mu \eta_{\mu} u_S \quad \text{in} \ \RN \setminus B_1. 
	\]
	
	Notice that for any $u \in \ov{\cS}_\mu$, we have 
	\[
		\la w_1 , \varphi \ra_{\dot{H}^s (\RN)} = 0 \leq \int_{B_1} |x|^\ell u^p \varphi  \, d x 
		\leq \la u , \varphi \ra_{\dot{H}^s (\RN)}
		\quad \text{for all $\varphi \in \cH^s_0(B_1)$ with $\varphi \geq 0$}.
	\]
	Hence, by Proposition \ref{Proposition:3.1}, 
	$0 \leq w_1 \leq u$ in $\RN$ for all $u \in \ov{\cS}_\mu$. 
	In particular, $w_1 \leq \Phi_\mu (\mu u_S)$ in $\RN$ 
	since $\Phi_\mu (\mu u_S) \in \ov{\cS}_\mu$ due to 
	Remark \ref{Remark:3.4}, hence, $w_1 \in L^\infty(\RN)$. 
	In addition, by \eqref{eq:3.17} and Proposition~\ref{Proposition:3.2}, we have $ w_1 > 0$ in $B_1$.

	Next, by $w_1 \in L^\infty(\RN)$, let $w_2 \in H^s(\RN)$ be a solution of 
	\[
		(-\Delta)^s w_2 = |x|^\ell w_1^p \geq 0 = (-\Delta)^s w_1 \quad \text{in} \ B_1, 
		\quad w_2 = \mu \eta_\mu u_S 
		\quad \text{in} \ \RN \setminus B_1.
	\]
	Since $|x|^\ell w_1^p \leq |x|^\ell u^p$ in $B_1$ for each $u \in \ov{\cS}_\mu$, 
	as in the above, we get 
	\[
		w_1 \leq w_2 \leq u \quad \text{in $\RN$ for any $u \in \ov{\cS}_\mu$}, \quad 
		w_2 \leq \Phi_\mu (\mu u_S), \quad w_2 \in L^\infty(\RN).
	\]
	We repeat this procedure and let $w_{n+1} \in H^s(\RN) $ be a solution of 
	\begin{equation}\label{eq:3.19}
		(-\Delta)^s w_{n+1} = |x|^\ell w_n^p \quad \text{in} \ B_1, \quad 
		w_n = \mu \eta_\mu u_S \quad \text{in} \ \RN\setminus B_1.
	\end{equation}
	Then $w_n \leq w_{n+1} \leq u$ in $\RN$ for each $u \in \ov{\cS}_\mu$, 
	$w_{n+1} \leq \Phi_\mu ( \mu u_S)$ in $\RN$ and $w_{n+1} \in L^\infty(\RN)$.

	Next, we show that $(w_n)_n$ is bounded in $H^s(\RN)$. 
	To see this, since $w_{n+1} - \wt{u}_{S,\mu} \in \cH^s_0(B_1)$, 
	it follows that 
	\begin{equation}\label{eq:3.20}
		\begin{aligned}
			\int_{B_1} |x|^\ell w_n^p ( w_{n+1} - \wt{u}_{S,\mu} ) \, d x 
			&= \la w_{n+1} , w_{n+1} - \wt{u}_{S,\mu} \ra_{\dot{H}^s(\RN)}
			\\
			&\geq \|  w_{n+1}  \|_{\dot{H}^s(\RN)}^2 - \| w_{n+1} \|_{\dot{H}^s(\RN)} 
			\| \wt{u}_{S,\mu} \|_{\dot{H}^s(\RN)}.
		\end{aligned}
	\end{equation}
	Since $ w_n \leq \Phi_\mu ( \mu u_S )$ in $\RN$ for each $n \geq 1$ and $\Phi_\mu (\mu u_S) \in L^\infty(\RN)$, 
	from \eqref{eq:3.20}, it follows that $( \| w_n \|_{\dot{H}^s(\RN)} )_n$ is bounded. 
	The boundedness of $( \| w_n \|_{L^2(\RN)} )_n$ follows from 
	$w_n \leq \Phi_\mu (\mu u_S)$ in $\RN$ and $w_n = \mu \eta_\mu u_S$ in $\RN \setminus B_1$.

	Owing to the monotonicity of $w_n$, we set $u_\mu(x) := \lim_{n \to \infty} w_n(x)$ and remark that 
	$w_{n} \rightharpoonup u_\mu$ weakly in $H^s(\RN)$.
	Then it is easily seen from \eqref{eq:3.19} that 
	\[
		\begin{aligned}
			 & u_\mu = \mu \eta_\mu u_S \quad \text{in} \ \RN \setminus B_1, \quad 
			0 < w_1 \leq  u_\mu \leq \min \left\{ u , \Phi_\mu (\mu u_S) \right\} 
			\quad \text{in} \ B_1  \ \text{for every $u \in \ov{\cS}_\mu$},
			\\
			 & \la u_\mu, \varphi \ra_{\dot{H}^s (\RN)} = \int_{B_1} |x|^\ell u_\mu^p \varphi \, d x 
			\quad \text{for any $\varphi \in \cH^s_0(B_1(0))$}.
		\end{aligned}
	\]
	Since $\ov{\cS}_\mu$ contains all nontrivial solutions of \eqref{eq:3.3} by definition, we observe that 
	$u_\mu \in H^s(\RN) \cap L^\infty(\RN)$ is a positive minimal solution of \eqref{eq:3.3}.

	Finally, noting that $u_\mu(R x)$ is also a solution of \eqref{eq:3.3} for all $R \in O(N)$
	due to the radial symmetry of $\eta_{\mu} u_S$, 
	we have $u_{\mu} (x) \leq u_{\mu} (Rx)$ for any $R \in O(N)$, 
	which implies $u_{\mu} (R^{-1} x) \leq u_{\mu} (x) \leq u_{\mu} (Rx)$ for any $R \in O(N)$. 
	Thus $u_{\mu} (x) = u_{\mu} (Rx)$ for every $R \in O(N)$ and 
	$u_\mu$ is radial. 
\end{proof}

	In the following, using $u_{\mu}$ in Proposition \ref{Proposition:3.5} and the stability of $u_{S}$, 
we shall construct stable solutions of \eqref{eq:1.1}. 
To this end, we show the separation property and the convergence result of $(u_{\mu})_{0<\mu<1}$. 

\begin{proposition}\label{Proposition:3.6}
	For each $\mu\in(0,1)$, let $u_\mu$ be a positive minimal solution of \eqref{eq:3.3} in Proposition~{\rm\ref{Proposition:3.5}}.
	Then the following hold:
	\begin{enumerate}
		\item[{\rm (i)}]
		      Let $0< \mu_1 < \mu_2 < 1$. Then 
		      $u_{\mu_1} \le u_{\mu_2}$ in $\RN$;
		\item[{\rm (ii)}]
		      As $\mu \nearrow 1$, $u_{\mu} \rightharpoonup u_S$ weakly in $H^s_{\rm loc} (\RN)$. 
	\end{enumerate}
\end{proposition}

\begin{proof}
	We first prove the separation property.
	Let $0<\mu_1<\mu_2<1$.
	Then it follows from \eqref{eq:3.2} that $\mu_1 \eta_{\mu_1} (x) \le \mu_2 \eta_{\mu_2}(x)$ for all $x \in \RN$. 
	Hence, by \eqref{eq:3.3}, $u_{\mu_1} \le u_{\mu_2}$ in $\RN \setminus B_1$. 
	Therefore,  since $(u_{\mu_1} - u_{\mu_2} )_+ \in \cH^s_0(B_1)$, similarly to \eqref{eq:3.5},
	it holds that
	\[
		\begin{aligned}
			\| ( u_{\mu_1} - u_{\mu_2} )_+ \|_{\dot{H}^s(\RN)}^2 
			&\leq 
			\la u_{\mu_1} - u_{\mu_2} , ( u_{\mu_1} - u_{\mu_2} )_+ \ra_{\dHs(\RN)}
			\\
			&= \int_{B_1} |x|^\ell \left( u_{\mu_1}^p - u_{\mu_2}^p \right)
			\left( u_{\mu_1} - u_{\mu_2} \right)_+  d x.
		\end{aligned}
	\]
	If $(u_{\mu_1} - u_{\mu_2})_+ \not\equiv 0$, then 
	we have 
	\[
		\begin{aligned}
			\int_{B_1} |x|^\ell \left( u_{\mu_1}^p - u_{\mu_2}^p \right)
			\left( u_{\mu_1} - u_{\mu_2} \right)_+ d x 
			 & = \int_{B_1} |x|^\ell p \int_{0}^{1}
			\left( \theta u_{\mu_1} + (1-\theta) u_{\mu_2}  \right)^{p-1} d \theta 
			(u_{\mu_1} - u_{\mu_2})_+^2 \,d x
			\\
			 & < p \int_{B_1} |x|^\ell u_{\mu_1}^{p-1} \left( u_{\mu_1} - u_{\mu_2} \right)^2_+ \,d x.
		\end{aligned}
	\]
	Since $(u_{\mu_1} - u_{\mu_2})_+ \in \cH^s_0(B_1)$ and $C^\infty_c(B_1)$ is dense in $\cH^s_0(B_1)$, 
	we may find a $w \in C^\infty_c(B_1)$ so that 
	\begin{equation}\label{eq:3.21}
		\| w \|_{\dot{H}^s(\RN)}^2 < p \int_{B_1} |x|^\ell u_{\mu_1}^{p-1} w^2 \, d x. 
	\end{equation}
	On the other hand, since $u_\mu \leq \Phi_\mu (\mu u_S) \leq \mu u_S$ with $\mu \in (0,1)$ and $u_S$ is stable, it follows that 
	\[
		\| \varphi \|_{\dot{H}^s(\RN)}^2 - p \int_{\RN} |x|^\ell u_\mu^{p-1} \varphi^2 \, d x 
		\geq 0 \quad \text{for all $\varphi \in C^\infty_c(B_1)$},
	\]
	which contradicts \eqref{eq:3.21} and $ u_{\mu_1} \leq u_{\mu_2}$ in $\RN$.

	Next we prove the convergence property.
	To this end, we first show that $( u_\mu )_{1/2<\mu<1}$ is bounded in $H^s_{\rm loc}(\RN)$. 
	By \eqref{eq:3.2} and \eqref{eq:3.18} it is easily seen that
	\begin{equation}\label{eq:3.22}
		\wt{u}_{S,\mu}  \varphi \to \wt{u}_{S} \varphi \quad \text{in} \ C^2(\RN) 
		\quad \text{for any $\varphi \in C^\infty_c(\RN)$ as $\mu \nearrow 1$}.
	\end{equation}
	Now choose $\varphi_k \in C^\infty_c(\RN)$ with $0 \leq \varphi_k \leq 1$, 
	$\varphi_k(x) \equiv 1$ on $B_k$ and $\varphi_k \equiv 0$ on $\RN \setminus B_{2k}$. 
	Since $\varphi_k ( u_\mu - \wt{u}_{S,\mu} ) \in \cH^s_0(B_1) $ holds 
	with $u_\mu \leq \Phi_\mu (\eta u_S) \leq \mu u_S$ and $\wt{u}_S \ge 0$ in $\RN$, 
	we find that 
	\begin{equation}\label{eq:3.23}
		\la u_\mu , \varphi_k ( u_\mu - \wt{u}_{S,\mu} ) \ra_{\dot{H}^s (\RN)}
		= 
		\int_{B_1} |x|^\ell u_{\mu}^p \varphi_k 
		\left( u_{\mu} - \wt{u}_{S,\mu} \right) d x 
		\leq \int_{B_1} |x|^\ell u_S^{p+1} \, d x < \infty. 
	\end{equation}
	Since $k \geq 1$ and $(1-\varphi_k)  u_\mu = (1-\varphi_k) \wt{u}_{S,\mu}$, 
	we remark that 
	\begin{equation}\label{eq:3.24}
		\begin{aligned}
			\la u_\mu , \varphi_k (u_\mu - \wt{u}_{S,\mu}) \ra_{\dot{H}^s (\RN)}
			 & = \la \varphi_k u_\mu + (1-\varphi_k) u_\mu , 
			\varphi_k u_\mu - \varphi_k \wt{u}_{S,\mu} \ra_{\dot{H}^s (\RN)}
			\\
			 & = \left\| \varphi_k u_{\mu} \right\|_{\dot{H}^s(\RN)}^2 
			- \la \varphi_k u_{\mu} , \varphi_k \wt{u}_{S,\mu} \ra_{\dot{H}^s (\RN)}
			\\
			 & \quad 
			+ \la \left( 1 - \varphi_k \right) \wt{u}_{S,\mu} , \varphi_k u_\mu \ra_{\dot{H}^s (\RN)}
			- \la \left( 1 - \varphi_k \right) \wt{u}_{S,\mu} , \varphi_k \wt{u}_{S,\mu} \ra_{\dot{H}^s (\RN)}.
		\end{aligned}
	\end{equation}
	As in \eqref{eq:3.13}, it follows from \cite[Lemma 2.1]{HIK-20} that 
	\[
		\begin{aligned}
			        & \left| \la \left( 1 - \varphi_k \right) \wt{u}_{S,\mu} , \varphi_k u_\mu \ra \right| 
			\\
			\leq \  & C_k \left( \left\| (1-\varphi_k) \wt{u}_{S,\mu} \right\|_{\dHs(B_{4k})} 
			\left\| \varphi_k u_\mu \right\|_{\dHs (B_{4k})} 
			+ \left\| (1-\varphi_k) \wt{u}_{S,\mu} \right\|_{L^2(B_{4k})} \left\| \varphi_k u_\mu \right\|_{L^2(B_{4k})}
			\right. 
			\\
			        & \qquad \left.
			+ \left\| (1-\varphi_k) \wt{u}_{S,\mu} \right\|_{L^1(\RN, (1+|x|)^{-N-2s} d x ) } 
			\left\| \varphi_k u_\mu \right\|_{L^1(B_{4k})} \right).
		\end{aligned}
	\]
	Remark that a similar estimate holds for 
	$\la (1-\varphi_k) \wt{u}_{S,\mu} , \varphi_k \wt{u}_{S,\mu} \ra_{\dot{H}^s (\RN)}$, 
	$( (1 - \varphi_k ) \wt{u}_{S,\mu} )_{1/2 < \mu < 1}$ is bounded in 
	$H^s(B_{4k}) \cap L^1(\RN, (1+|x|)^{-N-2s} d x )$, 
	and $(\varphi_k \wt{u}_{S,\mu}  )_{1/2<\mu<1}$ is bounded in $H^s(\RN)$. 
	Since $\| \varphi_k u_\mu \|_{L^{2^\ast_s} (\RN)} \le C \| \varphi_k u_\mu \|_{\dHs(\RN)}$
	and $\supp \varphi_k \subset B_{4k}$, 
	it follows from \eqref{eq:3.22}, \eqref{eq:3.23}, \eqref{eq:3.24} and $u_\mu \leq u_S$ that for any $\mu \in (1/2,1)$, 
	\[
		\begin{aligned}
			\int_{B_1} |x|^\ell u_S^{p+1} \, d x 
			 & \geq 
			\| \varphi_{k} u_\mu \|_{\dHs (\RN)}^2 
			- \| \varphi_k u_\mu \|_{\dHs(\RN)} \| \varphi_k \wt{u}_{S,\mu} \|_{\dHs (\RN)} 
			\\
			 & \quad 
			- C_k \left( \| \varphi_k u_\mu \|_{\dHs (\RN)} + \| \varphi_k u_\mu \|_{L^2(B_{4k})} 
			+ \| \varphi_k u_\mu \|_{L^1(B_{4k})} + 1 \right)
			\\
			 & \geq \| \varphi_k u_\mu \|_{\dHs(\RN)}^2 
			- C_k \left( \| \varphi_k u_\mu \|_{\dHs (\RN)} + \| \varphi_k u_\mu \|_{L^{2^\ast_s} (\RN) } + 1 \right)
			\\
			 & \geq 
			\| \varphi_k u_\mu \|_{\dHs(\RN)}^2 - C_k \left( \| \varphi_k u_\mu \|_{\dHs (\RN)} + 1 \right) .
		\end{aligned}
	\]
	Thus, $( \| \varphi_k u_\mu \|_{\dot{H}^s(\RN)})_{1/2<\mu<1}$ is bounded for each $k$.

	Let $u_1 (x) := \lim_{\mu \nearrow 1} u_{\mu} (x) \leq u_S(x)$. 
	Then, by Proposition~\ref{Proposition:2.1} (ii), $u_1 \in L^2_{\rm loc} (\RN)$.
	Furthermore, from the boundedness of $(\|\varphi_k u_\mu\|_{\dot{H}^s(\RN)})_{1/2<\mu<1}$, 
	it follows that $\varphi_ku_1 \in H^s(\RN)$ and 
	\begin{equation}\label{eq:3.25}
	\varphi_ku_{\mu} \rightharpoonup \varphi_ku_1\quad \text{weakly in $H^s (\RN)$ as $\mu\nearrow 1$}
	\end{equation}
	for each $k$.
	Next, we shall show that for each $\varphi \in C^\infty_c(B_1)$, 
	\begin{equation}\label{eq:3.26}
		\int_{B_1} |x|^\ell u_1^p \varphi \, d x 
		= \lim_{\mu \nearrow 1} \la u_\mu , \varphi \ra_{\dot{H}^s(\RN)}
		= \la u_1 , \varphi \ra_{\dot{H}^s (\RN)}.
	\end{equation}
	To this end, notice that $u_\mu \leq u_S$ on $\RN$ and $(1-\varphi_k) (u_\mu - u_1) = 0$ on $B_2$ for $k \geq 2$. 
	In particular, $u_\mu \to u_1$ in $L^1( \RN, (1+|x|)^{-N-2s}  d x)$. 
	Thus, as in \eqref{eq:3.13}, we infer from \cite[Lemma 2.1]{HIK-20},
	\eqref{eq:3.25} and $(1-\varphi_k) u_\mu \equiv 0 \equiv (1- \varphi_k) u_1$ on $B_2$ 
	that for each $k \geq 2$, as $\mu \to 1$,
	\[
		\la \varphi_k u_\mu , \varphi \ra_{\dHs(\RN)} \to \la \varphi_k u_1 , \varphi \ra_{\dHs(\RN)}, \quad 
		\la (1-\varphi_{k}) u_\mu , \varphi \ra_{\dHs(\RN)} \to \la (1-\varphi_{k}) u_1 , \varphi \ra_{\dHs(\RN)}.
	\]
	Therefore, 
	\[
		\begin{aligned}
			\int_{B_1} |x|^\ell u_1^p \varphi \, d x 
			= \lim_{\mu \nearrow 1} \int_{B_1} |x|^\ell u_\mu^p \varphi \, d x 
			&= \lim_{\mu \nearrow 1} \la u_\mu, \varphi \ra_{\dHs(\RN)}
			\\
			& = 
			\lim_{\mu \nearrow 1} \left\{  \la \varphi_k u_\mu + (1-\varphi_k) u_\mu , \varphi \ra_{\dHs(\RN)} \right\}
			= 
			\la u_1 , \varphi \ra_{\dHs(\RN)},
		\end{aligned}
	\]
	and \eqref{eq:3.26} holds. 
	
	Since $|x|^\ell u_S^p \in L^{2N/(N+2s)} (B_1)$, $u_1$ is a solution of \eqref{eq:3.1} with $\mu = 1$ and $u_1 \leq u_S$. 
	Noting $u_S - u_1 \in \cH^s_0(B_1)$, we have 
	\[
		\| u_S - u_1 \|_{\dot{H}^s(\RN)}^2 
		= \la u_S- u_1 , u_S - u_1 \ra_{\dot{H}^{s} (\RN)}
		= \int_{B_1} |x|^\ell \left( u_S^p - u_1^p \right) \left( u_S -u_1 \right) d x.
	\]
	If $u_S - u_1 \not\equiv 0$, then 
	\[
		\begin{aligned}
			\| u_S - u_1 \|_{\dot{H}^s(\RN)}^2 
			 &=  \int_{B_1} |x|^\ell 
			p \int_{0}^{1} \left( \theta u_S + (1-\theta) u_1 \right)^{p-1} \, d \theta 
			\left(u_S-u_1\right)^2 d x
			\\
			 & < p \int_{B_1} |x|^\ell u_S^{p-1} \left( u_S-u_1 \right)^2 d x,
		\end{aligned}
	\]
	which contradicts \eqref{eq:2.7}.
	Thus, $u_S \equiv u_1$. 
\end{proof}

\begin{remark}\label{Remark:3.5}
	Since $u_\mu \in L^\infty(\RN)$, we may prove $ u_\mu \in C (B_1)$ (see \cite[Proposition 2.1]{HIK-20}).
\end{remark}

Now, we shall construct a stable solution of \eqref{eq:1.1}.

\begin{proposition}\label{Proposition:3.7}
	There exists a positive radial stable solution $u$ of 
	\begin{equation}\label{eq:3.27}
		\left( - \Delta \right)^s u = |x|^\ell u^p \quad \text{in} \ \RN, 
		\quad \| u \|_{L^\infty(\RN)} = 1, \quad 
		0< u \leq u_S \quad \text{in} \ \RN.
	\end{equation}
\end{proposition}

\begin{proof}
	For each $\mu\in(0,1)$, let $u_\mu$ be as in Proposition \ref{Proposition:3.5}.
	Let $(\mu_j)_j\subset(0,1)$ with $\mu_j \nearrow 1$ as $j\to\infty$,
	and	set 
	\[
		\wt{u}_j(x) := \frac{1}{m_j} u_{\mu_j} \left( m_{j}^{-1/\theta_0} x \right), \quad 
		m_j := \| u_{\mu_j} \|_{L^\infty(\RN)},
	\]
	where $\theta_0$ is as in \eqref{eq:1.8}.
	By Propositions \ref{Proposition:3.5} and \ref{Proposition:3.6}, 
	Remark \ref{Remark:3.5} and $u_\mu \leq u_S$ in $B_1$, we have $m_j \to \infty$ as $j\to\infty$
	and 
	\begin{equation}\label{eq:3.28}
		\begin{aligned}
			 & \
			\la \wt{u}_j , \varphi \ra_{\dot{H}^s (\RN)}
			= \int_{\RN} |x|^\ell \wt{u}_j^p  \varphi \, d x
			\quad \text{for all $\varphi \in C^\infty_c ( B_{m_j^{1/\theta_{0}}} )$},
			\\
			 & \wt{u}_j(x) \leq \frac{1}{m_j} u_S \left( m_j^{-1/\theta_{0}} x \right) = u_S(x) \quad 
			\text{for each $x \in \RN \setminus \{0\}$}, \quad 
			\| \wt{u}_j \|_{L^\infty(\RN)} = 1.
		\end{aligned}
	\end{equation}

	Next, we claim that $(\varphi \wt{u}_j)_j$ is bounded in $H^s(\RN)$
	for any $\varphi \in C^\infty_c(\RN)$. 
	To this end, we set $\wt{U}_j (x,t) := (P_s(\cdot, t) \ast \wt{u}_j) (x)$ 
	where $P_s(x,t) := p_{N,s} t^{2s} (|x|^2+t^2)^{ - \frac{N+2s}{2} }$. 
	As in \cite[Lemmata 2.1 and 2.2]{HIK-20}, from $\wt{u}_j \in \Hsloc (\RN) \cap L^\infty(\RN) $, 
	we have  
	\[
		\begin{aligned}
			&\wt{U}_j \in \Hloc( \ov{\HS}, t^{1-2s} d X ), \quad 
			- \diver \left( t^{1-2s} \nabla \wt{U}_j \right) = 0 \quad \text{in} \ \HS, \quad 
			\\
			&
			\lim_{t \to 0} - \int_{ \RN } t^{1-2s} \partial_t \wt{U}_j (x,t) \psi(x,t) \, d x 
			= \kappa_s \la \wt{u}_j , \psi (\cdot , 0) \ra_{\dot{H}^s (\RN)} \quad 
			\text{for each $\psi \in C^\infty_c( \ov{\HS} )$}.
		\end{aligned}
	\]
	In particular, if $\psi(x,t) \in C^\infty_c (\ov{\HS})$ with 
	$\supp \psi(\cdot, 0) \subset B_{m_j^{1/\theta_{0}}} $, then we may choose 
	$\psi^2 (\cdot, 0) \wt{u}_j$ as a test function in \eqref{eq:3.28}. Thus,  
	\begin{equation}\label{eq:3.29}
		\int_{\HS} t^{1-2s} \nabla \wt{U}_j \cdot \nabla \left( \psi^2 \wt{U}_j \right) d X 
		= \kappa_s \la \wt{u}_j , \psi^2(\cdot, 0) \wt{u}_j \ra_{\dot{H}^s(\RN)}
		= \kappa_s \int_{B_{m_j^{1/\theta_{0}}}  } 
		|x|^\ell \wt{u}_j^{p+1} \psi^2(x,0) \, d x.
	\end{equation}
	Furthermore, it follows from $\| P_s(\cdot, t) \|_{L^1(\RN)}=1$ for all $t>0$ and \eqref{eq:3.28} that
	\[
		\| \wt{U}_j \|_{L^\infty(\HS)} \leq \sup_{ t > 0}
		\| P_s(\cdot, t) \|_{L^1(\RN)} \| \wt{u}_j \|_{L^\infty(\RN)} = 1.
	\]
	Since
	\begin{equation}\label{eq:3.30}
		\begin{aligned}
			 &\nabla \wt{U}_j \cdot \nabla \left( \psi^2 \wt{U}_j \right)
			 \\
			 \geq &  \left| \nabla \wt{U}_j \right|^2 \psi^2  
			- 2 \left| \wt{U}_j \psi \right| 
			\left| \nabla \wt{U}_j \right| \left| \nabla \psi \right| 
			\\
			 \geq & \frac{1}{2} \left| \nabla \wt{U}_j \right|^2  \psi^2 
			- 2 \wt{U}_j^2 |\nabla \psi|^2 
			\\
			 = & \frac{1}{4} \left|\nabla \wt{U}_j \right|^2 \psi^2 
			+ \frac{1}{4} \left\{  \left| \nabla \left( \psi \wt{U}_j \right) \right|^2 
			- 2 \psi \wt{U}_j \nabla \psi \cdot \nabla \wt{U}_j 
			- \wt{U}_j^2 |\nabla \psi|^2 \right\} 
			-2 \wt{U}_j^2 |\nabla \psi |^2 
			\\
			 \geq & \frac{1}{4} \left| \nabla \left(\psi \wt{U}_j \right) \right|^2 - 3 \wt{U}_j^2 |\nabla \psi|^2,
		\end{aligned}
	\end{equation}
	we see from \eqref{eq:3.29}  that $(\psi \wt{U}_j)_j$ is bounded in $H^1(\HS, t^{1-2s} d X )$. 
	By this boundedness, the trace theorem $H^1(\HS , t^{1-2s} dX) \subset H^s(\RN)$ 
	and $\wt{U}_j(x,0) = \wt{u}_j(x)$, 
	for any $\varphi \in C^\infty_c(\RN)$, $(\varphi \wt{u}_j)_j$ is bounded in $H^s(\RN)$.

	Now let $\wt{u}_j \rightharpoonup \wt{u}_0$ weakly in $H^s(B_R)$ as $j\to\infty$ for any $R>0$.
	By \eqref{eq:3.29}, $\| \wt{u}_j \|_{L^\infty(\RN)} = 1$, $|x|^\ell \in L^q_{\rm loc} (\RN)$
	for some $q < \frac{N}{2s}$ and \cite[Proposition 2.1]{HIK-20} (see \cite{JLX-14}), 
	we observe that $\wt{u}_j \to \wt{u}_0$ in $C^\beta_{\rm loc} (\RN)$
	for some $\beta \in (0,1)$. 
	Then $\| \wt{u}_0 \|_{L^\infty(\RN)} = 1$ holds due to $\wt{u}_j \leq u_S$ in $\RN\setminus\{0\}$. Moreover, 
	for $\varphi \in C^\infty_c(\RN)$, as in \eqref{eq:3.26}, it follows that 
	\[
		\la \wt{u}_0 , \varphi \ra_{\dot{H}^s(\RN)} = \lim_{j \to \infty} \la \wt{u}_j , \varphi \ra_{\dot{H}^s(\RN)}
		= \lim_{j \to \infty} \int_{\RN} |x|^\ell \wt{u}_j^p \varphi \, d x 
		= \int_{\RN} |x|^\ell \wt{u}_0^p \varphi \, d x.
	\]
	Hence, $\wt{u}_0$ is a solution of \eqref{eq:3.27}. 
	In addition, from 
	\[
		\la \wt{u}_0 , \varphi \ra_{\dot{H}^s(\RN)} = \int_{ \RN } |x|^\ell \wt{u}_0^p \varphi \, d x 
		\geq 0 = \int_{ \RN } 0 \varphi \, d x \quad 
		\text{for each $\varphi \in C_c^\infty(\RN)$ with $\varphi \geq 0$}
	\]
	and Proposition \ref{Proposition:3.2}, we observe that $\wt{u}_0 > 0$ in $B_R$ for each $R>0$, hence, 
	$\wt{u}_0 > 0$ in $\RN$. 
	The stability follows from $\wt{u}_0 \leq u_S$ in $\RN$ and \eqref{eq:2.7}.
\end{proof}

As a corollary of Proposition \ref{Proposition:3.7}, we have

\begin{corollary}\label{Corollary:3.1}
	Let $\wt{u}_0$ be a bounded radial stable solution of \eqref{eq:3.27}
	in Proposition {\rm\ref{Proposition:3.7}}. For $m>0$, set 
	\begin{equation}\label{eq:3.31}
		u_{m,0} (x) := m \wt{u}_{0} \left( m^{1/\theta_0} x \right),
	\end{equation}
	where $\theta_0$ is as in \eqref{eq:1.8}.
	Then $u_{m,0}$ satisfies 
	\begin{equation}\label{eq:3.32}
		(-\Delta)^s u_{m,0} = |x|^\ell u_{m,0}^p \quad \text{in} \ \RN, \quad 
		\| u_{m,0} \|_{L^\infty(\RN)} = m  , \quad 0< u_{m,0} < u_S \quad \text{in} \ \RN
	\end{equation}
	and $u_{m,0}$ is stable. 
\end{corollary}

\begin{proof}
	It is easy to check that $u_{m,0}$ satisfies $(-\Delta)^s u_{m,0} = |x|^\ell u_{m,0}^p$ in $\RN$ and 
	$\| u_{m,0} \|_{L^\infty(\RN)} = m$. In addition, by
	$u_{m,0} (x) = m \wt{u}_0( m^{1/\theta_0}x ) \leq m u_S ( m^{1/\theta_0}x ) = u_S (x)$ and 
	\eqref{eq:2.7}, we infer that $u_{m,0}$ is stable. 
	Finally, from $u_{m,0} \leq u_S$ in $\RN$ and $\la u_S - u_{m,0} , \varphi \ra_{\dHs (\RN)} \geq 0$ 
	for each $\varphi \in C^\infty_c(\RN)$ with $\varphi \geq 0$, 
	Proposition \ref{Proposition:3.2} yields $u_{m,0} < u_S$ in $\RN$. 
\end{proof}

	\begin{remark}\label{Remark:3.6}
	By the regularity result it holds that $\wt{u}_0 \in C(\RN)$.
	Furthermore, by \eqref{eq:3.31} and \eqref{eq:3.32} we see that
	$\|u_{m,0} \|_{L^\infty(\RN)} \wt{u}_0(0) =  m \wt{u}_0(0)  = u_{m,0} (0)$.
	These together with $\wt{u}_0 > 0$ in $\RN$ imply that 
	$u_{m_1,0} (0) < u_{m_2,0} (0)$ is equivalent to $ m_1 < m_2$ and 
	instead of $(u_{m,0})_{m>0}$, we may write $(u_\alpha)_{\alpha>0}$ where $\alpha = u_\alpha(0) = m \wt{u}_0(0)$. 
	\end{remark}


\section{Separation property and proof of Theorem \ref{Theorem:1.1} (ii)}\label{section:4}


	In this section, we aim to prove the comparison of positive radial stable solutions of \eqref{eq:1.1}
and study the properties of $(u_{m,0})_{m>0}$ in Corollary \ref{Corollary:3.1}.


\subsection{Separation property}


The aim of this subsection is to prove the following comparison result which is an extension of \cite[Proposition 7.1]{H}:

\begin{proposition}\label{Proposition:4.1}
	Let $u_1,u_2 \in \Hsloc (\RN) \cap L^1(\RN, (1+|x|)^{-N-2s} d x )$
	be positive radial stable solutions of 
	\begin{equation}\label{eq:4.1}
		(-\Delta)^s u = |x|^\ell u^p \quad \text{in} \ \RN.
	\end{equation}
	In addition, suppose that 
	\begin{equation}\label{eq:4.2}
		\begin{aligned}
			 & u_1 \in L^\infty(\RN), 
			\quad u_2 \in L^\infty_{\rm loc} \left( \RN \setminus \set{0} \right), \quad 
			u_2 \leq C u_S \quad \text{in $B_1$},
			\\
			&u_1(x),u_2 (x) \to 0 \ \text{as} \ |x| \to \infty,
			\quad  
			u_1(0) < \lim_{|x| \to 0} u_2(x),
		\end{aligned}
	\end{equation}
for some $C>0$. Then $u_1 < u_2$ in $\RN$. 
\end{proposition}

\begin{remark}\label{Remark:4.1}
	\begin{enumerate}
		\item[(i)]
		      In Proposition \ref{Proposition:4.1}, the case $u_2(x) \to \infty$ as $|x| \to 0$ may occur, however, 
		      by $u_2 \leq C u_S$ in $B_1$ and \eqref{eq:4.2}, 
		      we have $|x|^\ell u_2^p \in L^{ \frac{2N}{N+2s} }_{\rm loc} (\RN)$. 
		\item[(ii)] 
			  From $u_2 \leq C u_S$ in $B_1$, $u_2 \in L^\infty_{\rm loc} (\RN \setminus \set{0})$ and 
			  the Hardy type inequality \eqref{eq:2.8}, we see that the functional 
			  	\[
			  		\cH^s_0(B_R) \ni w \mapsto \int_{  \RN} |x|^\ell u_2^{p-1} w^2 \, dx \in \R
			  	\]
			  is bounded for each $R>0$, where 
			  $\cH^s_0(B_R) := \Set{ w \in H^s(\RN) | w \equiv 0 \ \text{on} \ \RN \setminus B_R }$. 
			  Since we assume that $u_2$ is stable, it follows that 
			  	\[
			  		p \int_{  \RN} |x|^\ell u_2^{p-1} w^2 \, dx \leq \| w \|_{\dHs(\RN)}^2 
			  	\]
			  for all $w \in H^s(\RN)$, which have compact support in $\RN$. 
		\item[(iii)]
		      From \eqref{eq:4.2} and the fact that $u_1$ and $u_2$ are solutions of \eqref{eq:4.1}, 
		      we have $u_1 \in C(\RN)$ and $u_2 \in C( \RN \setminus \set{0} )$ as in Remark \ref{Remark:3.5}. 
		      If $\lim_{|x| \to 0} u_2(x) < \infty$, then $u_2 \in C(\RN)$. 
		\item[(iv)] 
			  Let $u \in \Hsloc (\RN) \cap L^1( \RN, (1+|x|)^{-N-2s} d x ) $ 
			  be a positive solution of \eqref{eq:4.1}. 
			  Since we require $u$ to satisfy $|x|^\ell u^p \in L^{ \frac{2N}{N+2s} } (B_1) $, 
			  \cite[Lemma 2.2]{HIK-20} still holds 
			  for $U(x,t) := (P_s(\cdot, t) \ast u ) (x)$, 
			  that is, 
			  $U \in \Hloc( \ov{\HS} , t^{1-2s} d X )$ satisfies 
			  	\[
			  		\int_{\HS} t^{1-2s} \nabla U \cdot \nabla  \psi \, d X
			  		= \kappa_s \la u , \psi (\cdot, 0) \ra_{\dHs(\RN)} 
			  		= \kappa_s \int_{  \RN} |x|^\ell u^p \psi (x,0) \, dx
			  	\]
			  for every $\psi \in C^1_c( \ov{\HS} )$. 
	\end{enumerate}
\end{remark}

We first remark that if $u_{1} \leq u_{2}$ in $\RN$ holds, then 
for each $\varphi \in C^\infty_c(\RN)$ with $\varphi \geq 0$, 
\[
	\la u_2 - u_1 , \varphi \ra_{\dot{H}^s (\RN)}
	= \int_{ \RN } |x|^\ell \left( u_2^p - u_1^p \right) \varphi \, d x 
	\geq 0 = \int_{ \RN } 0 \varphi \, d x. 
\]
Since $u_i \in H^s_{\rm loc}(\RN)\cap L^1(\RN , (1+|x|)^{-N-2s} dx )$ by \eqref{eq:4.2}, 
Proposition \ref{Proposition:3.2} implies that $u_2 - u_1 > 0$ holds in $\RN$ if $u_2 \in L^\infty(B_1)$ and 
in $\RN \setminus \set{0}$ if $u_2(x) \to \infty$ as $|x| \to \infty$. 
By \eqref{eq:4.2}, in both cases, we have $u_1 < u_2$ in $\RN$ and Proposition \ref{Proposition:4.1} holds. 
Therefore, in what follows, it suffices to prove 
\begin{equation}\label{eq:4.3}
	u_1 \leq u_2  \quad \text{in} \ \RN.
\end{equation}
The argument below is inspired by \cite[Section 7]{H}.

To show \eqref{eq:4.3}, we set 
\[
	\begin{aligned}
		 & \left[ u_{1} - u_{2} > 0 \right] := \Set{ x \in \RN |
			u_{1} (x) - u_{2} (x) > 0 }.
	\end{aligned}
\]
We will prove Proposition \ref{Proposition:4.1} indirectly and hereafter, suppose 
\begin{equation}\label{eq:4.4}
	[u_1 - u_2 > 0] \neq \emptyset.
\end{equation}
We aim to derive a contradiction.

\begin{lemma}\label{Lemma:4.1}
	Under \eqref{eq:4.4}, both of $[ u_{1} - u_{2} > 0]$ and 
	$[ u_{2} - u_{1} > 0]$ are unbounded and 
	have infinitely many components. 
\end{lemma}

\begin{proof}
	Notice that 
	both of $[ u_{1} - u_{2} > 0]$ and $[ u_{2} - u_{1} > 0]$ are nonempty open sets in $\RN$
	due to Remark \ref{Remark:4.1}~(iii), \eqref{eq:4.2} and \eqref{eq:4.4}. 
	If $[ u_{1} - u_{2} > 0]$ is bounded, then 
	\[
		v(x) := \left( u_{1} (x) - u_{2}(x) \right)_+ \in H^s(\RN), \quad 
		\text{$\supp v$ is compact}, \quad v \not \equiv 0.
	\]
	Thus, due to Remark \ref{Remark:3.1}~(iv),
	we can take $v$ as a test function, and 
	\[
		\la u_{1} - u_{2} ,v \ra_{\dot{H}^s (\RN)}
		= \int_{ \RN } |x|^\ell (u_{1}^p - u_{2}^p) v\, d x .
	\]
	Since $(a_+-b_+)^2 \leq (a-b)(a_+-b_+)$ holds, we have 
	\[
		\| v \|_{\dot{H}^s(\RN)}^2 \leq \la u_{1} - u_{2} ,v \ra_{\dot{H}^s (\RN)} .
	\]
	Moreover, from $v \not \equiv 0$ and 
	\[
		\begin{aligned}
			\int_{ \RN } |x|^\ell (u_{1}^p - u_{2}^p) v \, d x 
			 & = \int_{ \RN } |x|^\ell p \int_{0}^{1} \left( \theta u_{1} + (1-\theta) u_{2} \right)^{p-1} \, d \theta 
			v^2 \, d x
			\\
			 & < p \int_{ \RN } |x|^\ell u_{1}^{p-1} v^2 \, d x,
		\end{aligned}
	\]
	it follows that 
	\[
		\| v \|_{\dot{H}^s(\RN)}^2 < p \int_{ \RN } |x|^\ell u_{1}^{p-1} v^2 \, d x,
	\]
	and this contradicts the stability of $u_{1}$.

	On the other hand, when $[ u_{2} - u_{1} > 0]$ is bounded, we exploit 
	\[
		w(x) := \left( u_{2}(x) - u_{1}(x) \right)_+ 
	\]
	and test this function to $(-\Delta)^s (u_{2} - u_{1}) = |x|^\ell (u_{2}^p - u_{1}^p)$. 
	Then we have a contradiction.

	Finally, suppose that either $ [ u_{1} - u_{2} > 0]$ or $[ u_2 - u_{1} > 0]$
	has only finitely many components. 
	Here we first assume that $[ u_{1} - u_{2} > 0]$ admits only finitely many components 
	$I_1,\ldots, I_k$. We know that $[ u_{1} - u_{2} > 0]$ is unbounded and 
	each $I_j$ is either ball or annulus due to the radial symmetry of $u_{1}$ and $u_2$. 
	Therefore, one of $I_1,\ldots, I_k$ should be unbounded, say $I_1$. 
	Since $I_1$ is unbounded and either ball or annulus, we may find an $r_0>0$ such that 
	$\RN \setminus B_{r_0} \subset I_1$. 
	This means $ u_{1} (x) - u_{2} (x) > 0$
	for all $|x| \geq r_0$ and 
	$[ u_{2} - u_{1} > 0]$ becomes bounded. 
	However, this is a contradiction. 
	In a similar way, we may prove that the case $[ u_{2} - u_{1} > 0]$
	has only finitely many components never occurs. 
\end{proof}

\begin{remark}\label{Remark:4.2}
	Let $I_{k,+}$ (resp. $I_{k,-}$) be a component of $[ u_{2} - u_{1} > 0 ]$
	(resp. $[ u_{1} - u_{2} > 0 ]$ ).  Thanks to \eqref{eq:4.2}, we may suppose $0 \in I_{1,+}$. 
	Furthermore, from the above argument, we may verify that each $I_{k,\pm}$ is bounded. 
\end{remark}

Let $u_i$ be as in Proposition~\ref{Proposition:4.1}.
For $u_i$, set 
\begin{equation}\label{eq:4.5}
	U_{i} (X) := \left( P_s (\cdot,t) \ast u_{i} \right) (x). 
\end{equation}
Then $U_i(X) \to 0$ as $|X| \to \infty$ thanks to \eqref{eq:4.2} and Lemma \ref{Lemma:B.2}. 
Moreover, noting Remark \ref{Remark:4.1}, we may prove 
\begin{equation}\label{eq:4.6}
	\begin{aligned}
		 & U_{1} \in H^1_{\rm loc} \left( \ov{\HS} , t^{1-2s} d X \right)
		\cap C \left( \ov{\HS} \right) \cap C^\infty \left( \HS \right) ,
		\\
		 & U_2 \in H^1_{\rm loc} \left( \ov{\HS} , t^{1-2s} d X  \right)
		\cap C \left( \ov{\HS} \setminus \set{0} \right) \cap C^\infty \left( \HS \right)
	\end{aligned}
\end{equation}
and for every $\varphi \in C^1_c( \ov{\HS} )$,
\begin{equation}\label{eq:4.7}
	\int_{\HS} t^{1-2s} \nabla U_i \cdot \nabla \varphi \, d X 
	= \kappa_s \la u_{i} , \varphi (\cdot , 0) \ra_{\dot{H}^s (\RN)}
	= \kappa_s \int_{\RN} |x|^\ell u_{i}^p \varphi (x,0) \, d x.
\end{equation}
For $U_i$, we also use the following notation:  
\begin{equation}\label{eq:4.8}
	\begin{aligned}
		V(X) &:= U_2(X) - U_1(X) \in \Hloc \left( \ov{\HS} , t^{1-2s} d X  \right) 
		\cap C \left( \ov{\HS} \setminus \set{0} \right) \cap C^\infty \left( \HS \right), 
		\\
		\left[ V > a \right] &:= \Set{ X \in \HS | V(X) > a }, 
		\quad 
		\left[ V < a \right] := \Set{ X \in \HS | V(X) < a }. 
	\end{aligned}
\end{equation}
	Then, by \eqref{eq:4.2}, \eqref{eq:4.4}, \eqref{eq:4.5}, \eqref{eq:4.6} and \eqref{eq:4.8} 
	we see that $[V > 0] \neq \emptyset$ and $[V < 0] \neq \emptyset$. 
	When $u_2 \in L^\infty(\RN)$, $V \in C( \ov{\HS} )$ holds due to Remark \ref{Remark:4.1}. 
	In addition, by virtue of Lemma \ref{Lemma:B.2}, 
	\begin{equation}\label{eq:4.9}
		\lim_{|X| \to 0} V(X) = \infty \quad \text{when $u_2(x) \to \infty$ as $|x| \to 0$}.
	\end{equation}
Moreover, for each component $\cO_+ \subset \HS$ of $[V > 0]$ and $\cO_- \subset \HS$ of $[V<0]$, define 
\begin{equation}\label{eq:4.10}
	V_{\cO_+ }(X) := V(X) \chi_{\ov{\cO_+}} (X), \quad 
	V_{\cO_- }(X) := V(X) \chi_{\ov{\cO_-}} (X),
\end{equation}
where $\ov{\cO_\pm} \subset \ov{\HS}$ is the closure of $\cO_\pm$ in $\R^{N+1}$ 
and $\chi_A$ is the characteristic function of $A$. 

\begin{lemma}\label{Lemma:4.2}
For each any component $\cO_\pm \subset\HS$ of $[V>0]$ and $[V<0]$,
let $V_{\cO_\pm}$ be as in \eqref{eq:4.10}. 
Then $V_{\cO_\pm} \in C( \ov{\HS} \setminus \set{0} ) \cap H^1_{\rm loc} (\ov{\HS} , t^{1-2s} d X )$, 
and $V_{\cO_\pm}$ satisfy
	\begin{equation}\label{eq:4.11}
		\begin{aligned}
			& \int_{\HS} t^{1-2s} \left| \nabla V_{\cO_\pm} \right|^2 \psi  \, d X 
			\\
			= \  & -\int_{\HS} t^{1-2s} \nabla \left( \frac{V^2_{\cO_\pm}}{2} \right) \cdot \nabla \psi \, d X 
			+ \kappa_s \int_{ \RN } |x|^\ell \left( u_{2}^p - u_{1}^p \right) V_{\cO_\pm}(x,0) \psi(x,0) \, d x
		\end{aligned}
	\end{equation}
	for each $\psi \in C^\infty_c( \ov{\HS} )$. 
	Furthermore, if $u_2 \in L^\infty (\RN)$, then $V_{\cO_\pm} \in C( \ov{\HS} )$.  
\end{lemma}

\begin{proof}
	We only treat $V_{\cO_+}$ and a similar argument works for $V_{\cO_-}$. 
	We first show $V_{\cO_+} \in C ( \ov{\HS} \setminus \set{0}  )$. 
	Since $V = 0$ on $\partial \cO_+ \cap \HS$ and $V \in C( \ov{\HS} \setminus \set{0} )$, 
	it suffices to show the continuity on $(\RN \setminus \set{0}) \times \{0\}$. 
	Let $x \in \RN \setminus \set{0}$. 
	If $V(x,0) = 0$, then the continuity of $V_{\cO_+}$ at $(x,0)$ follows from 
	the continuity of $V(x,t)$. 
	If $V(x,0) \neq 0$, then since $V \in C(\ov{\HS} \setminus \set{0} )$, for some $r_0>0$, 
	\[
		V(y,t) \neq 0 \quad \text{for every } (y,t) \in \ov{B_{r_0}^+(x,0)}
		:= \Set{ (x,t) \in \ov{\HS} | \ |x-y|+t \leq r_0 }. 
	\]
	Thus, if $V(x,0) < 0$, then $V_{\cO_+} (y,t) = 0$ for every $(y,t) \in \ov{B_{r_0}^+(x,0)}$
	and $V_{\cO_+}$ is continuous at $(x,0)$. 
	If $V(x,0) > 0$, then $\ov{B_{r_0}^+(x,0)} \cap \HS \subset \wt{\cO}$ where 
	$\wt{\cO}$ is some component of $[V>0]$. When $\cO_+ \neq \wt{\cO}$, 
	$V_{\cO_+} (y,t) = 0$ for all $(y,t) \in \ov{B_{r_0}^+(x,0)}$ and 
	$V_{\cO_+}$ is continuous at $(x,0)$. 
	On the other hand, when $\cO_+ = \wt{\cO}$, we have 
	$V_{\cO_+} (y,t) = V(y,t)$ for all $(y,t) \in \ov{B_{r_0}^+(x,0)}$. 
	Hence, $V_{\cO_+}$ is continuous on $\ov{\HS} \setminus \{0\} $. 
	In addition, when $u_2 \in L^\infty(\RN)$, we have $V \in C( \ov{\HS}  )$ and 
	from the above argument, it is easily seen that $V_{\cO_+} \in C( \ov{\HS} )$.

	Next, we prove $V_{\cO_+} \in H^1_{\rm loc} (\ov{\HS}, t^{1-2s} d X )$. 
	We argue as in \cite[Theorem 9.17]{Br11}. We choose $\zeta_0 \in C^\infty(\R)$ such that 
	\begin{equation}\label{eq:4.12}
		\zeta_0 (\tau) = 0 \quad \text{if} \ \tau \leq 1, \quad 
		0 \leq \zeta_0 (\tau) \leq \tau \quad \text{for every $\tau \in [0,\infty)$}, 
		\quad 
		\zeta_0 (\tau) = \tau \quad \text{if} \ \tau \geq 2,
	\end{equation}
	and define 
	\[
		W_n(X) := \frac{1}{n} \zeta_0\left( n V_{\cO_+}(X) \right) \psi (X)
		\quad \text{where $\psi \in C^\infty_c ( \ov{\HS} )$.}
	\]
	Remark that 
	\[
		\begin{aligned}
			&\text{$ \supp W_n$ is compact in $\ov{\HS}$}, \quad 
			V_{\cO_+} = 0 \quad \text{on}  \ \partial \cO_+ \cap \HS, 
			\\
			&V(x,0) = u_2(x) - u_1(x) >0 \quad 
			\text{if $V_{\cO_+} (x,0) > 0$}.
		\end{aligned}
	\]
	Thus, by \eqref{eq:4.2} and \eqref{eq:4.9}, we may find a $\delta_0>0$ such that 
	$B^+_{\delta_0} (0,0) \cap \partial \cO_+ = \emptyset$. 
	By $V_{\cO_+} \in C(\ov{\HS} \setminus \set{0} )$ and $\psi \in C^\infty_c ( \ov{\HS} )$, 
	for each $n \geq 1$, 
	we may find an $\e_{n}>0$ such that 
	\[
		W_n(X) = 0 \quad \text{if $X \in \HS$ and $ \dist ( X , \partial \cO_+ \cap \HS ) < \e_{n}$}.
	\]
	From this and the fact $V_{\cO_+} \equiv V$ in $[ V_{\cO_+} > 0 ] \cap \HS $, it follows that 
	\[
		W_n \in C^\infty(\HS) \cap H^1( \HS, t^{1-2s} d X ). 		
	\]
	By $V_{\cO_+} (X) > 0$ for every $X \in \cO_+$ and $V_{\cO_+} = 0$ on $\partial \cO_+ \cap \HS$, 
	as $n \to \infty$, we have 
	\[
		\begin{aligned}
			 & \zeta_0' \left( n V_{\cO_+} (X) \right) \to 1 \quad 
			\text{for all } X \in \cO_+, \quad 
			\zeta_0' \left( n V_{\cO_+} (X) \right) = 0 \quad 
			\text{for all $X \in \HS \setminus \cO_+$},
			\\
			 & \frac{1}{n} \zeta_0 \left( n V_{\cO_+} (X)\right) \to V_{\cO_+}(X)
			\quad \text{for each } X \in \ov{\cO_+} \cap \HS.
		\end{aligned}
	\]
	Since 
	\[
		\begin{aligned}
			&\nabla W_n = \psi \zeta_0' \left( n V_{\cO_+} \right) \nabla V  
			+ \frac{1}{n} \zeta_0 ( n V_{\cO_+} ) \nabla \psi,
			\quad 
			\| \zeta_0' \|_{L^\infty(\R)} < \infty, 
			\\
			&
			\left| \frac{1}{n} \zeta_0 \left( n V_{\cO_+} (X) \right) \right| 
			\leq \left| V_{\cO_+} (X) \right|\quad \text{for all $X\in\HS$},
		\end{aligned}
	\]
	the dominated convergence theorem implies 
	\[
		\int_{\HS} t^{1-2s} \left| W_n - V_{\cO_+} \psi \right|^2 d X + 
		\int_{\HS} t^{1-2s} 
		\left| \nabla W_n - \psi \chi_{\cO_+} \nabla V  - V_{\cO_+} \nabla \psi \right|^2 d X \to 0.
	\]
	Thus, for every $\psi \in C^\infty_c ( \ov{\HS} )$, 
	\[
		\begin{aligned}
			&W_n \to V_{\cO_+} \psi \quad \text{in} \ H^1( \HS , t^{1-2s} d X ), \quad 
			V_{\cO_+} \psi \in H^1 ( \HS , t^{1-2s} d X ), 
			\\
			&
			\nabla \left( V_{\cO_+} \psi \right) = \psi \chi_{\cO_+} \nabla V + V_{\cO_+} \nabla \psi .
		\end{aligned}
	\]
	Hence, 
	$V_{\cO_+} \in H^1_{\rm loc} (  \ov{\HS} ,t^{1-2s} d X )$.

	Now we prove \eqref{eq:4.11}. Since $C^\infty_c( \ov{\HS} )$ is dense in $H^1( \HS , t^{1-2s} d X )$, 
in view of Remark \ref{Remark:4.1} and $W_n(x,0) \in H^s(\RN)$ 
thanks to the trace theorem, \eqref{eq:4.7} yields 
	\[
		\int_{\HS} t^{1-2s} \nabla V \cdot \nabla W_n \, d X 
		= \kappa_s \int_{ \RN } |x|^\ell ( u_{2}^p - u_{1}^p ) W_n (x,0) \, d x.
	\]
	Observe that 
	\[
		\begin{aligned}
			&\int_{\HS} t^{1-2s} \nabla V \cdot \nabla W_n \, d X 
			\\
			= \ & 
			 \int_{\cO_+} t^{1-2s} \nabla V \cdot \nabla 
			\left( \frac{1}{n} \zeta_0 \left( n V_{\cO_+} \right) \psi  \right) d X
			\\
			 = \ & 
			 \int_{\cO_+} t^{1-2s} \left| \nabla V_{\cO_+} \right|^2 
			\zeta_0' \left( n V_{\cO_+} \right) \psi \, d X 
			+ \int_{\cO_+} t^{1-2s} \frac{1}{n} \zeta_0 \left( n V_{\cO_+} \right) \nabla V_{\cO_+} \cdot \nabla \psi \, d X.
		\end{aligned}
	\]
	The dominated convergence theorem yields 
	\[
		\begin{aligned}
			\lim_{n \to \infty}\int_{\HS} t^{1-2s} \nabla V \cdot \nabla W_n \, d X 
			 & = \int_{\cO_+} t^{1-2s} \left| \nabla V_{\cO_+} \right|^2 \psi \, d X 
			+ \int_{\cO_+} t^{1-2s} V_{\cO_+} \nabla V_{\cO_+} \cdot \nabla \psi \, d X
			\\
			 & = \int_{\cO_+} t^{1-2s} \left| \nabla V_{\cO_+} \right|^2 \psi \, d X 
			+ \int_{\cO_+} t^{1-2s} \nabla \left( \frac{V_{\cO_+}^2}{2} \right) \cdot \nabla \psi \, d X.
		\end{aligned}
	\]
	On the other hand, since $W_n \to V_{\cO_+} \psi$ strongly in $H^1(\HS, t^{1-2s} d X)$, 
	by the trace theorem and $V_{\cO_+} \in C( \ov{\HS} \setminus \set{0}  )$, 
	we see  $W_n(x,0) \to \left( V_{\cO_+} \psi \right) (x,0)$ 
	strongly in $H^s(\RN)$. 
	Therefore, thanks to $|x|^{\ell} u_i^p \in L^{ \frac{2N}{N+2s} } (B_1)$ and 
	$u_i \in L^\infty_{\rm loc} (\RN \setminus \set{0} )$, we obtain 
	\[
		\lim_{n \to \infty} \kappa_s \int_{ \RN } |x|^\ell \left( u_{2}^p - u_{1}^p \right) W_n (x,0) \, d x
		= \kappa_s \int_{ \RN } |x|^\ell \left( u_{2}^p - u_{1}^p \right) V_{\cO_+}(x,0) \psi(x,0) \, d x.
	\]
	Thus, \eqref{eq:4.11} holds. 
\end{proof}

\begin{lemma}\label{Lemma:4.3}
	All components $\cO_{\pm} \subset\HS$ of $[ V >0 ]$ and $[V<0]$ are unbounded.  
\end{lemma}

\begin{proof}
	As in Lemma \ref{Lemma:4.2}, we only deal with $\cO_+$ since the case $\cO_-$ can be shown similarly. 
	We argue by contradiction and suppose that $\cO_+$ is bounded. 
	Then $V_{\cO_+}$ (resp. $V_{\cO_{+}} (\cdot, 0) $ ) in \eqref{eq:4.10} has compact support in $\ov{\HS}$ 
	(resp. $\RN$), and $V_{\cO_+} \in H^1(\HS, t^{1-2s} d X)$ holds in view of Lemma \ref{Lemma:4.2}. 
	In addition, we may take $\psi \equiv 1$ in \eqref{eq:4.11} due to the compactness of $\supp V_{\cO_{+}}$ 
	and $\supp V_{\cO_{+}} (\cdot, 0)$.

	Now we distinguish two cases: (i) $V_{\cO_+} (\cdot ,0) \equiv 0$, (ii) $V_{\cO_+} (\cdot ,0) \not\equiv 0$. 
	When (i) happens, then putting $\psi \equiv 1$ in \eqref{eq:4.11}, we obtain 
	\[
		\int_{\HS} t^{1-2s} | \nabla V_{\cO_+} |^2 \, d X 
		= \kappa_s \int_{ \RN } |x|^\ell \left( u_{2}^p - u_{1}^p \right) V_{\cO_+}(x,0) \, d x 
		= 0.
	\]
	Thus, $V_{\cO_+} \equiv 0$ in $\HS$, however, this contradicts $V_{\cO_+} > 0$ in $\cO_+$.

	Next, we consider case (ii). 
	In this case, let $x_0 \in \RN$ satisfy $V_{\cO_+} (x_0,0) > 0$. Then 
	$0< V_{\cO_+} (x_0,0) = V(x_0,0) = u_{2}(x_0) - u_{1} (x_0)$ and 
	$B^+_{r}(x_0,0) \subset \cO_+$ holds for some $r>0$ and 
	$V_{\cO_+} (x,0) > 0$ if $|x-x_0|<r$. 
	Let $\psi \equiv 1$ in \eqref{eq:4.11} to get 
	\[
		\int_{\HS} t^{1-2s} \left| \nabla V_{\cO_+} \right|^2 d X 
		= \kappa_s \int_{ \RN } |x|^\ell \left( u_{2}^p - u_{1}^p \right) V_{\cO_+}(x,0) \, d x.
	\]
	Moreover, by the trace theorem, $V_{\cO_+} (x,0) \in H^s(\RN)$ and 
	put $W_{\cO_+}(x,t) := P_s(\cdot, t) \ast V_{\cO_+} (\cdot, 0)$. 
	By \cite[(2.8)]{HIK-20} and the density of 
	$C^\infty_c( \RN )$ in $H^s (\RN) $ and $C^\infty_c(\ov{\HS})$ in $H^1(\HS , t^{1-2s} d X )$, respectively, we have 
	\[
		\kappa_s \| V_{\cO_+} (\cdot, 0) \|_{\dot{H}^s(\RN)}^2 
		= \int_{\HS} t^{1-2s} \left| \nabla W_{\cO_+} \right|^2  d X 
		\leq \int_{\HS} t^{1-2s}  \left| \nabla V_{\cO_+} \right|^2 d X ,
	\]
	which gives 
	\[
		\begin{aligned}
			\| V_{\cO_+} (\cdot, 0) \|_{\dot{H}^s(\RN)}^2  
			 & \leq 
			\int_{ \RN } |x|^\ell \left( u_2^p - u_{1}^p \right) V_{\cO_+}(x,0) \, d x
			\\
			 & = \int_{ \RN } |x|^\ell p \int_{0}^{1} \left( \theta u_2 + (1-\theta) u_1 \right)^{p-1} \, d \theta 
			(u_2-u_1) V_{\cO_+} (x,0) \, d x.
		\end{aligned}
	\]
	From the fact that $V_{\cO_+}(x,0) > 0$ yields $ V_{\cO_+} (x) = u_{2} (x) - u_{1}(x) > 0$, 
	we infer that 
	\begin{equation}\label{eq:4.13}
		\| V_{\cO_+} (\cdot, 0) \|_{\dot{H}^s(\RN)}^2 < p \int_{ \RN } |x|^\ell u_{2}^{p-1}
		V_{\cO_+}^2 (x,0) \, d x.
	\end{equation}
	Since $V_{\cO_+} (\cdot , 0) \in H^s(\RN)$ and $\supp V_{\cO_{+}} (\cdot, 0)$ is compact in $\RN$, 
	\eqref{eq:4.13} contradicts the stability of $u_{2}$ (see Remark \ref{Remark:4.2}). 
	Hence, we conclude that $\cO_+$ is unbounded. 
\end{proof}

\begin{lemma}\label{Lemma:4.4}
	For $k \in \N$, let $I_{k,-}$ be a component of $[u_1-u_2]>0$ in Remark~ {\rm\ref{Remark:4.2}} and
	let $\cO_{k,-} \subset \HS$ be a component of $[ V < 0 ]$ satisfying
	$B_{r_x}^+( x,0 ) \cap \HS \subset \cO_{k,-}$ holds for each $x \in I_{k,-}$ and some $r_x>0$. 
	For each $\cO_{k,-}$, define $V_{\cO_{k,-}}$ same as in \eqref{eq:4.10}. Then 
	$\supp V_{\cO_{k,-}} (\cdot, 0)$ is compact in $\RN$. 
\end{lemma}

\begin{Remark}\label{Remark:4.3}
For each $I_{k,\pm}$, we can choose the components $\cO_{k,\pm} \subset \HS$ of 
$[ V > 0 ]$ and $[ V < 0 ]$ in same way as above. 
In fact, if $I_{k,\pm} = A_{a_{k,\pm} , b_{k,\pm}} = \set{ x \in \RN | a_{k,\pm} < |x| < b_{k,\pm} }$, 
then we may choose a small $\e_k > 0$ so that 
\[
	\begin{aligned}
		 &V(x,t) = U_2 (x,t) - U_1(x,t) > 0  &  & \text{if} \ 
		a_{k,+} + \e_k \leq |x| \leq b_{k,+} - \e_k, \ 
		0 \leq t \leq \e_k,
		\\
		 &V(x,t) = U_2 (x,t) - U_1(x,t) <  0 &  & \text{if} \ 
		a_{k,-} + \e_k \leq |x| \leq b_{k,-} - \e_k, \ 
		0 \leq t \leq \e_k. 
	\end{aligned}
\]
\end{Remark}

\begin{proof}[Proof of Lemma {\rm\ref{Lemma:4.4}}]
	Suppose that $\mathrm{supp}\, V_{\cO_{k,-}} (\cdot, 0) $ is unbounded. 
	By the assumption and $V_{\cO_{k,-}} \in C(\ov{\HS} \setminus \set{0})$ thanks to Lemma \ref{Lemma:4.2}, 
	we may take $(x_n)_n$ and $(r_n)_n$ so that 
	\[
		|x_n| < |x_{n+1}|, \quad |x_n| \to \infty \quad\text{as $n\to\infty$},
		\quad 
		V_{\cO_{k,-}} < 0 \quad \text{in}\ \ov{B_{r_n}^+(x_n,0)}.
	\]
	Hence, $B_{r_n}^+(x_n,0) \cap \HS \subset \cO_{k,-}$ for each $n \geq 1$.

	Next, we claim $V (x,0) \leq 0$ for all $|x| \geq |x_1|$. 
	If this is true, then the component $[V(x,0)>0] = [ u_{2} - u_{1} > 0 ]$ is bounded 
	and this contradicts Lemma \ref{Lemma:4.1}, and we conclude that $V_{\cO_{k,-}} (\cdot, 0)$ has compact support.

	We show $V(x,0) \leq 0$ for all $|x| \geq |x_1|$. 
	Suppose that there exists a $y_0 \in \RN$ such that 
	$ |y_0| > |x_1|$ and $V (y_0,0) > 0$. Select $s_0>0$ and $n_0 \in \N$ so that 
	$|x_{n_0}| < |y_0| < |x_{n_0+1}|$ and $V > 0$ in $\ov{B_{s_0}^+(y_0,0)}$. 
	Since $B_{r_{n_0}}^+(x_{n_0},0) \cap \HS \subset \cO_{k,-}$, 
	$B_{r_{n_0+1}}^+(x_{n_0+1},0) \cap \HS \subset \cO_{k,-}$ and 
	$\cO_{k,-}$ is open and connected in $\HS$, there exists a $\gamma_0(\tau) \in C( [0,1] , \ov{\HS} )$ such that 
	\[
		\gamma_0(0) = (x_{n_0},0), \quad 
		\gamma_0 (1) = (x_{n_0+1},0), \quad 
		\gamma_0 (\tau) \in \cO_{k,-} \quad \text{for each $\tau \in (0,1)$}.
	\]
	Let $\cO_{+} \subset \HS$ be a component of $[V>0]$ satisfying 
	$B_{s_0}^+(y_0,0) \cap \HS \subset \cO_+$. 
	Due to the existence of $\gamma_0$ and the radial symmetry of $V$, $\cO_{+}$ becomes bounded 
	(see Appendix \ref{section:A} for the details). 
	However, this contradicts Lemma \ref{Lemma:4.3}. 
	Thus, $V(x,0) \leq 0$ for all $|x| \geq |x_1|$ and 
	we complete the proof. 
\end{proof}

\begin{lemma}\label{Lemma:4.5}
	For each component $\cO_- \subset\HS$ of $[V<0]$, let $V_{\cO_-}$ be as in \eqref{eq:4.10}.
	If $\psi \in C^2_c( \ov{\HS} )$ satisfies 
	\begin{equation}\label{eq:4.14}
		t^{1-2s} \partial_t \psi (x,t) \to 0 \quad \text{as $t\to 0$} ,
	\end{equation}
	then 
	\begin{equation}\label{eq:4.15}
		\begin{aligned}
			     & \int_{\HS} t^{1-2s} \left| \nabla V_{\cO_-} \right|^2 \psi \, d X 
			\\
			= \  & \frac{1}{2} \int_{\cO_-} V_{\cO_-}^2 \diver \left( t^{1-2s} \nabla \psi \right) d X 
			+ \kappa_s \int_{ \RN } |x|^\ell \left( u_{2}^p - u_{1}^p \right) V_{\cO_{-}}(x,0) \psi(x,0) \, d x.
		\end{aligned}
	\end{equation}
\end{lemma}

\begin{proof}
	From \eqref{eq:4.11}, it suffices to prove 
	\[
		-\frac{1}{2}\int_{\HS} t^{1-2s}\nabla \left( V_{\cO_-}^2 \right) \cdot \nabla \psi \, d X
		=
		\frac{1}{2}\int_{\cO_-} V_{\cO_{-}}^2 \diver \left( t^{1-2s} \nabla \psi \right)
		d X.
	\]
	To this end, we use $\zeta_0$ in \eqref{eq:4.12}. 
	We note that for each $X \in \cO_-$, as $n \to \infty$, 
	\[
		\begin{aligned}
			\nabla \left[ \left( \frac{1}{n} \zeta_0 \left( n V_{\cO_-}(X) \right) \right)^2 \right] \cdot \nabla \psi (X)
			 & = 2 \frac{1}{n} \zeta_0 \left( n V_{\cO_-} (X) \right) \zeta_0' \left( nV_{\cO_-}(X) \right) 
			\nabla V_{\cO_-} (X) \cdot \nabla \psi(X)
			\\
			 & \to 2 V_{\cO_-}(X) \nabla V_{\cO_-}(X) \cdot \nabla \psi(X).
		\end{aligned}
	\]
	Let $R>0$ satisfy $ \supp \psi \subset (-R,R)^{N} \times [0,R)$. 
	Since $V_{\cO_{-}} \in C(\ov{\HS})$ due to \eqref{eq:4.9} and Lemma \ref{Lemma:4.2}, 
	by \eqref{eq:4.14}, $ \zeta_0 (n V_{\cO_-}) \in C^\infty(\HS)$ and the dominated convergence theorem, 
	\[
		\begin{aligned}
			-\frac{1}{2}\int_{\HS} t^{1-2s}\nabla \left( V_{\cO_-}^2 \right) \cdot \nabla \psi \, d X
			 & = - \frac{1}{2} \lim_{n \to \infty}
			\int_{\cO_-} t^{1-2s}\nabla \left\{ \frac{1}{n} \zeta_0 \left( n V_{\cO_-}  \right) \right\}^2 \cdot 
			\nabla \psi \, d X
			\\
			 & =- \frac{1}{2} \lim_{n \to \infty}
			\int_{[-R,R]^{N} \times [0,R] } t^{1-2s}\nabla \left\{ \frac{1}{n} \zeta_0 \left( n V_{\cO_-} \right) \right\}^2 
			\cdot \nabla \psi \, d X
			\\
			 & = \frac{1}{2} \lim_{n \to \infty} \int_{[-R,R]^{N} \times [0,R] }
			\left\{ \frac{1}{n} \zeta_0 \left( n V_{\cO_-} \right) \right\}^2 
			\diver \left( t^{1-2s} \nabla \psi \right)  d X 
			\\
			 & = \frac{1}{2} \int_{\cO_-} V_{\cO_-}^2 \diver \left( t^{1-2s} \nabla \psi \right) d X.
		\end{aligned}
	\]
	Therefore, under \eqref{eq:4.14}, we get \eqref{eq:4.15} through \eqref{eq:4.11}. 
\end{proof}

\begin{lemma}\label{Lemma:4.6}
	For $k\in\mathbb N$, let $\cO_{k,-}$ and $V_{\cO_{k,-}}$ be as in Lemma~{\rm\ref{Lemma:4.4}}.
	Then 
	\begin{equation}\label{eq:4.16}
		\left| V_{\cO_{k,-}} (x,t) \right| \leq C_0 \left( |x|^2 + t^2 \right)^{- \frac{N-2s}{2} } 
		\quad \text{for all } (x,t) \in \HS.
	\end{equation}
	Furthermore, $\nabla V_{\cO_{k,-}} \in L^2(\HS , t^{1-2s} d X)$.
\end{lemma}

\begin{proof}
	We first prove \eqref{eq:4.16}.
	Following \cite{CS-07}, we set 
	\[
		\Gamma (x,t) := \left( |x|^2 + t^2 \right)^{-\frac{N-2s}{2}}.
	\]
	Then $-\diver ( t^{1-2s} \nabla \Gamma ) = 0$ in $\HS$. 
	Since $V_{\cO_{k,-}}(\cdot,0)$ has compact support in $\RN$ by Lemma \ref{Lemma:4.4}
	and $V_{\cO_{k,-}} (X) = 0$ for $|X| \ll 1$ due to \eqref{eq:4.9}, 
	we may find a $C_0>0$ such that 
	\[
		0 \leq -V_{\cO_{k,-}} (x,0) < C_0 \Gamma(x,0) \quad \text{for all $x \in \RN$}.
	\]
	Notice also that $-\diver (t^{1-2s} \nabla V ) = 0$ in $\HS$ 
	due to \eqref{eq:4.5} and \eqref{eq:4.8}.  
	This together with $V_{\cO_{k,-}} = V$ in $\cO_{k,-}$ yields 
	\[
		- \Delta_{t,x} \left( C_0 \Gamma + V_{\cO_{k,-}} \right) - \frac{1-2s}{t}
		\partial_t \left( C_0 \Gamma + V_{\cO_{k,-}} \right) = 0 \quad 
		\text{in} \ \cO_{k,-}, \quad V_{\cO_{k,-}} = 0 \quad \text{on} \ 
		\partial \cO_{k,-} \cap \HS.
	\]
	Since $U_i(X) \to 0$ as $|X| \to \infty$, we have 
	$V_{\cO_{k,-}}(X) \to 0$ as $|X| \to \infty$ and 
	the strong maximum principle asserts that 
	$C_0 \Gamma + V_{\cO_{k,-}} \ge 0$ in $\cO_{k,-}$. 
	Otherwise, since $C_0 \Gamma + V_{\cO_{k,-}} \geq 0$ on $\partial \cO_{k,-} \cup \RN \times \{0\}$, 
	we may find a negative global minimum $X_0 \in \cO_{k,-}$ of $C_0 \Gamma + V_{\cO_{k,-}}$. 
	Then the strong maximum principle yields 
	$C_0 \Gamma + V_{\cO_{k,-}} \equiv (C_0\Gamma + V_{\cO_{k,-}}) (X_0) < 0$
	in $B_{r}^+(X_0)$ as long as $B_r^+(X_0) \subset \cO_{k,-}$. 
	Thus, enlarging $r>0$ and noting $V_{\cO_{k,-}} (X) = 0$ for $|X| \ll 1$, 
	we have $ \ov{B_r^+(X_0)} \cap \partial \cO_{k,-} \neq \emptyset$ or 
	$\ov{B_r^+(X_0)} \cap \RN \times \{0\} \neq \emptyset $
	and this leads a contradiction. 
	Therefore, $ C_0 \Gamma + V_{\cO_{k,-}} \geq 0$ in $\cO_{k,-}$. 
	By $V_{\cO_{k,-}} \equiv 0$ in $\HS \setminus \cO_{k,-}$, 
	\eqref{eq:4.16} holds.

	Next we prove $\nabla V_{\cO_{k,-}} \in L^2(\HS , t^{1-2s} d X)$.
	In \eqref{eq:4.15}, we compute the term:
	\[
		\int_{\cO_{k,-}} V_{\cO_{k,-}}^2 
		\diver \left( t^{1-2s} \nabla \psi  \right) d X.
	\]
	For $R>0$, we consider 
	\[
		\psi(x,t) = \psi_R(x,t) := \psi_0( R^{-1}|x|) \psi_0 ( R^{-1}t),
	\]
	where $\psi_0 \in C^\infty_c(\R)$ with $\psi_0 (\tau) = 1$ for $|\tau| \leq 1$
	and $\psi_0 (\tau) = 0$ for $|\tau| \geq 2$. 
	Then \eqref{eq:4.14} is satisfied and 
	\[
		\begin{aligned}
			     & \diver \left( t^{1-2s} \nabla \psi \right)
			\\
			= \  & t^{1-2s} \Delta_{t,x} \psi 
			+ (1-2s) t^{-2s} \partial_t \psi
			\\
			= \  & R^{-2} t^{1-2s} \psi_0''(R^{-1} t) \psi_0(R^{-1} |x|)
			+ t^{1-2s}  \psi_0(R^{-1} t)
			\left[ R^{-2}\psi_0''(R^{-1} |x|) + R^{-1}\frac{N-1}{|x|} \psi_0'( R^{-1} |x| ) \right]
			\\
			     & \qquad + (1-2s) t^{-2s} R^{-1} \psi_0 ( R^{-1} |x| ) \psi_0'( R^{-1} t ).
		\end{aligned}
	\]
	By \eqref{eq:4.16}, we observe that 
	\begin{equation}\label{eq:4.17}
		\begin{aligned}
			& \int_{\cO_{k,-}} V_{\cO_{k,-}}^2 \left| R^{-2} t^{1-2s} \psi_0''(R^{-1} t) 
			\psi_0 (R^{-1}|x| ) \right| d X
			\\
			\leq \  & C R^{-2}
			\int_{R}^{2R} d t \int_{  |x| \leq 2R}
			\frac{t^{1-2s}}{(|x|^2 + t^2)^{N-2s} } \, d x 
			\\
			= \     & C R^{-2}
			\int_R^{2R} t^{1-2s} d t 
			\int_{  |y| \leq 2R/t} t^{-2N+4s} \left( 1 + |y|^2 \right)^{ -(N-2s) } t^N \, d y 
			\\
			\leq \  & C R^{-2}
			\int_{R}^{2R} t^{ -N + 2s + 1 } d t \int_{  |y| \leq 2} (1+|y|^2)^{-(N-2s)} \, d y
			\\
			\leq \  & CR^{-N+2s}.
		\end{aligned}
	\end{equation}
	In a similar way, 
	\begin{equation*}
		\begin{aligned}
			        & \int_{\cO_{k,-}} V_{\cO_{k,-}}^2 \left| t^{-2s} R^{-1} \psi_0 (R^{-1} |x|) \psi_0'(R^{-1} t) \right| d X 
			\\
			\leq \  & C R^{-1} \int_R^{2R}  d t 
			\int_{  |x| \leq 2R} \frac{t^{-2s}}{(|x|^2 + t^2)^{N-2s}} \, d x 
			\leq C R^{-1} \int_R^{2R} t^{ -N+2s } \, d t \leq C R^{-N+2s}
		\end{aligned}
	\end{equation*}
	and 
	\begin{equation}\label{eq:4.18}
		\begin{aligned}
			        & \int_{\cO_{k,-}} V_{\cO_{k,-}}^2 t^{1-2s} \psi_0(R^{-1} t)
			\left| R^{-2}\psi_0''(R^{-1} |x|) + R^{-1}\frac{N-1}{|x|} \psi_0'( R^{-1} |x| ) \right| d X
			\\
			\leq \  & C
			\int_0^{2R} d t \int_{ R \leq |x| \leq 2R}
			\left( |x|^2 + t^2 \right)^{-(N-2s)} t^{1-2s}R^{-2} \, d x
			\\
			\leq \  & C \int_{0}^{2R}
			\left(  R^2 + t^2\right)^{-N+2s} t^{1-2s} R^{-2+N} \, d t
			\\
			= \     & C\int_0^{2} R^{-2N+4s} \left( 1 + \tau^2 \right)^{-N+2s}  (\tau R)^{1-2s} R^{N-2} R \, d \tau
			= C R^{-N+2s}.
		\end{aligned}
	\end{equation}
	Thus, 
	\[
		\limsup_{R \to \infty} \int_{\cO_{k,-}} V_{\cO_{k,-}}^2 
		\left| \diver \left( t^{1-2s} \nabla \psi_R \right) \right| d X = 0.
	\]
	Noting that $V_{\cO_{k,-}}(x,0)$ has compact support by Lemma \ref{Lemma:4.4}, we also have 
	\[
		\limsup_{R \to \infty} \int_{ \RN } |x|^\ell (u_{2}^p - u_{1}^p) V_{\cO_{k,-}} (x,0) 
		\psi_R (x,0) \, d x 
		= \int_{ \RN } |x|^\ell \left( u_2^p - u_1^p \right) V_{\cO_{k,-}} (x,0) \, d x
		< \infty.
	\]
	Hence, \eqref{eq:4.15}, $u_1 \in L^\infty(\RN)$ and the fact that $V_{\cO_{k,-}} (x,0) < 0$ implies $u_2(x) - u_1(x) < 0$ give
	\begin{equation}\label{eq:4.19}
		\begin{aligned}
			\int_{\HS} t^{1-2s} \left| \nabla V_{\cO_{k,-}} \right|^2 d X 
			&= \kappa_s \int_{ \RN } |x|^\ell (u_{2}^p - u_{1}^p) V_{\cO_{k,-}} (x,0) \, d x 
			\\
			&= \kappa_s \int_{ \RN } |x|^\ell (u_{1}^p - u_{2}^p) \left( - V_{\cO_{k,-}} (x,0) \right) d x 
			\\
			&< \kappa_s p \int_{ \RN } |x|^\ell u_{1}^{p-1} V_{\cO_{k,-}}^2 (x,0) \, d x < \infty.
		\end{aligned}
	\end{equation}
	Therefore, $\nabla V_{\cO_{k,-}} \in L^2( \HS , t^{1-2s} d X )$. 
\end{proof}

Now, we shall complete the proof of Proposition \ref{Proposition:4.1}.
\begin{proof}[Proof of Proposition {\rm\ref{Proposition:4.1}}]
	Recalling Lemmata \ref{Lemma:4.2} and \ref{Lemma:4.4}, we have 
	$V_{\cO_{k,-}} (\cdot, 0)\in H^s(\RN)$ and 
	\[
		\kappa_s \| V_{\cO_{k,-}} (\cdot,0) \|_{\dot{H}^s(\RN)}^2 
		= \int_{ \HS } t^{1-2s} \left| \nabla ( P_s (\cdot,t) \ast V_{\cO_{k,-}} (\cdot, 0)  ) (x) \right|^2 d X. 
	\]
	Notice that the last assertion follows from \cite[(2.7)]{HIK-20}. 
	
	In the following, we claim 
	\begin{equation}\label{eq:4.20}
		\int_{\HS} t^{1-2s} \left| \nabla ( P_s (\cdot,t) \ast V_{\cO_{k,-}} (\cdot, 0)  ) \right|^2 d X 
		\leq \int_{\HS} t^{1-2s} | \nabla V_{\cO_{k,-}}|^2 \, d X. 
	\end{equation}
	If this is true, then we get a contradiction for the stability of $u_1$ from \eqref{eq:4.19}, \eqref{eq:4.20} 
	and the compactness of $\supp V_{\cO_{k,-}} (\cdot, 0)$:
	\[
		\| V_{\cO_{k,-}} (\cdot,0) \|_{\dot{H}^s(\RN)}^2  < 
		p \int_{ \RN } |x|^\ell u_{1}^{p-1} V_{\cO_{k,-}}^2 (x,0) \, d x.
	\]
	Hence, \eqref{eq:4.4} does not hold and we complete the proof of Proposition \ref{Proposition:4.1}.

	We show \eqref{eq:4.20}. 
	Set $W(x,t) := ( P_s (\cdot,t) \ast V_{\cO_{k,-}} (\cdot, 0)  ) (x) $. 
	Notice that 
	\[
		- \diver (t^{1-2s} \nabla W) = 0 \quad \text{in} \ \HS, \quad 
		W (x,0) = V_{\cO_{k,-}} (x,0) \leq 0, \quad 
		\nabla W \in L^2( \HS , t^{1-2s} d X ).
	\]
	By Lemma \ref{Lemma:B.2} and the fact that $V_{\cO_{k,-}} (x,0)$ has compact support,  
	$W(X) \to 0$ as $|X| \to \infty$. 
	Since $V_{\cO_{k,-}} (x,0) = 0$ for $|x| \ll 1$ by \eqref{eq:4.2} and 
	$V_{\cO_{k,-}} (\cdot, 0)  \in C(\RN)$, 
	as in Lemma \ref{Lemma:4.6}, one may check that 
	\begin{equation}\label{eq:4.21}
		0 \leq -W(X) \leq C_0( |x|^2 + t^2 )^{ - \frac{N-2s}{2} } \quad \text{for all $X \in \HS$}. 
	\end{equation}
	Moreover, for $\psi_R(x,t) = \psi_0(R^{-1}t) \psi_0(R^{-1} |x|)$ with $R \gg 1$ 
	and $\psi_0$ as in the proof of Lemma \ref{Lemma:4.6}, 
	we obtain $\psi_R(x,0) V_{\cO_{k,-}} (x,0) = V_{\cO_{k,-}} (x,0)$. 
	By $\psi_R ( W - V_{\cO_{k,-}} ) \in H^1(\HS, t^{1-2s} d X)$ with 
	$\psi_R (x,0) ( W(x,0) - V_{\cO_{k,-}} (x,0) ) \equiv 0$, 
	we infer from \cite[Lemma 2.2]{HIK-20} that 
	\begin{equation}\label{eq:4.22}
		\int_{\HS} t^{1-2s} \nabla W \cdot \nabla \left( \psi_R \left( W - V_{\cO_{k,-}} \right) \right) d X 
		=\kappa_s \la V_{\cO_{k,-}} , \psi_R (\cdot, 0) \left( W(\cdot, 0) - V_{\cO_{k,-}} (\cdot, 0) \right)  \ra_{\dHs (\RN)} 
		= 0.
	\end{equation}
	Thus, if 
	\begin{equation}\label{eq:4.23}
		\int_{\HS} t^{1-2s} \nabla W \cdot \nabla \psi_R \left( W - V_{\cO_{k,-}} \right) d X 
		\to 0 \quad \text{as $R \to \infty$},
	\end{equation}
	then the facts $\nabla V_{\cO_{k,-}} , \nabla W \in L^2(\HS, t^{1-2s} d X)$, \eqref{eq:4.22} and \eqref{eq:4.23} imply  
	\[
		\begin{aligned}
			\int_{\HS} t^{1-2s} \left| \nabla W \right|^2 d X 
			 & = \int_{\HS} t^{1-2s} \nabla W \cdot \nabla V_{\cO_{k,-}} \, d X 
			\\
			 & \leq \left[ \int_{ \HS} t^{1-2s} \left| \nabla W \right|^2 d X \right]^{1/2}
			\left[ \int_{ \HS} t^{1-2s} \left|\nabla V_{\cO_{k,-}} \right|^2 d X \right]^{1/2} < \infty
		\end{aligned}
	\]
	and \eqref{eq:4.20} holds.

	To prove \eqref{eq:4.23}, we see that 
	\[
		\begin{aligned}
			        & \int_{ \HS} t^{1-2s} \left| \nabla W \cdot (\nabla \psi_R) \left( W - V_{\cO_{k,-}} \right)  \right| d X 
			\\
			\leq \  & \left[ \int_{ \HS} t^{1-2s}  \left|\nabla W \right|^2 d X \right]^{\frac{1}{2}}
			\left[ \int_{\HS} t^{1-2s} \left| \nabla \psi_R \right|^2
				2 \left( V_{\cO_{k,-}}^2 + W^2 \right) d X  \right]^{\frac{1}{2}}.
		\end{aligned}
	\]
	By Lemma \ref{Lemma:4.6} and \eqref{eq:4.21}, as in \eqref{eq:4.17} and \eqref{eq:4.18}, we may prove 
	\[
		\begin{aligned}
			        & \int_{\HS} t^{1-2s} | \nabla \psi_R |^2 
			\left( V_{\cO_{k,-}}^2 + W^2 \right)  d X 
			\\
			\leq \  & 
			C R^{-2} \int_R^{2R} d t \int_{  |x| \leq 2R} \frac{t^{1-2s}}{(|x|^2+t^2)^{N-2s}} \, d x 
			+ C R^{-2} \int_0^{2R} d t \int_{R \leq |x| \leq 2R}
			\frac{t^{1-2s}}{(|x|^2 + t^2)^{ N-2s }} \, d x
			\\
			\leq \  & C R^{-N+2s} \to 0 \quad (R \to \infty).
		\end{aligned}
	\]
	Thus, \eqref{eq:4.23} holds and we complete the proof of Proposition \ref{Proposition:4.1}. 
\end{proof}


\subsection{Some properties of stable solutions}\label{section:4.2}


	In this subsection, we investigate some properties of positive radial stable solutions. 
We first obtain the separation property of $(u_{m,0})_{m>0}$ 
as a corollary of Corollary \ref{Corollary:3.1} and Proposition \ref{Proposition:4.1}:

\begin{corollary}\label{Corollary:4.1}
	For $(u_{m,0})_{m>0}$ in Corollary {\rm\ref{Corollary:3.1}}, if $m_1 < m_2$, then 
	$0< u_{m_1,0} < u_{m_2,0}$ in $\RN$. 
\end{corollary}

	\begin{proof}
Notice that $\wt{u}_0 \in C(\RN)$, $\wt{u}_0(0) > 0$, $m \wt{u}_0(0) = u_{m,0}(0)$ and 
$m = \| u_{m,0} \|_{L^\infty (\RN) }$. 
Applying Proposition \ref{Proposition:4.1} for $u_{m_1,0}$ and $u_{m_2,0}$ with $m_1 < m_2$, 
we have $u_{m_1,0} < u_{m_2,0}$ in $\RN$. 
	\end{proof}

	Remark that by Corollary \ref{Corollary:4.1}, 
	the limit $ \lim_{m \to \infty} u_{m,0}(x) =:u_{\infty,0} (x) $ exists for any $x \in \RN$.

\begin{proposition}\label{Proposition:4.2}
	The sequence $(\varphi  u_{m,0})_m$ is bounded in $H^s(\RN)$ for each $\varphi \in C^\infty_c(\RN)$. 
	Furthermore, $u_{\infty,0} = u_S$ and 
	$u_{m,0} \to u_S$ strongly in $H^s_{\rm loc} (\RN)$ as $m \to \infty$. 
\end{proposition}

\begin{proof}
	Set $U_m(x,t) := (P_s (\cdot, t) \ast u_{m,0} ) (x)$ and let $\psi \in C^\infty_c(\ov{\HS})$. 
	By the trace theorem, it suffices to show that $( \psi U_m)_m$ is bounded in $H^1(\HS,t^{1-2s} d X)$. 
	From $0 \leq u_{m,0} \leq u_S$, we remark that 
	\[
		0 \leq U_m (X) \leq U_S(X) = (P_s(\cdot, t ) \ast u_S ) (x).
	\]
	By $U_S \in H^1_{\rm loc} ( \ov{\HS}, t^{1-2s} d X)$ due to \cite[Lemma 2.1]{HIK-20}, 
	we observe that
	$(U_m)_m$ is bounded in $L^2_{\rm loc} ( \HS,t^{1-2s} d X )$.

	Next, since $U_m$ satisfies 
	\[
		\int_{ \HS} t^{1-2s} \nabla U_m \cdot \nabla \varphi \, d X 
		= \kappa_s \la u_{m,0} , \varphi (\cdot, 0) \ra_{\dot{H}^s (\RN)}
		= \kappa_s \int_{ \RN } |x|^\ell u_{m,0}^p \varphi(x,0) \, d x
	\]
	for all $\varphi \in H^1(\HS,t^{1-2s} d X )$ with compact support in $\ov{\HS}$, 
	we may test $ \psi^2 U_m$ and get 
	\[
		\int_{ \HS} t^{1-2s} \nabla U_m \cdot \nabla \left( \psi^2 U_m \right) d X
		= 
		\kappa_s \int_{ \RN } |x|^\ell u_{m,0}^{p+1} \psi^2(x,0) \, d x 
		\leq 
		\kappa_s \int_{ \RN } |x|^\ell u_S^{p+1} \psi^2(x,0) \, d x.
	\]
	Since $|x|^\ell u_S^{p+1} \in L^1_{\rm loc} (\RN)$ due to $p>p_S(N,\ell)$, 
	applying the same argument as in \eqref{eq:3.30},
	we may verify that 
	$ ( \nabla (\psi U_m ) )_m$ is bounded in $L^2 (\HS, t^{1-2s} d X )$.

	Since $ u_{m,0}(x) \nearrow u_{\infty,0}(x)$ as $m \to \infty$, 
	it holds that $u_{m,0} \rightharpoonup u_{\infty,0}$ weakly in $H^s_{\rm loc} (\RN)$ as $m \to \infty$.
	By $ u_{m_1,0}(x) \leq u_{m_2,0} (x)$ for $m_1 < m_2$ and $u_{m,0}(x) \leq u_S(x)$, one sees that 
	\[
		u_{\infty,0}(x) \leq u_S(x), \quad 
		\int_{ \RN } |x|^\ell u_{m,0}^p(x) \varphi(x) \, d x 
		\to \int_{ \RN } |x|^\ell u_{\infty,0}^p(x) \varphi (x) \, d x 
		\quad \text{for all $\varphi \in C^\infty_c(\RN)$}.
	\]
	Furthermore, from $ u_{m,0} (x) \leq u_{\infty,0}(x) \leq u_S(x)$
	and $u_{m,0} \to u_{\infty,0}$ weakly in $H^s_{\rm loc} (\RN)$, 
	as in \eqref{eq:3.26}, 
	it follows that $0< u_{\infty,0}$ and 
	\[
		\la u_{m,0} , \varphi \ra_{\dot{H}^s (\RN)} \to \la u_{\infty,0} , \varphi \ra_{\dot{H}^s (\RN)}
		\quad \text{for each $\varphi \in C^\infty_c(\RN)$}.
	\]
	Hence, $(-\Delta)^s u_{\infty,0} = |x|^\ell u_{\infty,0}^p$ in $\RN$.

	From the definition of $u_{m,0}$ in Corollary \ref{Corollary:3.1}, it follows that for every $m > 0$, 
	\[
		\begin{aligned}
			m u_{\infty,0} ( m^{1/\theta_0} x ) 
			= \lim_{k \to \infty} m u_{k,0} ( m^{1/\theta_0} x ) 
			&= \lim_{k \to \infty} m k \wt{u}_0 ( m^{ 1 / \theta_{0} } k^{1/\theta_{0}} x ) 
			\\
			&= \lim_{ t \to \infty } t \wt{u}_0 ( t^{1 / \theta_0} x ) 
			= \lim_{ t \to \infty } u_{t,0} (x)
			= u_{\infty,0}(x).
		\end{aligned}
	\]
	For any $|x|>0$, choose $m= |x|^{-\theta_0}$. Then, since $u_{\infty,0}$ is radially symmetric, 
	we obtain 
	\[
		u_{\infty, 0}(x) = |x|^{-\theta_0} u_{\infty, 0}( (|x|^{-\theta_0})^{1/\theta_0} x ) 
		= |x|^{ -\theta_{0} } u_{\infty,0}(1) 
		= C u_S(x).
	\]
	Recalling that $u_{\infty,0}$ satisfies $(-\Delta)^s u_{\infty,0} = |x|^\ell u_{\infty,0}^p$ in $\RN$, 
	we have $C=1$ and $u_{\infty,0} = u_S$.

	Since $u_S \in L^2_{\rm loc} (\RN)$ and 
$u_{m,0}(x) \nearrow u_{\infty,0} (x) = u_S(x)$ for $x \in \RN$ as $m \to \infty$, 
we have $u_{m,0} \to u_S$ strongly in $L^2_{\rm loc} (\RN)$. 
What remains to prove is 
	\begin{equation}\label{eq:4.24}
		\lim_{m \to \infty} 
		\int_{B_R \times B_R} \frac{ \left( v_m (x) - v_m (y) \right)^2 }{|x-y|^{N+2s}} \, dx dy 
		= 0 \quad \text{for each $R>0$},
	\end{equation}
where $v_m (x) := u_S(x) - u_{m,0} (x)$.

	Let $R > 0$ be given and choose a radial cut-off function $\varphi_R \in C^\infty_c(\RN)$ so that 
	\begin{equation}\label{eq:4.25}
		\varphi_R \equiv 1 \quad \text{on} \ B_{R}, \quad 
		\varphi_R \equiv 0 \quad \text{on} \  B_{2R}^c.
	\end{equation}
Since $\varphi_R v_m \in H^s(\RN)$ with compact support, we obtain 
	\begin{equation}\label{eq:4.26}
		\la v_m , \varphi_R v_m \ra_{\dHs (\RN)} 
		=
		\int_{\RN} |x|^\ell \left( u_S^p - u_{m,0}^p \right) \varphi_R v_m \, d x.
	\end{equation}
By $v_m (x) \to 0$ for each $x \in \RN \setminus \{0\}$ and the dominated convergence theorem, 
	\begin{equation}\label{eq:4.27}
		\lim_{m \to \infty} \int_{\RN} |x|^\ell \left( u_S^p - u_{m,0}^p \right) \varphi_R v_{m,0} \, d x = 0.
	\end{equation}

	To compute the left-hand side of \eqref{eq:4.26}, we decompose $\RN$ into 
$\RN = B_R \cup A_{R,2R} \cup B_{2R}^c$, 
where $A_{R,2R} := \Set{ x \in \RN | R \leq |x| < 2R }$. 
It follows from \eqref{eq:4.25} that 
	\begin{equation}\label{eq:4.28}
		\begin{aligned}
			&\la v_m, \varphi_R v_m \ra_{\dHs(\RN)} 
			\\
			= \ &
			\int_{B_R \times B_R} \frac{ \left( v_m (x) - v_m (y) \right)^2 }{|x-y|^{N+2s}} \, dx dy 
			\\
			& \quad 
			+ \int_{ A_{R,2R} \times A_{R,2R}} 
			\frac{ \left(  v_m (x) - v_m (y) \right) \left( (\varphi_R v_m) (x) - (\varphi_R v_m) (y) \right)}
			{|x-y|^{N+2s}} \, dx dy
			\\
			&\quad + 
			2 \left( \int_{ B_R \times A_{R,2R} } + \int_{B_R \times B_{2R}^c} 
			+ \int_{ A_{R,2R} \times B_{2R}^c } \right) 
			\frac{ \left(  v_m (x) - v_m (y) \right) \left( (\varphi_R v_m) (x) - (\varphi_R v_m) (y) \right)}
			{|x-y|^{N+2s}} \, dx dy
			\\
			=: \ & 	\int_{B_R \times B_R} \frac{ \left( v_m (x) - v_m (y) \right)^2 }{|x-y|^{N+2s}} \, dx dy 
			+  I_{m,1} +2 \left( I_{m,2} + I_{m, 3} + I_{m,4} \right).
		\end{aligned}
	\end{equation}

	We shall prove $I_{m,j} \to 0$ as $m  \to \infty$ for $j=1,2,3,4$. 
To this end, by \cite[Proposition 3.2 and Lemma 3.3]{FF-14} (\cite[Proposition 2.1]{HIK-20}), we remark that 
$(v_{m})_{m \geq 1}$ is bounded in $C^{1,\beta}_{\rm loc} (B_{R/2}^c)$ for some $\beta > 0$. 
From $v_m(x) \to 0$ for any $x \in \RN \setminus \set{0}$, we obtain 
	\begin{equation}\label{eq:4.29}
		v_m \to 0 \quad \text{in} \ C^1_{\rm loc} \left( B_{R/2}^c \right).
	\end{equation}
Since $v_m$ and $\varphi_R$ are radially symmetric, for each $x \in A_{R/2,2R}$ and $y \in A_{R/2,2R}$, we have 
	\begin{equation}\label{eq:4.30}
		\begin{aligned}
			\left| v_m (x) - v_m (y) \right| 
			= \left| v_m \left( |x| \bm{e}_1 \right) - v_m \left( |y| \bm{e}_1 \right) \right| 
			&\leq \left\| v_m \right\|_{C^1( A_{R/2,2R} )} \left| \left|x\right| - \left| y \right| \right| 
			\\
			&\leq \left\| v_m \right\|_{C^1( A_{R/2,2R} )} \left| x - y \right|
		\end{aligned}
	\end{equation}
and 
	\begin{equation}\label{eq:4.31}
		\left| \varphi_R (x) v_m (x) - \varphi_R(y) v_m (y) \right| 
		\leq \| \varphi_R v_m \|_{C^1 (A_{R/2,2R}) } \left| x - y \right|.
	\end{equation}
Using \eqref{eq:4.30} and \eqref{eq:4.31}, we obtain 
	\begin{equation}\label{eq:4.32}
		\left| I_{m,1} \right| 
		\leq 
		\left\| v_m \right\|_{C^1( A_{R/2,2R} )} \left\| \varphi_R v_m \right\|_{C^1( A_{R/2,2R} )} 
		\int_{ A_{R,2R} \times A_{R,2R}} |x-y|^{-N+2-2s} \, \rd x \rd y = o(1).
	\end{equation}

	Next, by $v_m (x) \to 0$ and $|v_m (x)| \leq u_S(x)$ for each $x \in \RN \setminus \set{0}$, 
\eqref{eq:4.29}, \eqref{eq:4.30} and the dominated convergence theorem, 
we obtain 
	\begin{equation}\label{eq:4.33}
		\begin{aligned}
		\left| I_{m,2} \right| 
		& \leq 
		\int_{ B_R \times A_{R,2R} } 
		\frac{\left| v_m (x) - v_m (y) \right| \left| (\varphi_R  v_m) (x) - (\varphi_R  v_m) (y) \right| }
		{ |x-y|^{N+2s} } \, dx dy
		\\
		& = 
		\left( \int_{ B_{R/2} \times A_{R,2R} } + \int_{ A_{R/2, R} \times A_{R,2R} } \right)
		\frac{\left| v_m (x) - v_m (y) \right| \left| (\varphi_R  v_m) (x) - (\varphi_R  v_m) (y) \right| }
		{ |x-y|^{N+2s} } \, dx dy
		\\
		& \leq 
		\left( \frac{R}{2} \right)^{N+2s} 
		\int_{ B_{R/2} \times A_{R,2R} } \left| v_m (x) - v_m (y) \right| 
		\left| (\varphi_R  v_m) (x) - (\varphi_R  v_m) (y) \right| dxdy
		\\
		&\qquad 
		 + \| v_m \|_{C^1 (A_{R/2, 2R } ) } \| \varphi_R v_m \|_{C^1 ( A_{R/2,2R} ) } 
		\int_{ A_{R/2, R} \times A_{R,2R} } |x-y|^{2-N+2s} \, dx dy 
		\\
		& = o(1).
		\end{aligned}
	\end{equation}

	For $I_{m,3}$, the fact $\| v_m \|_{L^\infty (B_{2R}^c) } \leq \| u_S \|_{L^\infty (B_{2R}^c)} \leq C_R$ and 
	the dominated convergence theorem yield 
	\begin{equation}\label{eq:4.34}
		\begin{aligned}
			\left| I_{m,3} \right| 
			&\leq \int_{B_R \times B_{2R}^c} \frac{ \left| v_m (x) - v_m (y) \right| v_m (x) }
			{|x-y|^{N+2s}} \, dx dy 
			\\
			&\leq \int_{B_R} \left[ \left\{ v_m (x) + \|v_m \|_{L^\infty (B_{2R}^c) }  \right\} v_m (x) 
			\int_{B_{2R}^c} |x-y|^{-N-2s} \, dy
			\right] dx
			\\
			& \leq C_R \int_{B_R} \left\{ v_m (x) + 1  \right\} v_m (x) \,dx 
			= o(1).
		\end{aligned}
	\end{equation}

	For $I_{m,4}$, we have 
	\[
		\begin{aligned}
			\left| I_{m,4} \right| 
			& \leq \int_{ A_{R,2R} \times B_{2R}^c } 
			\frac{ \left| v_{m} (x) - v_{m} (y) \right| \left| (\varphi_R  v_m) (x) - (\varphi_R  v_m) (y) \right|}
			{|x-y|^{N+2s}} 
			\, dx dy 
			\\
			&= \left( \int_{A_{R,2R} \times A_{2R,3R}  }  + \int_{A_{R,2R} \times B_{3R}^c } \right) 
			\frac{ \left| v_{m} (x) - v_{m} (y) \right| \left| (\varphi_R v_m) (x) - (\varphi_R  v_m) (y) \right|  }
			{|x-y|^{N+2s}}
			\, dx dy.
		\end{aligned}
	\]
We remark that \eqref{eq:4.30} and \eqref{eq:4.32} hold for $x \in A_{R,2R}$ and $y \in A_{2R,3R}$ 
with $\| v_m \|_{C^1( A_{R,3R} ) }$ and $\| \varphi_R v_m \|_{C^1(A_{R,3R})}$, hence, 
	\[
		\int_{A_{R,2R} \times A_{2R,3R}  } 
		\frac{ \left| v_{m} (x) - v_{m} (y) \right| \left| (\varphi_R v_m )(x) - (\varphi_R  v_m) (y) \right|  }
		{|x-y|^{N+2s}} \, dx dy 
		= o(1).
	\]
On the other hand, as in \eqref{eq:4.34}, we may obtain 
	\[
		\begin{aligned}
			& \int_{A_{R,2R} \times B_{3R}^c } 
			\frac{ \left| v_{m} (x) - v_{m} (y) \right| \left| (\varphi_R v_m) (x) - (\varphi_R v_m) (y) \right|  }
			{|x-y|^{N+2s}} \, dx dy 
			\\
			= \ & 
			\int_{ A_{R,2R} \times B_{3R}^c}
			\frac{ \left| v_{m} (x) - v_{m} (y) \right| \varphi_R (x) v_m(x)   }
			{|x-y|^{N+2s}} \, dx dy 
			= o(1).
		\end{aligned}
	\]
Thus, 
	\begin{equation}\label{eq:4.35}
		I_{m,4} = o(1).
	\end{equation}

	From \eqref{eq:4.26}--\eqref{eq:4.28} and \eqref{eq:4.32}--\eqref{eq:4.35}, it follows that 
	\[
		\int_{B_R \times B_R} \frac{ \left( v_m (x) - v_m (y) \right)^2 }{|x-y|^{N+2s}} \, dx dy = o(1)
	\]
and \eqref{eq:4.24} holds. Thus, we complete the proof. 
\end{proof}

Finally, as a corollary of Proposition \ref{Proposition:4.1}, we obtain the following:
\begin{Corollary}\label{Corollary:4.2}
	\begin{enumerate}
	\item[{\rm (i)}]	
		Let $u$ be a positive radial stable solution of \eqref{eq:1.1}
		and $u(x) \to 0$ as $|x| \to \infty$. 
			\begin{enumerate}
			\item[{\rm (a)}] 
				If $u \in L^\infty (\RN)$, then $u (x) = u_{\wt{m},0}(x)$ holds for $\wt{m} = \| u \|_{L^\infty (\RN)}$. 
			\item[{\rm (b)}] 
				If $u \in L^\infty_{\rm loc} (\RN \setminus \set{0} )$, $u \leq C u_S$ in $B_1$ for some $C>0$ 
				and $u(x) \to \infty$ as $|x| \to 0$, 
				then $u_S \leq u$ in $\RN$. 
			\end{enumerate}
	\item[{\rm (ii)}] 
	If $u(x) \in H^s_{\rm loc} (\RN) \cap L^\infty_{\rm loc} (\RN \setminus \set{0} ) $ 
	is a positive radial solution of \eqref{eq:1.1} with $u \leq u_S$ in $\RN$, 
	then either $u = u_{\wt{m},0}$ with $\wt{m} =  \| u \|_{L^\infty (\RN)} $ or else $u=u_S$. 
	\end{enumerate}
\end{Corollary}

\begin{proof}
(i) We first consider the case (a). In this case, we remark that $u \in C(\RN)$ and $u(0) > 0$. 
	Let $(u_{m,0})_{m>0}$ be as in Corollary \ref{Corollary:3.1}. 
	By definition and Corollary \ref{Corollary:4.1}, 
	a map $m \mapsto u_{m,0} (0) : (0,\infty) \to (0,\infty)$ is continuous and strictly increasing. Define $\wt{m}$ by 
	\[
		\wt{m} := \inf \Set{ m \in (0,\infty) | u(0) < u_{m,0} (0) }
		= \sup \Set{ m \in (0,\infty) | u(0) > u_{m,0} (0) }. 
	\]
	For $m_1 < \wt{m} < m_2$, notice that $u_{m_1,0} (0) < u(0) = u_{\wt{m},0} (0) < u_{m_2,0} (0) $. 
	Applying Proposition \ref{Proposition:4.1} for the pairs $(u_{m_1,0}, u)$ and $(u,u_{m_2,0})$, we obtain 
	$u_{m_1,0} \leq u \leq u_{m_2,0}$ in $\RN$ for every $m_1 < \wt{m} < m_2$. 
	Thus, letting $m_1 \nearrow \wt{m}$, $m_2 \searrow \wt{m}$ and noting 
	$u_{m,0}(x) \to u_{\wt{m},0}(x)$ as $m \to \wt{m}$, we have 
	$u_{\wt{m},0} \leq u \leq u_{\wt{m},0}$ in $\RN$, which yields $u \equiv u_{\wt{m},0}$. 
	Hence, $\wt{m} = \| u_{\wt{m},0} \|_{L^\infty(\RN)} = \| u \|_{L^\infty(\RN)}$.

	Next, we consider the case (b). We apply Proposition \ref{Proposition:4.1} for $u_1 = u_{m,0}$ and $u_2 = u$ to get 
	$u_{m,0} \leq u$ in $\RN$ for each $m>0$. 
	Thus, by letting $m \to \infty$, Proposition \ref{Proposition:4.2} gives $u_S \leq u$ in $\RN$.

(ii) If $u \in H^s_{\rm loc} (\RN) \cap L^\infty_{\rm loc} (\RN \setminus \set{0} ) $ is 
	a positive radial solution of \eqref{eq:1.1} with $u \leq u_S$ in $\RN$, then 
	$u$ becomes stable since 
	\[
		p \int_{ \RN } |x|^\ell u^{p-1} \varphi^2 \, \rd x 
		\leq p \int_{ \RN } |x|^\ell u_S^{p-1} \varphi^2 \, \rd x 
		\leq \| \varphi \|_{\dHs (\RN)}^2 \quad 
		\text{for every $\varphi \in C^\infty_c (\RN)$}.
	\]
Hence, (i) yields $u = u_{m,0}$ when $u \in L^\infty(\RN)$ and $u_S \leq u$ when $u(x) \to \infty$ as $|x| \to 0$. 
In the latter case, by the assumption $u \leq u_S$, we obtain $u=u_S$. 
\end{proof}

	Now we complete the proof of Theorem \ref{Theorem:1.1} (ii):
\begin{proof}[Proof of Theorem {\rm\ref{Theorem:1.1}} {\rm(ii)}]
	Let $(u_{m,0})_{m>0}$ be the family of positive radial stable solutions of \eqref{eq:1.1} 
	obtained in Corollary \ref{Corollary:3.1}. By Remark \ref{Remark:3.6}, this family may be rewritten as 
	$(u_{\alpha})_{\alpha>0}$ where $\alpha = u_\alpha(0) = m \wt{u}_0(0)$. 
	Then properties (a)--(e) clearly follow from Corollaries \ref{Corollary:3.1}, \ref{Corollary:4.1}, \ref{Corollary:4.2} 
	and Proposition \ref{Proposition:4.2}. Thus, Theorem \ref{Theorem:1.1} (ii) holds. 
\end{proof}



\section{Proof of Theorems \ref{Theorem:1.2} and \ref{Theorem:1.3}, and Corollary \ref{Corollary:1.1}}
\label{section:5}


In this section, we prove Theorems \ref{Theorem:1.2} and \ref{Theorem:1.3}, and Corollary \ref{Corollary:1.1}. 
Throughout this section, we always assume \eqref{eq:1.2}.


\subsection{Preliminaries}\label{section:5.1}


We first recall the definition of $\lambda (\alpha)$: 
\[
	\lambda (\alpha) := 
	2^{2s} \frac{\Gamma \left( \frac{N+2s}{4} + \frac{\alpha}{2} \right)
		\Gamma \left( \frac{N+2s}{4} - \frac{\alpha}{2} \right)}
	{\Gamma \left( \frac{N-2s}{4} - \frac{\alpha}{2} \right)
		\Gamma \left( \frac{N-2s}{4} + \frac{\alpha}{2} \right) }.
\]
By Definition \ref{Definition:1.2}, we only consider the case 
\begin{equation}\label{eq:5.1}
	p_S(N,\ell) = \frac{N+2s+2\ell}{N-2s} < p < \infty. 
\end{equation}
Under \eqref{eq:5.1}, $p$ is JL-subcritical (resp. JL-supercritical) if and only if 
\begin{equation}\label{eq:5.2}
	p \lambda \left( \beta_{\ell,N,p,s}  \right) > \lambda (0)
	\qquad 
	\left( \text{resp.} \quad  p \lambda \left( \beta_{\ell,N,p,s}  \right) < \lambda (0)  \right), \quad 
	\beta_{\ell,N,p,s}:= \frac{N-2s}{2} - \frac{2s+\ell}{p-1}.
\end{equation}

Next, we change the variable from $p$ to $x$ as follows. 
Noting that $\beta_{\ell,N,p,s}$ is strictly increasing in $p \in (p_{S} (N,\ell), \infty)$
and 
\[
	\lim_{p \to p_S(N,\ell)} \beta_{\ell,N,p,s} = 0, \quad \lim_{p \to \infty} \beta_{\ell,N,p,s} = \frac{N-2s}{2}, 
\]
we set 
\[
	A_{N,s} := \frac{N-2s}{4}, \quad x := \frac{\beta_{\ell,N,p,s}}{2} \in \left( 0, A_{N,s} \right). 
\]
Then $p$ is expressed as 
\[
	p = \frac{N+2s+2 \ell - 4x}{N-2s - 4x}.
\]
It is also convenient to use $A_{N,s}$ instead of $N$ and notice that 
\[
	\begin{aligned}
		 & \frac{N+2s}{4} = A_{N,s} + s, \quad N-2s - 4x = 4A_{N,s} -4x = 4(A_{N,s}-x), 
		\\
		 & N+2s + 2\ell -4x = 4 ( A_{N,s} - x ) + 4s + 2\ell.
	\end{aligned}
\]
Using $x$, we shall study the validity of the following inequality for $x \in (0,A_{N,s})$ instead of \eqref{eq:5.2}: 
\begin{equation}\label{eq:5.3}
	g_{\ell}(x) 
	> \frac{ \left(\Gamma (A_{N,s} + s) \right)^2}{ \left( \Gamma (A_{N,s}) \right)^2} 
	=: \wt{M}_{N,s}
	\qquad 
	\left( \text{resp.} \quad
	g_{\ell}(x) < \wt{M}_{N,s} \right),
\end{equation}
where 
\begin{equation}\label{eq:5.4}
	\begin{aligned}
		g_{\ell}(x) 
		:= \  & \frac{ A_{N,s} -x + s + \frac{\ell}{2} }{A_{N,s}-x}
		\cdot \frac{ \Gamma ( A_{N,s} + s - x ) \Gamma (A_{N,s} + s + x ) }
		{\Gamma (A_{N,s} - x ) \Gamma ( A_{N,s} + x ) } 
		\\
		= \   & \left( A_{N,s} -x + s + \frac{\ell}{2} \right) \frac{ \Gamma ( A_{N,s} + s - x ) \Gamma (A_{N,s} + s + x ) }
		{\Gamma (A_{N,s} + 1 - x  ) \Gamma ( A_{N,s} + x ) } \in C \left( \left[  0, A_{N,s} \right] \right).
	\end{aligned}
\end{equation}
Remark that \eqref{eq:5.3}
is equivalent to the case where $p$ is JL-subcritical (resp. JL-supercritical). 
In addition, $x=0$ (resp. $x=A_{N,s}$) corresponds to $p=p_S(N,\ell)$ (resp. $p=\infty$). 
It is also easily seen that  
	\begin{equation}\label{eq:5.5}
		g_{0}(x) = 
		\frac{ \Gamma ( A_{N,s} + 1+ s - x ) \Gamma (A_{N,s} + s + x ) }
		{\Gamma (A_{N,s} + 1 - x  ) \Gamma ( A_{N,s} + x ) }
	\end{equation}
and 
	\begin{equation}\label{eq:5.6}
		g_{\ell} (x) = \frac{ A_{N,s} - x + s + \frac{\ell}{2}  }{ A_{N,s} - x + s } g_{0}(x).
	\end{equation}

We first prove that $p$ is JL-subcritical if $p$ is close to $p_S(N,\ell)$. 

\begin{lemma}\label{Lemma:5.1}
	There exists a $ p_0 = p_0(\ell,N,p,s) \in (p_S(N,\ell) , \infty) $ such that 
	each $ p \in (p_S(N,\ell) , p_{0})$ is JL-subcritical.  
\end{lemma}

\begin{proof}
	It is easily seen from \eqref{eq:1.2} that 
	\[
		\lim_{x \searrow 0} \frac{A_{N,s} - x + s+ \frac{\ell}{2} }{A_{N,s}-x} =
		\frac{A_{N,s} + s+ \frac{\ell}{2} }{A_{N,s}} = p_{S} (N,\ell) > 1. 
	\]
	Furthermore, 	
	\[
		\lim_{x \searrow 0}  \frac{ \Gamma ( A_{N,s} + s - x ) \Gamma (A_{N,s} + s + x ) }
		{\Gamma (A_{N,s} - x ) \Gamma ( A_{N,s} + x ) }
		= \frac{ \left( \Gamma (A_{N,s} + s) \right)^2 }{ \left( \Gamma (A_{N,s}) \right)^2 }. 
	\]
	From these two facts with \eqref{eq:5.4}, 
	we may find an $x_0 = x_0(\ell,N,p,s) \in (0,A_{N,s})$ such that 
	\eqref{eq:5.3} holds for any $x \in (0,x_{0})$. From the change of variables, 
	we see that Lemma \ref{Lemma:5.1} holds. 
\end{proof}

	Let $-2s<\ell\le0$ and set
	\begin{equation}\label{eq:5.7}
		\begin{aligned}
			h_{\ell}(x) := \  & \log g_{\ell}(x) 
			\\
			= \          & \log \left( A_{N,s} + s + \frac{\ell}{2} - x \right) + 
			\log \Gamma (A_{N,s}  + s -x ) - \log \Gamma (A_{N,s} + 1  -x)
			\\
			             & \quad + \log \Gamma (A_{N,s} + s + x) - \log \Gamma (A_{N,s} + x) .
		\end{aligned}
	\end{equation}
	Then we see that \eqref{eq:5.3} is equivalent to 
	\begin{align}
	 & h_{\ell}(x) > 2 \left[ \log \Gamma \left( A_{N,s} +s \right) - \log \Gamma \left( A_{N,s} \right) \right]
	\label{align:5.8}
	\\
	 & \left( \text{resp.} \quad h_{\ell}(x) 
	< 2 \left[ \log \Gamma \left( A_{N,s} + s\right) - \log \Gamma \left( A_{N,s} \right) \right]   \right).
	\label{align:5.9}
	\end{align}
	We next show that $h_{\ell}(x)$ is strictly concave in $[0,A_{N,s}]$. 
	This is helpful to see the range of $x$, where \eqref{align:5.8} or \eqref{align:5.9}, namely \eqref{eq:5.3}, holds.

\begin{lemma}\label{Lemma:5.2}
	Let $-2s < \ell \le 0$ and $h_\ell$ be as in \eqref{eq:5.7}.
		\begin{enumerate}
		\item[{\rm (i)}]
		The function $h_\ell$ is strictly concave in $[0,A_{N,s}]$.
		\item[{\rm (ii)}] 
		Suppose $\ell = 0$. 
		\begin{enumerate}
		\item[{\rm (a)}]
			If $A_{N,s} > 1/2$, then $h_{0}'(x) > 0$ for each $x \in (0, 1/2 )  $,  $h_{0}'( 1/2 ) = 0$ and 
			$h_{0}'(x) < 0$ for each $x \in ( 1/2 , A_{N,s} )$.
		\item[{\rm (b)}] 
			If $A_{N,s} \leq  1/2 $,
			then $h_{0}'(x) > 0$ for all $x \in (0,A_{N,s})$. 
		\end{enumerate}
		\end{enumerate}

\end{lemma}

\begin{proof}
	(i)
	Denote by $\psi (z)$ the digamma function, that is
	\begin{equation}\label{eq:5.10}
		\psi(z) := \frac{d}{d z} \log \Gamma (z) = \frac{\Gamma'(z)}{\Gamma(z)}.
	\end{equation}
	By \eqref{eq:5.7} and \eqref{eq:5.10} it holds that 
	\begin{equation*}
		\begin{aligned}
			h_{\ell}'(x) & = - \frac{1}{A_{N,s} + s + \ell/2 - x} - \psi (A_{N,s} + s - x)
			+ \psi (A_{N,s}+1-x)
			\\
			        & \qquad + \psi(A_{N,s}+s+x) - \psi(A_{N,s}+x)
		\end{aligned}
	\end{equation*}
	and 
	\begin{equation}\label{eq:5.11}
		\begin{aligned}
			h_{\ell}''(x) & =  - \frac{1}{(A_{N,s}+s+\ell/2 - x)^2}
			+ \psi'(A_{N,s} +s-x) - \psi'(A_{N,s} + 1 -x) 
			\\
			         & \qquad + \psi'(A_{N,s}+s+x) - \psi'(A_{N,s}+x).
		\end{aligned}
	\end{equation}
	For $\psi'(z)$, we have the following expression (see, e.g., \cite{A-64,L-72}) 
	\begin{equation}\label{eq:5.12}
		\psi'(z) = \sum_{n=0}^\infty \frac{1}{(z+n)^2}.
	\end{equation}
	Thus, for $x \in (0,A_{N,s})$, it follows that 
	\[
		\begin{aligned}
			\psi'(A_{N,s} + s+x ) - \psi'(A_{N,s}+x) 
			=  \sum_{n=0}^\infty 
			\left[ \frac{1}{(A_{N,s} + s  +x + n )^2} - \frac{1}{(A_{N,s}+x+n)^2} \right] < 0.
		\end{aligned}
	\]
	For the first three terms in \eqref{eq:5.11}, since $-2s < \ell \leq 0$, we observe that 
	\[
		\begin{aligned}
			     & - \frac{1}{(A_{N,s} + s + \ell/2-x)^2} + \psi' (A_{N,s}+s-x) - \psi'(A_{N,s} + 1- x) 
			\\
			= \, & - \frac{1}{(A_{N,s} + s + \ell/2 - x )^2} + \sum_{n=0}^\infty \frac{1}{(A_{N,s}+s-x+n)^2}
			- \sum_{n=0}^\infty \frac{1}{(A_{N,s}+1-x+n)^2}
			\\
			= \, & - \frac{1}{(A_{N,s}+s+\ell/2-x)^2} + \frac{1}{(A_{N,s}+s-x)^2}
			\\
			     & \quad + \sum_{n=0}^\infty \left\{ \frac{1}{(A_{N,s}+s-x+1+n)^2} - \frac{1}{(A_{N,s}+1-x+n)^2} \right\}
			\\
			< \, & 0.
		\end{aligned}
	\]
	Hence, $h_{\ell}''(x) < 0$ for each $x \in (0,A_{N,s})$ when $-2 s < \ell \leq 0$ and $h_{\ell}$ is strictly concave. 

	(ii)
	We assume $\ell = 0$ and compute $h_0'(x)$. 
	To this end, we use the following expression for $\psi(z)$
	(see, e.g., \cite{A-64}):
	\begin{equation}\label{eq:5.13}
		\psi (z) = - \gamma + \sum_{n=0}^\infty \left( \frac{1}{n+1} - \frac{1}{n+z} \right), \quad 
		\gamma := \lim_{n \to \infty} \left( -\log n + \sum_{k=1}^n \frac{1}{k}  \right) \in (0,1).
	\end{equation}
	Remark that when $\ell = 0$, by \eqref{eq:5.5}, we have 
	\begin{equation}\label{eq:5.14}
		h_0'(x) = - \psi ( A_{N,s} + s + 1 - x ) + \psi(A_{N,s} + 1 - x) + \psi (A_{N,s} + s + x) - \psi (A_{N,s} + x).
	\end{equation}
Since 
	\begin{equation*}
\begin{aligned}
&- \psi (A_{N,s} +1+ s - x) + \psi (A_{N,s}+1-x)\\
= \ &\sum_{n=0}^\infty \left( \frac{1}{n+A_{N,s} +1+ s - x} - \frac{1}{n+A_{N,s}+1-x} \right)\\
= \ &\sum_{n=0}^\infty \frac{-s}{(n+A_{N,s} +1+ s - x)(n+A_{N,s}+1-x)},
\end{aligned} 
	\end{equation*}
and
	\begin{equation*}
	\begin{aligned}
	& \psi(A_{N,s}+s+x) - \psi(A_{N,s}+x)
	\\
	= \ & \sum_{n=0}^\infty \left( -\frac{1}{n+A_{N,s} + s + x} + \frac{1}{n+A_{N,s}+x} \right)\\
	= \ & \sum_{n=0}^\infty \frac{s}{(n+A_{N,s} + s + x)(n+A_{N,s}+x)},
	\end{aligned} 
	\end{equation*}
it follows from \eqref{eq:5.14} that 
	\begin{equation}\label{eq:5.15}
		\begin{aligned}
			&\frac{h_0'(x)}{s} 
			\\
			= & \sum_{n=0}^\infty 
			\frac{(n+A_{N,s} +1+ s - x)(n+A_{N,s}+1-x)-(n+A_{N,s} + s + x)(n+A_{N,s}+x)}{(n+A_{N,s} +1+ s - x)(n+A_{N,s}+1-x)(n+A_{N,s} + s + x)(n+A_{N,s}+x)}.
		\end{aligned}
	\end{equation}
Then we compute the numerator in \eqref{eq:5.15}, that is
	\begin{equation}\label{eq:5.16}
	\begin{aligned}
	&(n+A_{N,s} +1+ s - x)(n+A_{N,s}+1-x)-(n+A_{N,s} + s + x)(n+A_{N,s}+x)\\
	= \ &(-2x+1)(2n+2A_{N,s}+1+s). 
	\end{aligned}
	\end{equation}
Hence, from \eqref{eq:5.15} and \eqref{eq:5.16}, 
(a) and (b) clearly hold. 
\end{proof}


\subsection{Analysis at $x=A_{N,s}$}
\label{section:5.2}


Let $h_\ell$ be as in \eqref{eq:5.7}. 
Then,
by Lemmata \ref{Lemma:5.1} and \ref{Lemma:5.2}, when $-2s < \ell \leq 0$, 
in order to determine the range of $x$ (namely, $p$), where \eqref{eq:5.3} holds, 
it suffices to see the validity of \eqref{align:5.8} or \eqref{align:5.9} at $x = A_{N,s}$. 
In what follows, we analyze the cases $N=1$ and $N \geq 2$ separately 
since the range of $s$ is different in these two cases.

	Suppose $N=1$. In this case, we have $A_{1,s} = 1/4 - s/2$ and \eqref{align:5.8} 
and \eqref{align:5.9} with $x=A_{1,s}$ become
\begin{align}
	&\log \left( s + \frac{\ell}{2} \right) + \log \Gamma (s) + \log \Gamma \left( \frac{1}{2} \right)
	- \log \Gamma \left( \frac{1}{2} - s \right)
	- 2 \log \Gamma \left( \frac{1}{4} + \frac{s}{2} \right) + 2 \log \Gamma \left( \frac{1}{4} - \frac{s}{2} \right)
	> 0,
	\label{align:5.17}
	\\
		&\log \left( s + \frac{\ell}{2} \right) + \log \Gamma (s) + \log \Gamma \left( \frac{1}{2} \right)
		- \log \Gamma \left( \frac{1}{2} - s \right)
		- 2 \log \Gamma \left( \frac{1}{4} + \frac{s}{2} \right) + 2 \log \Gamma \left( \frac{1}{4} - \frac{s}{2} \right)
		< 0.
		\label{align:5.18}
\end{align}
For $(\ell,s)$, notice that $-2s < \ell \leq 0$ and $0<s< 1/2$ are equivalent to 
$- 1 < \ell \leq 0$ and $- \ell/2 < s < 1/2$. 
To analyze \eqref{align:5.17}, we first fix an $\ell \in (-1,0]$ and vary $s \in (-\ell/2,1/2)$:

\begin{lemma}\label{Lemma:5.3}
	Assume $N=1$, fix an $\ell \in (-1,0]$ and put $s_\ell := - \ell /2$. 
	Then there exists an $ \bar{s}_\ell \in [s_\ell , 1/2)$ such that 
	\eqref{align:5.17} holds if and only if $ s \in ( \bar{s}_\ell , 1/2 )$. 
	Moreover, when $\ell = 0$, $\bar{s}_\ell = 0 = s_\ell$, and 
	when $ - 1 < \ell < 0$, $ \bar{s}_\ell > s_\ell$ and \eqref{align:5.18} holds for $s \in (s_\ell, \bar{s}_\ell)$. 
\end{lemma}


\begin{proof}
	For $s \in (s_\ell , 1/2)$, let us consider 
	\[
		\begin{aligned}
			H_{\ell}(s) := & \log \left( s + \frac{\ell}{2} \right)
			+\log \Gamma (s) + \log \Gamma \left( \frac{1}{2} \right)
			- \log \Gamma \left( \frac{1}{2} - s \right)
			\\
			          & \quad 
			- 2 \log \Gamma \left( \frac{1}{4} + \frac{s}{2} \right)
			+ 2 \log \Gamma \left( \frac{1}{4} - \frac{s}{2} \right).
		\end{aligned}
	\]
	Note that when $\ell = 0$, we have 
	\begin{equation}\label{eq:5.19}
		s_\ell = 0, \quad \log \left( s + \frac{\ell}{2} \right) + \log \Gamma (s)
		= \log s \Gamma (s) = \log \Gamma (s+1), \quad H_0(0) = 0.
	\end{equation}
	On the other hand, when $\ell <0$, we get $H_{\ell}(s) \to - \infty$
	as $s \searrow s_\ell $.

	Next, we compute $H_{\ell}'(s)$. 
	It follows from \eqref{eq:5.13} that for $s \in (s_\ell,1/2)$
	\[
		\begin{aligned}
			H_{\ell}'(s) & = \frac{1}{s+\frac{\ell}{2}} + \psi(s) + \psi \left( \frac{1}{2} - s \right)
			- \psi \left( \frac{1}{4} + \frac{s}{2} \right) - \psi \left( \frac{1}{4} - \frac{s}{2} \right)
			\\
			        & = \frac{1}{s+\frac{\ell}{2}}
			+ \sum_{n=0}^\infty \left[ - \frac{1}{n+s} - \frac{1}{n+\frac{1}{2}-s}
				+ \frac{1}{n+\frac{1}{4}+ \frac{s}{2}} + \frac{1}{n+ \frac{1}{4} - \frac{s}{2}} \right]
			\\
			        & = \frac{1}{s+\frac{\ell}{2}} - \frac{1}{s}
			+ \sum_{n=0}^\infty  \left[  - \frac{1}{n+1+s} - \frac{1}{n+\frac{1}{2}-s}
				+ \frac{1}{n+\frac{1}{4}+ \frac{s}{2}} + \frac{1}{n+ \frac{1}{4} - \frac{s}{2}} \right].
		\end{aligned}
	\]
	Since $ - 1 < \ell \leq 0$, we see that 
	\[
		\frac{1}{s+\frac{\ell}{2}} - \frac{1}{s} \geq 0.
	\]
	In addition, one sees that 
	\[
		\frac{1}{n+\frac{1}{4}+\frac{s}{2}} - \frac{1}{n+1+s} > 0, \quad 
		\frac{1}{n+\frac{1}{4}-\frac{s}{2}} - \frac{1}{n+\frac{1}{2}-s}
		= \frac{1}{n+ \frac{1}{4}-\frac{s}{2} } - \frac{1}{n + 2 \left( \frac{1}{4} - \frac{s}{2} \right) } > 0.
	\]
	Therefore, 
	\begin{equation}\label{eq:5.20}
		H_{\ell}'(s) > 0 \quad \text{for every $s \in \left(s_\ell,\frac{1}{2}\right)$}.
	\end{equation}
	In particular, by \eqref{eq:5.19}, when $\ell = 0$, 
	we obtain $H_{\ell}(s) > 0$ for each $s \in (0,1/2)$ and 
	\eqref{align:5.17} holds for each $s \in (0,1/2)$.

	To show \eqref{align:5.17} in the case $ \ell < 0$, 
	we observe the behavior of $H_{\ell}(s)$ as $s \nearrow 1/2$. For this purpose, we use the following expansion: 
	\begin{equation}\label{eq:5.21}
		\log \Gamma (z) = - \log z + O(1) \quad \text{as $z \searrow 0$} .
	\end{equation}
	As $s \nearrow 1/2$, we obtain 
	\[
		\begin{aligned}
			 & -\log \Gamma \left( \frac{1}{2} - s \right) = \log \left( \frac{1}{2} - s \right) + O(1), 
			\\ 
			 & 2 \log \Gamma \left( \frac{1}{4} - \frac{s}{2} \right)
			= -2 \log \left( \frac{1}{2} \left( \frac{1}{2} - s \right) \right) + O(1) 
			= 2 \log 2 -2 \log \left( \frac{1}{2} - s \right) + O(1)
		\end{aligned}
	\]
	and 
	\[
		H_{\ell}(s) = - \log \left( \frac{1}{2} - s \right) + O(1) \to \infty \quad 
		\text{as } s \nearrow \frac{1}{2}.
	\]
	From these, we infer that if $\ell < 0$, then by \eqref{eq:5.20} and $H_{\ell}(s) \to - \infty$ as $s \searrow s_\ell$, 
	there exists a unique $\bar{s}_\ell \in (s_\ell, 1/2 )$
	such that $H_{\ell}(\bar{s}_\ell) =0$ and $H_{\ell}(s) > 0$ in $(\bar{s}_\ell , 1/2)$. This completes the proof. 
\end{proof}

	Next, we turn to the case $N \geq 2$. 
As in the above, we consider \eqref{align:5.8} at $x=A_{N,s}$ and set 
\[
	\begin{aligned}
		H_{\ell}(s,N) & := \log \left( s+ \frac{\ell}{2} \right) + \log \Gamma (s)
		+ \log \Gamma \left( \frac{N}{2} \right) - \log \Gamma \left( \frac{N-2s}{2} \right)
		\\
		         & \qquad 
		- 2 \log \Gamma \left( \frac{N+2s}{4} \right) + 2 \log \Gamma \left( \frac{N-2s}{4} \right).
	\end{aligned}
\]
We shall observe when the inequality $H_{\ell}(s,N) > 0$ holds, 
which is equivalent to \eqref{align:5.8} at $x=A_{N,s}$. 
Remark that in Lemmata \ref{Lemma:5.4} and \ref{Lemma:5.5}, 
we also treat not only $\ell \in (-2,0]$, but also $\ell > 0$.

\begin{lemma}\label{Lemma:5.4}
	Suppose $N \geq 2$ and $\ell \in (-2,\infty)$. 
	Then the following hold:
	\begin{enumerate}
		\item[{\rm (i)}]
		      Assume $ \ell \in (-2, 0]$ and set $s_\ell := -\ell/2$. 
		      If $\ell = 0$, then $H_{\ell}(s,N) \to 0$
		      as $s \searrow s_\ell$. 
		      On the other hand, if $- 2< \ell < 0$, 
		      then $ H_{\ell} (s,N) \to - \infty$ as $s \searrow s_\ell$;
		\item[{\rm (ii)}]
		      When $2 \leq N \leq 6$ and $-2 < \ell \leq 0$, $\frac{\partial H_{\ell}}{\partial s} (s,N) > 0$ 
		      for each $s \in (s_\ell,1)$;
		\item[\emph{(iii)}]
		      When $N \geq 6$ and $\ell \geq 0$, $H_{\ell} (\cdot ,N)$ is strictly convex in $ (0,1)$; 
		\item[{\rm (iv)}]
		      When $N=7$ and $\ell = 0$, $ \lim_{s \searrow 0} \frac{\partial H_{\ell}}{\partial s} (s,7) > 0$. 
		      When $N \geq 8$ and $\ell = 0$, there exists a $t_N  \in (0,1] $ such that 
		      $ \frac{\partial H_{\ell}}{\partial s} (s,N) < 0$ for $s \in (0,t_N)$ and 
		      $\frac{\partial H_{\ell}}{\partial s} (s,N) > 0$ if $t_N<1$ and $t_N < s < 1$;
		\item[{\rm (v)}] For each $\ell \in (-2,\infty)$, we have
		      \[
			      \lim_{s \nearrow 1 } H_{\ell} (s,N) = H_{\ell} (1,N) \left\{\begin{aligned}
				       & > 0 &  & \text{if} \ 2 \leq N < 10 + 4 \ell, \\
				       & = 0 &  & \text{if} \ N = 10 + 4 \ell,        \\
				       & < 0 &  & \text{if} \ 10 + 4 \ell < N.
			      \end{aligned}\right.
		      \]
	\end{enumerate}
\end{lemma}

\begin{remark}\label{Remark:5.1}
	When $s=1$, it is known that each $p \in (1,\infty)$ is JL-subcritical if and only if 
	$N \leq 10 + 4 \ell$ (see \eqref{eq:1.4} and \cite{DDG}). 
	Since $x=A_{N,s}$ corresponds to $p=\infty$, by Lemma \ref{Lemma:5.4} (v), 
	we see that the case $s=1$ and $p=\infty$ appear in the limit $s \nearrow 1$. 
\end{remark}

\begin{proof}[Proof of Lemma {\rm\ref{Lemma:5.4}}]
	(i) When $\ell = 0$, similarly to \eqref{eq:5.19},
	we have $H_{0} (\cdot, N) \in C([0,1 ) )$ and $H_{0}(0,N) = 0$.
	On the other hand, when $-2 < \ell < 0$, $\lim_{s \searrow s_\ell} H_{\ell} (s,N) = -\infty$ 
	due to the term $\log (s+\ell/2)$ and $s_\ell > 0$.

	(ii) Recalling \eqref{eq:5.13}, we compute $\frac{\partial H_{\ell}}{\partial s} (s,N)$:
		\begin{equation}\label{eq:5.22}
			\begin{aligned}
				\frac{\partial H_{\ell}}{\partial s} (s,N) 
				& = 
				\frac{1}{s+\frac{\ell}{2}} + \psi(s) + \psi \left( \frac{N-2s}{2} \right)
				- \psi \left( \frac{N+2s}{4} \right) - \psi \left( \frac{N-2s}{4} \right)
				\\
                & = 
                \frac{1}{s+\frac{\ell}{2}}
				+ \sum_{n=0}^\infty \left[ - \frac{1}{n+s} - \frac{1}{n + \frac{N-2s}{2} }
				+ \frac{1}{n+ \frac{N+2s}{4} } + \frac{1}{n+\frac{N-2s}{4}} \right]
				\\
                & = 
                \frac{1}{s+\frac{\ell}{2}} - \frac{1}{s}
				+ \sum_{n=0}^\infty 
				\left[ - \frac{1}{n+1+s} - \frac{1}{n + \frac{N-2s}{2}} + \frac{1}{n+ \frac{N+2s}{4}}
				+ \frac{1}{n + \frac{N-2s}{4}} \right].
			\end{aligned}
		\end{equation}
		Since $ - 2 < \ell \leq 0$, we have
			\[
				\frac{1}{s+\frac{\ell}{2}} - \frac{1}{s} \geq 0 \quad \text{for $s \in (s_\ell , 1)$}.
			\]
		On the other hand, we have 
			\[
				\begin{aligned}
				  & - \frac{1}{n+1+s} - \frac{1}{n + \frac{N-2s}{2}} + \frac{1}{n+ \frac{N+2s}{4}}
				+ \frac{1}{n + \frac{N-2s}{4}}
					\\
							= & \ \frac{ 1 + \frac{s}{2} - \frac{N}{4}  }{ \left\{ n + \frac{N+2s}{4} \right\} \left\{ n+1+s \right\} }
							+ \frac{\frac{N-2s}{4}}{ \left\{ n + \frac{N-2s}{4} \right\} \left\{ n + \frac{N-2s}{2} \right\} }
							\\
							= & \ 
							\frac{
								32n^2 + 4 n (N-2s) (-N+6s + 8) + (N-2s) \{  N(6-N) +  2s(3N-2) \}
							}
							{ 32 \left\{ n + \frac{N+2s}{4} \right\} \left\{ n+1+s \right\} \left\{ n + \frac{N-2s}{4} \right\}
								\left\{ n + \frac{N-2s}{2} \right\}   }.
						\end{aligned}
					\]
					From this expression, we infer that for any $\ell \in (-2,0]$ and $2 \leq N \leq 6$, 
	$\frac{\partial H_{\ell}}{\partial s} (s,N) > 0$ for every $s \in (s_\ell,1)$.

	(iii) Assume $N \geq 6$ and $\ell \geq 0$. Then it follows from \eqref{eq:5.12} and \eqref{eq:5.22} that 
	\[
		\begin{aligned}
			     & \frac{\partial^2 H_{\ell}}{\partial s^2} (s,N)
			\\
			= \  & - \frac{1}{\left( s+\frac{\ell}{2} \right)^2 } + \psi'(s) - \psi' \left( \frac{N-2s}{2} \right)
			- \frac{1}{2} \psi' \left( \frac{N+2s}{4} \right) + \frac{1}{2} \psi' \left( \frac{N-2s}{4} \right)
			\\
			= \  & - \frac{1}{\left( s + \frac{\ell}{2} \right)^2}
			+ \sum_{n=0}^\infty \left[ \frac{1}{(n+s)^2} - \frac{1}{\left( n + \frac{N-2s}{2} \right)^2}
				+ \frac{1}{2} \left\{ \frac{1}{\left( n + \frac{N-2s}{4} \right)^2} - \frac{1}{\left( n + \frac{N+2s}{4} \right)^2} \right\}
				\right]
			\\
			= \  & \frac{1}{s^2} - \frac{1}{\left( s + \frac{\ell}{2} \right)^2}
			\\
			     & \quad 
			+ \sum_{n=0}^\infty
			\left[
				\frac{1}{(n+1+s)^2} - \frac{1}{\left( n + \frac{N-2s}{2} \right)^2}
				+ \frac{1}{2} \left\{ \frac{1}{\left( n + \frac{N-2s}{4} \right)^2} - \frac{1}{\left( n + \frac{N+2s}{4} \right)^2} \right\}
				\right].
		\end{aligned}
	\]
	Since $\ell \ge 0$ and $N \ge 6$, we see that 
	\[
		\frac{1}{s^2} - \frac{1}{\left( s + \frac{\ell}{2} \right)^2} \geq 0, \quad 
		1 + s < \frac{N-2s}{2}.
	\]
	Thus, $\frac{\partial^2 H_{\ell}}{\partial s^2} (s,N) > 0$ holds for each $s \in (0,1)$ 
	and $H_\ell (\cdot, N)$ is strictly convex in $(0,1)$.

	(iv) Assume $N=7$ and $\ell = 0$. Then by \eqref{eq:5.13} and \eqref{eq:5.22}, 
	\[
		\begin{aligned}
			\lim_{s \searrow 0} \frac{\partial H_{0}}{\partial s} (s,7)
			 & = \sum_{n=0}^\infty \left[ - \frac{1}{n+1} - \frac{1}{n + \frac{7}{2}} + \frac{2}{n + \frac{7}{4}} \right]
			= \psi (1) + \psi \left( \frac{7}{2} \right) - 2 \psi \left( \frac{7}{4} \right).
		\end{aligned}
	\]
	It follows from \cite[(1.3.3)--(1.3.5)]{L-72} that 
	\[
		\begin{aligned}
			\psi \left( \frac{7}{2} \right)
			= \psi \left( \frac{5}{2} \right) + \frac{2}{5}
			= \psi \left( \frac{3}{2} \right) + \frac{2}{3} + \frac{2}{5}
			= \psi \left( \frac{1}{2} \right) + 2 + \frac{2}{3} + \frac{2}{5}
		\end{aligned}
	\]
	and 
	\[
		-2 \psi \left( \frac{7}{4} \right) = - 2 \left[ \psi \left( \frac{3}{4} \right) + \frac{4}{3} \right]
		= - \frac{8}{3} - 2 \psi \left( \frac{1}{2} \right) - \pi + 2 \log 2.
	\]
	From $\psi(1) = - \gamma$ and $\psi(1/2) = - 2 \log 2 - \gamma$, we infer that 
	\[
		\lim_{s \searrow 0}\frac{\partial H_{0}}{\partial s} (s,7)
		= \psi(1) + \psi \left( \frac{1}{2} \right) + \frac{2}{5} - 2 \psi \left( \frac{1}{2} \right) - \pi + 2 \log 2 
		= \frac{2}{5} + 4 \log 2 - \pi > 0.
	\]
	Hence, the first assertion holds.

	To prove the second assertion, we first suppose $N = 8$ and $\ell = 0$. 
	Then by \eqref{eq:5.22}, we have 
	\[
		\begin{aligned}
			\lim_{s \searrow 0}\frac{\partial H_{0}}{\partial s} (s,8) = \sum_{n=0}^\infty 
			\left[ - \frac{1}{n+1} - \frac{1}{n+4} + \frac{2}{n + 2} \right]
			 & = 
			\left[ -1 - \frac{1}{2} - \frac{1}{3} + 2 \left( \frac{1}{2} + \frac{1}{3} \right)
				\right]
			=  - \frac{1}{6} < 0.
		\end{aligned}
	\]
	Hence, there exists an $\wt{s}_0 \in (0,1)$ such that 
	\begin{equation}\label{eq:5.23}
		\frac{\partial H_{0}}{\partial s} (s,8) < 0 \quad 
		\text{for all $s \in (0, \wt{s}_0 )$}. 
	\end{equation}
	On the other hand, for $H_{0}(s,N)$, we regard $N$ as a real variable in $[8,\infty)$ and 
	differentiate $\frac{\partial H_{0}}{\partial s}$ with respect to $N$ to obtain
	\[
		\begin{aligned}
			\frac{\partial^2 H_{0}}{\partial N \partial s} (s,N)
			 & = \frac{1}{2} \psi' \left( \frac{N-2s}{2} \right) - \frac{1}{4} \psi' \left( \frac{N+2s}{4} \right)
			- \frac{1}{4} \psi' \left( \frac{N-2s}{4} \right)
			\\
			 & = \frac{1}{4}
			\sum_{n=0}^\infty 
			\left[
				\frac{2}{\left( n + \frac{N-2s}{2} \right)^2} - \frac{1}{\left( n + \frac{N+2s}{4} \right)^2}
				- \frac{1}{\left( n + \frac{N-2s}{4} \right)^2}
				\right].
		\end{aligned}
	\]
	Since 
	\[
		\frac{N-2s}{2} > \frac{N+2s}{4} > \frac{N-2s}{4} \quad 
		\text{for every $(s,N) \in (0,1) \times [8,\infty)$}, 
	\]
	we get 
	\[
		\frac{\partial^2 H_{0}}{\partial N \partial s} (s,N) < 0
		\quad \text{for each $(s,N) \in (0,1) \times [8,\infty)$}. 
	\]
	Combining this with \eqref{eq:5.23}, we obtain 
	\[
		\frac{\partial H_{0}}{\partial s} (s, N) < 0 \quad \text{for any $(s,N) \in (0,\wt{s}_0) \times [8,\infty)$}. 
	\]

	Set 
	\[
		t_N := \sup \Set{ t \in (0,1) | \frac{\partial H_{0}}{\partial s} (t,N) < 0 }.
	\]
	Recalling (iii), we have $\frac{\partial H_0}{\partial s} (s,N) < 0$ for every $s \in (0,t_N)$ and 
	$\frac{\partial H_0}{\partial s} (s,N) > 0$ if $t_N<1$ and $t_N < s < 1$.

	(v) Suppose $-2 < \ell < \infty$. We first remark that when $N =2$, by \eqref{eq:5.21}, 
	\[
		\begin{aligned}
			\lim_{s \nearrow 1 } H_{\ell} (s,2) 
			 & = \log \left( 1 + \frac{\ell}{2} \right) + 
			\lim_{s \nearrow 1 } \left[  - \log \Gamma \left( 1 - s \right)
				+ 2 \log \Gamma \left( \frac{1}{2} - \frac{s}{2} \right) \right]
			\\
			 & = \log \left( 1 + \frac{\ell}{2} \right) + 
			\lim_{s \nearrow 1} \left[ \log (1-s) - \log \frac{(1-s)^2}{4} + O(1) \right]
			= \infty.
		\end{aligned}
	\]
	Hence, we consider the case $N \ge 3$. We notice that 
	\[
		\log \Gamma \left( \frac{N}{2} \right) - \log \Gamma \left( \frac{N}{2} - 1 \right)
		= \log \frac{\Gamma \left( N/2 \right)}{\Gamma \left( N/2 - 1 \right) }
		= \log \left( \frac{N}{2} - 1 \right)
	\]
	and that 
	\[
		\begin{aligned}
			2 \log \Gamma \left( \frac{N}{4} - \frac{1}{2} \right) - 2 \log \Gamma \left( \frac{N}{4} + \frac{1}{2} \right)
			= 2 \log \frac{\Gamma (N/4 - 1/2)}{\Gamma (N/4 + 1/2)}
			 & = 2 \log \frac{1}{ \frac{N}{4} - \frac{1}{2} }        \\
			 & = - 2 \log \left( \frac{N}{4} - \frac{1}{2} \right). 
		\end{aligned}
	\]
	Therefore, we obtain 
	\[
		\begin{aligned}
			\lim_{s \nearrow 1 } H_{\ell} (s,N) =H_{\ell} (1,N) 
			 & = \log \left[ \left( 1 + \frac{\ell}{2} \right) \left( \frac{N}{2} - 1 \right)
				\left( \frac{N}{4} - \frac{1}{2} \right)^{-2} \right]
			\\
			 & = \log \left[ \left( 1 + \frac{\ell}{2} \right) 4 \left( \frac{N}{2} - 1 \right)^{-1} \right].
		\end{aligned}
	\]
	Since 
	\[
		\begin{aligned}
			4 \left( 1 + \frac{\ell}{2} \right) \left( \frac{N}{2} - 1 \right)^{-1}
			\geq 1 
			\quad \Leftrightarrow \quad 
			8 + 4 \ell  \geq N - 2  \quad 
			\Leftrightarrow \quad 10 + 4 \ell \geq N,
		\end{aligned}
	\]
	we see that (v) holds. 
\end{proof}

\begin{lemma}\label{Lemma:5.5}
	Let $(\ell,s) \in (-2,\infty) \times (0,1)$ with $-2s < \ell$. 
	Then there exists an $N_\ast = N_\ast (s,\ell) \geq 2$ such that 
	$H_{\ell}(s,N) < 0$ if $N \geq N_\ast$. 
\end{lemma}

\begin{proof}
	As in the proof of Lemma \ref{Lemma:5.4} (iv), we regard $N$ as a real variable. 
	Let $(\ell,s) \in (-2,\infty) \times (0,1)$ with $\ell > -2s$. 
	For $N \geq 2$, it follows from  \eqref{eq:5.13} that 
	\[
		\begin{aligned}
			\frac{\partial H_{\ell}}{\partial N} (s,N)
			 & = \frac{1}{2} \psi \left( \frac{N}{2} \right) - \frac{1}{2} \psi \left( \frac{N-2s}{2} \right)
			- \frac{1}{2} \psi \left( \frac{N+2s}{4} \right) + \frac{1}{2} \psi \left( \frac{N-2s}{4} \right)
			\\
			 & = \frac{1}{2} \sum_{n=0}^\infty 
			\left( - \frac{1}{n + \frac{N}{2}} + \frac{1}{n + \frac{N-2s}{2}} + \frac{1}{n+\frac{N+2s}{4}}
			- \frac{1}{n + \frac{N-2s}{4}} \right)
			\\
			 & = 
			\frac{1}{2} \sum_{n=0}^\infty 
			\left[\frac{s}{ \left( n + \frac{N}{2} \right) \left( n + \frac{N-2s}{2} \right) }
				- \frac{s}{ \left( n + \frac{N+2s}{4} \right) \left( n + \frac{N-2s}{4} \right) }
				\right]
			\\
			 & = \frac{s}{32} \sum_{n=0}^\infty 
			\frac{ - (3N^2 - 8Ns + 4s^2) - (8N - 16 s) n }
			{ \left( n + \frac{N}{2} \right) \left( n + \frac{N-2s}{2} \right) \left( n + \frac{N+2s}{4} \right) \left( n + \frac{N-2s}{4} \right)  }.
		\end{aligned}
	\]
	For $N \in [2,\infty)$, we have $3N^2 - 8Ns + 4 s^2 > 0$ and $ 8N - 16s > 0$, which yields 
	\begin{equation}\label{eq:5.24}
		\frac{\partial H_{\ell}}{\partial N} (s,N) < 0 \quad 
		\text{for each $\ell \in (-2,\infty)$, $s \in \left( \max \left\{ 0, -\frac{\ell}{2} \right\} , \  1 \right)$ and $N \in [2,\infty)$}. 
	\end{equation}

	By \eqref{eq:5.24}, to prove Lemma \ref{Lemma:5.5}, 
	it is enough to observe the asymptotic behavior of $H_{\ell}(s,N)$ as $N \to \infty$. 
	To this end, we shall use the following: (see \cite[p.10--p.11]{L-72})
	\[
		\log \Gamma (z) = z \log z - z - \frac{1}{2} \log z + O(1)\quad 
		\text{as $z \to \infty$}.
	\]
	Since $\log (1+t) = t +O(t^2)$ as $ t \searrow 0$, we have 
	\[
		\begin{aligned}
			\log \Gamma \left( \frac{N}{2} \right) - \log \Gamma \left( \frac{N-2s}{2} \right)
			 & = \frac{N}{2} \log \frac{N}{2} - \frac{N-2s}{2} \log \frac{N-2s}{2} + O(1)
			\\
			 & = \frac{N}{2} \log \frac{N}{N-2s} + s \log \frac{N-2s}{2} +O(1) 
			\\
			 & = s\log (N-2s) + O(1)
		\end{aligned}
	\]
	and 
	\[
		\begin{aligned}
			     & 2 \log \Gamma \left( \frac{N-2s}{4} \right) - 2 \log \Gamma \left( \frac{N+2s}{4} \right)
			\\
			= \  & 2 \left[ \frac{N-2s}{4} \log \frac{N-2s}{4} - \frac{N+2s}{4} \log \frac{N+2s}{4} \right] + O(1) 
			\\
			= \  & -2 s \log (N+2s) + O(1),
		\end{aligned}
	\]
	which gives 
	\[
		H_{\ell}(s,N) =s \log (N-2s)  - 2 s \log (N+2s) + O(1) \to - \infty \quad \text{as $N \to \infty$}.
	\]
	From this fact and \eqref{eq:5.24}, we observe that Lemma \ref{Lemma:5.5} holds. 
\end{proof}

\begin{remark}
	By the change of variables and Lemma \ref{Lemma:5.5}, 
	for any given $(\ell,s) \in (-2,\infty) \times (0,1)$ with $\ell > -2s$, 
	if $p$ and $N$ are sufficiently large, then $p$ is JL-supercritical. 
\end{remark}


\subsection{Proof of Theorem \ref{Theorem:1.2}, Corollary \ref{Corollary:1.1} and Theorem \ref{Theorem:1.3}}\label{section:5.3}


We prove Theorem \ref{Theorem:1.2}.

\begin{proof}[Proof of Theorem {\rm\ref{Theorem:1.2}}]
	(i) Let $0<s< \min \{1, N/2\}$, $-2 s < \ell \leq 0$. 
	Remark that the assertion in (i) is equivalent to the following claim: 
	either \eqref{align:5.8} holds for all $x \in [0,A_{N,s})$ or else 
	there exists an $x_{JL} \in (0,A_{N,s})$ such that 
	$h_\ell (x_{JL}) = 2 \left[ \log \Gamma (A_{N,s} + s ) - \log \Gamma (A_{N,s}) \right]$ and 
	\eqref{align:5.8} (resp. \eqref{align:5.9}) holds in $(0,x_{JL})$ (resp. in $(x_{JL}, A_{N,s})$). 
	By Lemmata \ref{Lemma:5.1} and \ref{Lemma:5.2}, it is easily seen that this claim follows and 
	(i) holds.

	(ii) Assume $\ell = 0$. When $N=1$ and $0<s<1/2$,
	Lemma \ref{Lemma:5.3}
	gives 
		\[
			h_{0}(A_{N,s}) > 2 \left[ \log \Gamma (A_{N,s} + s ) - \log \Gamma (A_{N,s}) \right],
		\]
	hence, Lemmata \ref{Lemma:5.1} and \ref{Lemma:5.2} yield
	\[ 
		h_{0}(x) > 2 \left[ \log \Gamma (A_{N,s} + s ) - \log \Gamma (A_{N,s}) \right] 
		\quad \text{for all $x \in [0,A_{N,s}]$}.
	\] 
	Therefore, when $\ell = 0$, $N=1$ and $s \in (0,1/2)$, 
	any $p \in (1,\infty)$ is JL-subcritical.

	When $2 \leq N \leq 7$, from Lemma \ref{Lemma:5.4} (i)--(iv), 
	it follows that 
	\[
		\frac{\partial H_{0}}{\partial s} (s,N) > 0 \quad \text{for each $s \in (0,1)$}, \quad 
		\lim_{s \searrow 0} H_{0}(s, N) = 0. 
	\]
	Thus, $H_{0}(s,N) > 0$ holds for every $2 \leq N \leq 7$ and $s \in (0,1)$. 
	Since $H_{0}(s,N) > 0$ is equivalent to 
	$h_{0}(A_{N,s}) > 2 \left[ \log \Gamma (A_{N,s} + s ) - \log \Gamma (A_{N,s}) \right]$, 
	the assertion for $2 \leq N \leq 7$ holds.

	Suppose $N=8,9$. 
	In this case, by Lemma \ref{Lemma:5.4} (i), (iii), (iv) and (v), there exists a unique $s_N \in (0,1)$ such that 
	$H_{0}(s,N) < 0 = H_{0}(s_N,N) < H_{0}(t,N)$ hold for $0<s<s_N<t<1$.
	By Lemma \ref{Lemma:5.2} and the fact
	\[
		\begin{aligned}
			&H_{0}(s,N) = h_0(A_{N,s}) - 2 \left[ \log \Gamma (A_{N,s} + s ) - \log \Gamma (A_{N,s}) \right],
			\\
			&h_0(0) - 2 \left[  \log \Gamma (A_{N,s} + s ) - \log \Gamma \left( A_{N,s} \right) \right] > 0
		\end{aligned}
	\]
	when $0<s<s_N$, there exists a unique  $x_{JL} \in (0,A_{N,s})$ such that 
	\begin{equation}\label{eq:5.25}
		h_{0}(x_{JL}) = 2 \left[ \log \Gamma (A_{N,s} + s ) - \log \Gamma (A_{N,s}) \right].
	\end{equation}
	On the other hand, for every $s \in [s_N,1)$, we get  
	$h_{0}(x) > 2 \left[ \log \Gamma (A_{N,s} + s ) - \log \Gamma (A_{N,s}) \right]$
	in $[0,A_{N,s})$. Therefore, the assertion for $N=8,9$ holds.

	Finally, when $N \geq 10$, Lemma \ref{Lemma:5.4} (i), (iii), (iv) and (v) yield 
	$H_{0}(s,N) < 0$
	for all $s \in (0,1)$. Hence, by Lemma \ref{Lemma:5.2}, 
	we may find a unique $x_{JL} \in (0,A_{N,s})$ satisfying \eqref{eq:5.25}. 
	Thus, the assertion for $N \geq 10$ holds.

	(iii) According to Lemmata \ref{Lemma:5.1}--\ref{Lemma:5.3}, 
	we may prove the assertion for $N=1$ as in the case (ii). Therefore, we omit the details for $N=1$.
	
	 Suppose $2 \leq N < 10 + 4\ell $. From Lemma \ref{Lemma:5.4} (i) and (v), 
	 we find $s_{N,\ell}$, $\tilde{s}_{N,\ell}\in ( - \ell/2 , 1 )$ with $s_{N,\ell}\le\tilde{s}_{N,\ell}$ 
	 such that $H_{\ell}(s,N)<0$ for $s\in (- \ell/2, s_{N,\ell})$ and $H_{\ell}(s,N)>0$ for $s\in (\tilde{s}_{N,\ell},1)$. 
	 In particular, when $ 2 \leq N \leq 6$, Lemma \ref{Lemma:5.4} (ii) yields $s_{N,\ell}=\tilde{s}_{N,\ell}$. 
	 Then, by the same argument as in the case $N=8, 9$ of (ii), the assertion for $2 \leq N < 10 + 4\ell $ holds.
	
	When $N \geq 10+ 4\ell$, applying Lemma \ref{Lemma:5.4} (ii) and (v), 
	we derive $H_{\ell}(s,N)<0$ for $s\in (-\ell/2,1)$ and $3 \leq N \leq 6$. 
	Then, it follows from \eqref{eq:5.24} that $H_{\ell}(s,N)<0$ for $s\in (-\ell/2,1)$ and $N \geq 3$. 
	Hence, by the same method as in the case $N \geq 10$ of (ii), the assertion for $N \geq 10+ 4\ell$ holds.
\end{proof}

	Next we show Corollary \ref{Corollary:1.1}:

	\begin{proof}[Proof of Corollary {\rm\ref{Corollary:1.1}}]
We only treat $p > p_{S}(N,\ell)$.

(i) In this case, by Theorem \ref{Theorem:1.2} (ii) and \eqref{eq:5.3}, we know that $g_{0}(x) > \wt{M}_{N,s}$ 
for all $x \in [0,A_{N,s})$. From $\ell > 0$ and \eqref{eq:5.6}, we observe that 
$g_{\ell} (x) > g_{0} (x) > \wt{M}_{N,s}$ for any $x \in [0,A_{N,s})$. 
Hence, each $p > p_{S}(N,\ell)$ is JL-subcritical.

(ii) Since $-2s < \ell < 0$, \eqref{eq:5.6} gives $g_{\ell} (x) < g_{0} (x) $ for each $x \in [0,A_{N,s})$. 
By $s \in (0,s_N)$ and Theorem \ref{Theorem:1.2}, there exists a unique $x_0 \in (0,A_{N,s})$ such that 
$g_{0} (x_0) = \wt{M}_{N,s} $. Thus, $g_{\ell} (x_0) < \wt{M}_{N,s}$ and Theorem \ref{Theorem:1.2} 
yield $p_{JL} < \infty$. This fact implies that we may choose $s_{N,\ell}$ in Theorem \ref{Theorem:1.2} (iii) 
to satisfy $s_{N,\ell} \geq s_N$.
	\end{proof}

	Finally, we prove Theorem \ref{Theorem:1.3}:

	\begin{proof}[Proof of Theorem {\rm\ref{Theorem:1.3}}]
We remark that due to $N \geq 8$ and $0<s<1$, $A_{N,s} >1/2$ holds. 
Moreover, $g_0(A_{N,s}) < \wt{M}_{N,s}$ holds provided $N=8,9$ with $0<s<s_N$ or $10 \leq N$ with $0<s<1$. 
We shall find an $\wt{s}_N \in (0,1)$ with $\wt{s}_N \leq s_N$ when $N=8,9$ such that if $s \in (0,\wt{s}_N)$, then 
there exists an $x_1 = x_1(N,s) \in (1/2, A_{N,s})$ such that 
	\[
		g_0(x_1) < \wt{M}_{N,s}, \quad 
		\ell_{1,s} := \frac{2s}{ M_{0,s} }  \left( \wt{M}_{N,s} - M_{0,s} \right), \quad 
		\ell_{2,s} :=  \frac{2}{ M_{1,s} } \left( A_{N,s} - x_1 + s \right) \left( \wt{M}_{N,s} - M_{1,s} \right)
	\]
satisfy Theorem \ref{Theorem:1.3} (i) and (ii) with $\ell_{1,s} < \ell_{2,s}$ where 
	\[
		M_{0,s} := g_0(A_{N,s}), \quad M_{1,s} := g_0(x_1).
	\]

	We remark that \eqref{eq:5.6} gives
	\begin{equation}\label{eq:5.26}
		\begin{aligned}
			\wt{M}_{N,s} \geq g_\ell (x_1) = \frac{A_{N,s} - x_1 + s + \frac{\ell}{2} }{A_{N,s} - x_1 + s } M_{1,s} 
			\quad 
			&\Leftrightarrow \quad \frac{\ell}{2} M_{1,s} \leq \left( A_{N,s} - x_1 + s \right) \left( \wt{M}_{N,s} - M_{1,s} \right)
			\\
			&\Leftrightarrow \quad \ell \leq \ell_{2,s}
		\end{aligned}
	\end{equation}
and 
	\begin{equation}\label{eq:5.27}
		\wt{M}_{N,s} \geq g_{\ell} (A_{N,s}) = \frac{s+\frac{\ell}{2}}{s} M_{0,s} 
		\quad \Leftrightarrow \quad 
		\frac{\ell}{2} M_{0,s} \leq s \left( \wt{M}_{N,s} - M_{0,s} \right) 
		\quad \Leftrightarrow \quad 
		\ell \leq \ell_{1,s}.
	\end{equation}
Assuming $\ell_{1,s} < \ell_{2,s}$ for a while, we check Theorem \ref{Theorem:1.3} (i) and (ii). 
Indeed, if $s \in (0,\wt{s}_N)$ and $\ell \in (0, \ell_{1,s} )$, \eqref{eq:5.27} and Lemma \ref{Lemma:5.1} 
yield $g_\ell (A_{N,s}) < \wt{M}_{N,s} $ and $g_\ell(x) > \wt{M}_{N,s}$ for sufficiently small $x >0$. 
Therefore, Theorem \ref{Theorem:1.3} (i) holds by virtue of \eqref{eq:5.3}. 
On the other hand, if $s \in (0,\wt{s}_N)$ and $\ell \in (\ell_{1,s} , \ell_{2,s})$, then 
\eqref{eq:5.26} and \eqref{eq:5.27} give $g_\ell (x_1) < \wt{M}_{N,s} < g_\ell (A_{N,s})$. 
By Lemma \ref{Lemma:5.1} and \eqref{eq:5.3}, Theorem \ref{Theorem:1.3} (ii) holds.

	Now let us show the existence of $\wt{s}_N$ and $\ell_{1,s} < \ell_{2,s}$ for all $s \in (0,\wt{s}_N)$. 
Set 
	\[
		L(x) := \frac{2}{g_0 (x)} \left( A_{N,s} - x + s  \right) \left( \wt{M}_{N,s} - g_{0} (x) \right).
	\]
Then $L(A_{N,s}) = \ell_{1,s}$ and $L(x_1) = \ell_{2,s}$. 
We shall prove 
$L'(A_{N,s}) < 0$ for each $s \in (0,\wt{s}_N) $. 
By 
	\[
		\begin{aligned}
			L'(A_{N,s}) 
			&= - \frac{2s}{M_{0,s}^2} \left( \wt{M}_{N,s} - M_{0,s} \right) 
			g_0'(A_{N,s}) - \frac{2}{M_{0,s}} \left( \wt{M}_{N,s} - M_{0,s} \right) 
			- \frac{2s}{M_{0,s}} g_0'(A_{N,s})
			\\
			&= \frac{2}{M_{0,s}^2} 
			\left[ - M_{0,s} \left( \wt{M}_{N,s} - M_{0,s} \right) - s \wt{M}_{N,s} g_0'(A_{N,s}) \right].
		\end{aligned}
	\]
From \eqref{eq:5.3} and \eqref{eq:5.5}, we have $\wt{M}_{N,s} \to 1$, $M_{0,s} = g_0(A_{N,s}) \to 1$ as $s \searrow 0$. 
Therefore, 
	\begin{equation}\label{eq:5.28}
		\lim_{s \searrow 0} \frac{L'(A_{N,s})}{s} 
		= -2 \left[  \frac{\partial \wt{M}_{N,s}}{\partial s} \Big|_{s=0} - \frac{\partial M_{0,s}}{\partial s} \Big|_{s=0} 
		+ \lim_{s \searrow 0} g_0'(A_{N,s}) \right].
	\end{equation}

	It follows from \eqref{eq:5.3} that 
	\[
		\frac{\partial \wt{M}_{N,s}}{\partial s} \Big|_{s=0}
		= \wt{M}_{N,s} \frac{\partial }{\partial s} \log \wt{M}_{N,s} \Big|_{s=0} 
		= 2 \psi \left( \frac{N}{4} \right). 
	\]
On the other hand, by \eqref{eq:5.14}, $h_0'(A_{N,s}) \to 0$ as $s \searrow 0$. 
Since $g_0'(A_{N,s}) = g_0(A_{N,s}) h_0'(A_{N,s})$, we see that 
	\[
		\lim_{s \searrow 0} g_0'(A_{N,s}) = 0.
	\]
Moreover, 
	\[
		\frac{\partial M_{0,s}}{\partial s} \Big|_{s=0} = 
		\frac{\partial}{\partial s}  \left( \frac{ \Gamma (1+s) \Gamma (2A_{N,s} + s ) }{\Gamma ( 2 A_{N,s} )} \right) 
		 \Big|_{s=0} = \psi(1) + \psi \left( \frac{N}{2} \right).
	\]
Therefore, \eqref{eq:5.28} yields 
	\[
		\lim_{s \searrow 0} \frac{L'(A_{N,s})}{s} 
		= -2 \left[ 2 \psi \left( \frac{N}{4} \right) - \psi (1) - \psi \left( \frac{N}{2} \right) \right].
	\]
Since \eqref{eq:5.12} gives  
	\[
		\frac{d}{dN}\left[ 2 \psi \left( \frac{N}{4} \right) - \psi (1) - \psi \left( \frac{N}{2} \right) \right] 
		= \frac{1}{2} \left[ \psi' \left( \frac{N}{4} \right) - \psi' \left( \frac{N}{2} \right) \right] > 0,
	\]
by $N \geq 8$, we observe that 
	\[
		\lim_{s \searrow 0} \frac{L'(A_{N,s})}{s} 
		\leq -2 \left[ 2 \psi (2) - \psi (1) - \psi (4) \right] = - \frac{1}{3} < 0.
	\]
Hence, there exists an $\wt{s}_N>0$ such that $L'(A_{N,s}) < 0$ for any $s \in (0,\wt{s}_N)$ and 
$\ell_{1,s} < \ell_{2,s}$ follows if we choose $x_1$ close to $A_{N,s}$.

	Finally, we prove the existence of $(\ell_{3,s})$. 
By the assumption $N=8,9$ with $0<s<s_N$ or $N \geq 10$, \eqref{eq:5.3}, Lemma \ref{Lemma:5.2} (ii) and 
Theorem \ref{Theorem:1.2} (ii), there exists a unique $\wt{x}_{N,s} \in (0,A_{N,s})$ such that 
$g_0(\wt{x}_{N,s}  ) = \wt{M}_{N,s}$. 
For any $\ell>0$, \eqref{eq:5.6} and Lemma \ref{Lemma:5.2} (ii) yield 
	\[
		\begin{aligned}
			\min_{\wt{x}_{N,s} \leq x \leq A_{N,s} } g_{\ell} (x) 
			&> 
			\left[ \min_{ \wt{x}_{N,s} \leq x \leq A_{N,s}  } \frac{A_{N,s} - x + s + \frac{\ell}{2} }{A_{N,s} - x + s} \right] 
			\left[ \min_{ \wt{x}_{N,s} \leq x \leq A_{N,s} }  g_0(x) \right]
			\\
			&= \frac{A_{N,s} - \wt{x}_{N,s} + s + \frac{\ell}{2} }{A_{N,s} - \wt{x}_{N,s} + s } M_{0,s}.
		\end{aligned}
	\]
We set 
	\[
		\ell_{3,s} := \frac{2}{M_{0,s}} \left( A_{N,s} - \wt{x}_{N,s} + s \right) \left( \wt{M}_{N,s} - M_{0,s} \right).
	\]
Then it is easily seen that 
	\[
		\begin{aligned}
			\frac{A_{N,s} - \wt{x}_{N,s} + s + \frac{\ell}{2} }{A_{N,s} - \wt{x}_{N,s} + s } M_{0,s} 
			\geq \wt{M}_{N,s}
			\quad &\Leftrightarrow \quad 
			\frac{\ell}{2} M_{0,s} \geq \left( A_{N,s} - \wt{x}_{N,s} + s \right) \left( \wt{M}_{N,s} - M_{0,s} \right)
			\\
			&\Leftrightarrow \quad 
			\ell \geq \ell_{3,s}.
		\end{aligned}
	\]
Since $g_{\ell} (x) > g_0(x)$ for any $x \in [0, \wt{x}_{N,s} ]$ and $\ell>0$, if $\ell \geq \ell_{3,s}$, then 
	\[
		\min_{0 \leq x \leq A_{N,s}}  g_{\ell} (x) > \wt{M}_{N,s}.
	\]
This completes the proof. 
	\end{proof}

\section*{Acknowledgments}
The first author (S.H.) was supported by JSPS KAKENHI Grant Numbers JP 
20J01191.
The second author (N.I.) was supported by JSPS KAKENHI Grant Numbers JP 
17H02851, 19H01797 and 19K03590. 
The third author (T.K.) was supported by JSPS KAKENHI Grant Numbers JP 18H01126, 
19H05599 and 20K03689.

\appendix

\section{Appendix A}\label{section:A}

In this Appendix, we provide the details of the proof that $\cO_+$ is bounded in the proof of Lemma \ref{Lemma:4.4}.
We recall that there exists a $\gamma_0 \in C( [0,1] , \ov{\HS} )$ such that 
\[
	\gamma_0 (0) = (x_{n_0}, 0), \quad 
	\gamma_0(1) = (x_{n_0+1}, 0), \quad 
	\gamma_0(\tau) \in \cO_{k,-} \quad \text{for each $\tau \in (0,1)$},
\]
where $V(X) = U_2(X) - U_1(X)$, 
$V(x_{n_0},0), V(x_{n_0+1},0) < 0 < V(y_0,0)$ and $|x_{n_0}| < |y_0| < |x_{n_0+1}|$. 
In what follows, by $V(x,t) = V(|x|,t)$, 
set $\gamma_1(\tau) := ( | \gamma_{0,x}(\tau) | , \gamma_{0,t} (\tau) )$
where $\gamma_0(\tau) = (\gamma_{0,x} (\tau) , \gamma_{0,t} (\tau) ) \in \RN \times [0,\infty)$. 
Then, 
\[
	\begin{aligned}
		&\gamma_1 \in C \left( [0,1] , [0,\infty)^2 \right), \quad 
		\gamma_1 (0) = (|x_{n_0}|, 0), \quad 
		\gamma_1 (1) = ( |x_{n_0+1}|, 0 ), 
		\\
		& 
		V( \gamma_1(\tau) ) < 0 \quad \text{for each $\tau \in [0,1]$}.
	\end{aligned}
\]

To proceed, we choose an $R_0>0$ so that 
\begin{equation}\label{eq:A.1}
	\max_{\tau \in [0,1]} \left\{ \gamma_{1,|x|} (\tau) + \gamma_{1,t} (\tau) \right\} < R_0,
\end{equation}
where $\gamma_1 (\tau) = ( \gamma_{1,|x|}(\tau) , \gamma_{1,t} (\tau) )$. 
Now we prove the following result:

\begin{lemma}\label{Lemma:A.1}
	If $\gamma_2 = (\gamma_{2,|x|}, \gamma_{2,t})  \in C \left( [0,1] , [0,\infty)^2 \right)$
	satisfies 
	\begin{equation}\label{eq:A.2}
		\gamma_2(0) = (|y_0|,0), \quad 
		\gamma_{2,t} (\tau) > 0 \quad \text{for each $0 < \tau \leq 1$}, \quad 
		R_0 < \gamma_{2,|x|} (1) + \gamma_{2,t} (1),
	\end{equation}
	then there exist $0 < \tau_1, \tau_2 < 1$ such that 
	$\gamma_1(\tau_1) = \gamma_{2} (\tau_2)$. 
\end{lemma}

As a consequence of Lemma \ref{Lemma:A.1}, 
the component $\cO_+\subset \HS$ of $[V<0]$ satisfying $(y_0,0) \in \ov{\cO_{+}}$ becomes bounded. 
In fact, if $\cO_+$ is unbounded, then from the fact that $\cO_{+}$ is open and connected, 
we may find $\gamma_3 \in C([0,1], \ov{\HS})$ such that 
\[
	\gamma_3(0) = (y_0,0), \quad 
	\gamma_{3,\tau} (\tau) \in \cO_+ \quad \text{for each $0 < \tau \leq 1$}, \quad 
	2 R_0 < |\gamma_3 (1) |. \quad 
\]
However, by setting $\gamma_2(\tau) := ( |\gamma_{3,x}(\tau) |, \gamma_{3,t} (\tau) )$, 
$\gamma_2$ satisfies \eqref{eq:A.2} and 
there exist $0<\tau_1 , \tau_2 < 1$ so that $\gamma_1 (\tau_1) = \gamma_2(\tau_2) \in \cO_{+} \cap \cO_{k,-}$. 
However, this is a contradiction since $V<0$ on $\cO_{k,-}$ and $V > 0$ on $\cO_{+}$. 
Therefore, $\cO_{+}$ must be bounded.

What remains to prove is Lemma \ref{Lemma:A.1}:

\begin{proof}[Proof of Lemma {\rm\ref{Lemma:A.1}}]
	From \eqref{eq:A.1}, we observe that $\gamma_1$ does not intersect with $\gamma_2$ in the region 
	\[
		\Omega := \Set{ (z,t) \in [0,\infty)^2 | z+t > R_0 }.
	\]
	Since $\Omega$ is path-connected, without loss of generality, 
	we may suppose that $\gamma_2(1) = (R_0,R_0)$.

	Next, consider a function $\eta: [0,1]^2 \to \R^2$ defined by 
	\[
		\eta(\tau_1 , \tau_2) = 
		\begin{pmatrix}
			\eta_1 (\tau_1 , \tau_2) \\
			\eta_2 (\tau_1, \tau_2)
		\end{pmatrix}
		:= \gamma_2 (\tau_2) - \gamma_1 (\tau_1)
		= \begin{pmatrix}
			\gamma_{2,|x|} (\tau_2) - \gamma_{1,|x|} (\tau_1) \\
			\gamma_{2,t} (\tau_2) - \gamma_{1,t} (\tau_1)
		\end{pmatrix}.
	\]
	To prove Lemma \ref{Lemma:A.1}, it is enough to find $(\tau_1,\tau_2) \in (0,1)^2$ so that 
	$\eta(\tau_1,\tau_2) = (0,0)$. To this end, we shall show 
	\begin{equation}\label{eq:A.3}
		\deg \left( \eta, [0,1]^2, (0,0) \right) \neq 0.
	\end{equation}

	To see \eqref{eq:A.3}, we will exploit the homotopy invariance of degree 
	and observe $\eta$ on $\partial [0,1]^2$. 
When $(\tau_1 , \tau_2) = (0,\tau_2)$ with $0<\tau_2 \leq 1$, 
	recalling $\gamma_{2,t} (\tau) > 0$ for $ 0 < \tau \leq 1$
	and $\gamma_{1,t} (0) = 0$, we have 
	\begin{equation}\label{eq:A.4}
		\eta_2 (0,\tau_2) = \gamma_{2,t}(\tau_2) > 0 \quad \text{for every $0 < \tau_2 \leq 1$}.
	\end{equation}
	Similarly, when $(\tau_1, \tau_2) = (\tau_1,0)$ with $0 < \tau_1 < 1$, 
	it follows that 
	\begin{equation}\label{eq:A.5}
		\eta_2(\tau_1,0) = - \gamma_{1,t} (\tau_1) < 0 \quad \text{for every $0 < \tau_1 < 1$}.
	\end{equation}
	On the other hand, since $\gamma_{1,t}(1) = 0$ and $\gamma_{2,t} (1) =R_0$, 
	it follows that 
	\begin{equation}\label{eq:A.6}
		\begin{aligned}
			 & \eta_2(1,\tau_2) = \gamma_{2,t} (\tau_2) > 0       &          & \text{for every $0<\tau_2 \leq 1$},
			\\
			 & \eta_2(\tau_1,1) = R_0 - \gamma_{1,t} (\tau_1) > 0
			 &                                                    & \text{for every $0 \leq \tau_1 \leq 1$}. 
		\end{aligned}
	\end{equation}
	Next, by $\eta (0,0) = ( |y_0| - |x_{n_0}| , 0 )$, $\eta(1,0) = ( |y_0| - |x_{n_0+1}| , 0 )$
	and the continuity of $\gamma_i$, we may find $ 0 < \e_0 < 1/4$ satisfying 
	\begin{equation}\label{eq:A.7}
		\begin{aligned}
			 & 0 \leq \tau_1, \tau_2 \leq \e_0       &                              & \Rightarrow 
			 &                                       & \eta_1(\tau_1,\tau_2) > 0, 
			\\
			 & 0 \leq 1 - \tau_1 , \tau_2 \leq \e_0 
			 &                                       & \Rightarrow 
			 &                                       & \eta_1 (\tau_1,\tau_2) < 0. 
		\end{aligned}
	\end{equation}
	With this choice of $\e_0$, we define $F$ by 
	\[
		F(\tau_1 , \tau_2) 
		= \begin{pmatrix}
			F_1 (\tau_1, \tau_2) \\ F_2(\tau_1,\tau_2)
		\end{pmatrix} := \nabla 
		\left[ - \left( \tau_1 - \frac{1}{2} \right)^2 + \left( \tau_2 - \frac{\e_0}{2} \right)^2 \right]
		= \begin{pmatrix}
			-2 \tau_1 + 1 \\ 2 \tau_2 - \e_0
		\end{pmatrix}.
	\]
	Then it follows that 
	\begin{equation}\label{eq:A.8}
		\begin{aligned}
			 & 0 \leq \tau_1 \leq \e_0     &  & \Rightarrow &  & F_1(\tau_1,\tau_2) = 1 - 2 \tau_1 > 0,
			\\
			 & 1 - \e_0 \leq \tau_1 \leq 1 &  & \Rightarrow &  & F_1(\tau_1,\tau_2) = 1 - 2 \tau_1 < 0,
			\\
			 & \e_0 \leq \tau_2 \leq 1     &  & \Rightarrow &  & F_2(0,\tau_2) = F_2(1,\tau_2) = 2 \tau_2 - \e_0 > 0,
			\\
			 & 0 \leq \tau_1 \leq 1        &  & \Rightarrow &  & F_2(\tau_1,0) = - \e_0 < 0,
			\\
			 & 0 \leq \tau_1 \leq 1        &  & \Rightarrow &  & F_2 (\tau_1,1) = 2 - \e_0 > 0.
		\end{aligned}
	\end{equation}

	Finally, we define 
	\[
		H_\theta (\tau_1, \tau_2) := \theta \eta(\tau_1,\tau_2) + (1-\theta) F(\tau_1,\tau_2).
	\]
	By \eqref{eq:A.4}--\eqref{eq:A.8}, it is easily seen that $H_\theta(\tau_1, \tau_2) \neq (0,0)$
	for all $(\tau_1,\tau_2) \in \partial [0,1]^2$ and $\theta \in [0,1]$. 
	Hence, the homotopy invariance of degree implies 
	\[
		\deg \left( \eta, [0,1]^2 , (0,0) \right) = 
		\deg \left( F , [0,1]^2 , (0,0) \right) = -1.
	\]
	Thus, \eqref{eq:A.3} holds and we complete the proof. 
\end{proof}



\section{Appendix B}\label{section:B}

Here we prove some technical lemmata. 

\begin{lemma}\label{Lemma:B.1}
	Let $\varphi \in W^{2,\infty} ((0,\infty))$ with $\varphi(0) = 0$
	and $u \in C^\infty_c(\RN)$ with $u \geq 0$ in $\RN$. 
	Assume that $\varphi '(t) > 0$ and $\varphi '$ is nonincreasing in $(0,\infty)$. Then 
	\begin{equation}\label{eq:B.1}
		\left( - \Delta \right)^s \left( \varphi (u) \right) (x)
		\geq \varphi' \left( u(x) \right) \left( - \Delta \right)^s u(x) \quad 
		\text{for each $x \in \RN$.}
	\end{equation}
\end{lemma}

\begin{proof}
	By $\varphi(u) \in W^{2,\infty} (\RN)$ and the fact 
	$\supp \varphi (u)$ is compact, we have 
	\[
		\left(  - \Delta  \right)^s \left( \varphi (u) \right) (x)
		= C_{N,s} \lim_{\e \to 0} \int_{|y-x| \geq \e}
		\frac{ \varphi(u(x)) - \varphi(u(y)) }{|x-y|^{N+2s}} \, dy.
	\]
	Since 
	\[
		\varphi (u(y)) - \varphi(u(x)) = \int_{0}^1 \varphi' \left( u(x) + \theta (u(y) - u(x) ) \right)  d \theta  
		\left(u(y) - u(x) \right),
	\]
	we observe that 
	\begin{equation}\label{eq:B.2}
		\begin{aligned}
			     & \left( - \Delta \right)^s \left( \varphi(u) \right) (x) - \varphi' \left( u(x) \right) \left( - \Delta \right)^s u(x) 
			\\
			= \  & 
			C_{N,s} \lim_{\e \to 0}
			\int_{|y-x| \geq \e}
			\frac{d y}{|x-y|^{N+2s}} 
			\\
			     & \qquad \times 
			\left[ \int_{0}^{1} \varphi' \left( u(x) + \theta (u(y) - u(x) ) \right) d \theta
				\left( u(x) - u(y) \right) - \varphi' \left( u(x) \right) \left( u(x) - u(y) \right)
				\right]
			\\
			= \  & 
			C_{N,s} \lim_{\e \to 0}
			\int_{|y-x| \geq \e} \frac{d y}{|x-y|^{N+2s}} 
			\int_{0}^{1} \left( u(x) - u(y) \right) 
			\\
			& \hspace{5.5cm} \times 
			\left\{ 
			\varphi' \left( u(x) + \theta \left( u(y) - u(x) \right) \right) - \varphi'(u(x)) 
			\right\} d \theta .
		\end{aligned}
	\end{equation}
	Since $\varphi'(t) $ is nonincreasing in $t$, we see that 
	for each $x,y \in \RN$ and $\theta \in (0,1) $, 
	\[
		\left( u(x) - u(y) \right) 
		\left\{ 
		\varphi' \left( u(x) + \theta \left( u(y) - u(x) \right) \right) - \varphi'(u(x)) 
		\right\}  \geq 0.
	\]
	Using this inequality and \eqref{eq:B.2}, we infer that \eqref{eq:B.1} holds. 
\end{proof}

\begin{lemma}\label{Lemma:B.2}
	\begin{enumerate}
	\item[{\rm (i)}]
		Let $u \in \Hsloc (\RN) \cap L^1(\RN , (1+|x|)^{-N-2s} d x  )$
		satisfy $u(x) \to 0$ as $|x| \to \infty$. 
		Then, $U(X) := (P_s (\cdot, t) \ast u) (x) \to 0 $ as $|X| \to \infty$. 
	\item[{\rm (ii)}] If 
	$u \in \Hsloc (\RN) \cap C (\RN \setminus \set{0} ) \cap L^1(\RN , (1+|x|)^{-N-2s} \rd x  )$ 
	satisfies $u(x) \to 0 $ as $|x| \to \infty$ and $u(x) \to \infty$ as $|x| \to 0$, then 
	$U(X) \to \infty$ as $|X| \to 0$. 
	\end{enumerate}

\end{lemma}

\begin{proof}
(i)
	From the assumption, for any given $\e>0$, there exists an $R_\e>0$ so that $|u(x)| \leq \e$ for $|x| \geq R_\e$. 
	Write 
	\[
		u(x) = \chi_{B_{R_\e}} (x) u(x) + (1- \chi_{B_{R_\e}} (x)  ) u(x) =: u_{\e,1}(x) + u_{\e,2}(x), 
		\quad U_{\e,i} (X) := \left( P_s(\cdot, t) \ast u_{\e,i} \right) (x).
	\]
	By $\| u_{\e,2} \|_{L^\infty(\RN)} \leq \e$, we see that 
	\[
		\sup_{ t > 0} \| U_{\e,2} (\cdot, t) \|_{L^\infty(\RN)} \leq \sup_{ t > 0} \| P_s (\cdot, t) \|_{L^1(\RN)} 
		\| u_{\e,2} \|_{L^\infty(\RN)} \leq \e. 
	\]

	On the other hand, we see from the compactness of $\supp u_{\e,1}$  that 
	$u_{\e,1} \in L^1(\RN)$ and 
	\begin{align}
			\left| U_{\e,1} (X) \right| 
			&\leq \int_{ |y| \leq R_\e } \frac{p_{N,s} t^{2s} }{(|x-y|^2 + t^2)^{ \frac{N+2s}{2} }} \left| u_{\e,1} (y) \right| dy
			\label{align:B.3}
			\\
			&\leq C \int_{ |y| \leq R_\e } t^{2s} t^{ - (N+2s) } \left| u_{\e,1} (y) \right| dy
			= C \| u_{\e,1} \|_{L^1(\RN)} t^{-N}.
			\nonumber
	\end{align}
	Hence, there exists a $t_\e>0$ such that 
	\[
		\sup_{ t \geq t_\e, x \in \RN} |U_{\e,1} (X) | \leq \e.  
	\]

	Finally, notice that $|x-y| \geq |x|/2$ for all $|x| \geq 2 R_\e$ and $|y| \leq R_\e$. 
Thus, in \eqref{align:B.3}, if $|x| \geq 2 R_\e$, then 
	\[
		\begin{aligned}
			\left| U_{\e,1} (X) \right| 
			& \leq 
			\int_{ |y| \leq R_\e } p_{N,s} t^{2s} 2^{N+2s} |x|^{-N-2s} 
			\left| u_{\e,1} (y) \right| dy
			= C_{N,s} t^{2s} \| u_{\e,1} \|_{L^1(\RN)} |x|^{-N-2s}.
		\end{aligned}
	\]
Therefore, we may find an $r_\e>0$ such that 
	\[
		\sup_{ 0 < t \leq t_\e, |x| \geq r_\e  } |U_{\e,1} (X)| \leq \e.
	\]
	Now we conclude that $|X| \geq t_\e + r_\e$ implies $|U(X)| \leq 2 \e$.

	(ii) For each $L>0$, set 
	\[
		\begin{aligned}
			u_L(x) 
			&:= 
			\min \{ u(x) , L \} \in \Hsloc (\RN) \cap C(\RN) \cap L^\infty(\RN) \cap L^1(\RN, (1+|x|)^{-N-2s} d x  ),
			\\
			U_L(X) 
			&:= 
			\left( P_s \left( \cdot , t \right) \ast u_L \right) (x) \in C( \ov{\HS} ).
		\end{aligned}
	\]
Since $u_L(x) \leq u(x)$ in $\RN$, we have $U_L(X) \leq U(X)$ for each $L>0$. 
From the assumption $u(x) \to \infty$ as $|x| \to 0$ and $U_L(x,0) = u_L(x)$, 
there exists an $r_L >0$ such that 
	\[
		|y| + t \leq r_L \quad \Rightarrow \quad \frac{L}{2} \leq U_L(y,t) \leq U(y,t).
	\]
Since $L>0$ is arbitrary, $U(X) \to \infty$ as $|X| \to 0$. 
\end{proof}

%
%

\end{document}